\documentclass[aos,preprint]{imsart}

\RequirePackage{natbib}
\usepackage{graphicx}
\usepackage{placeins}
\usepackage{enumitem}
\usepackage{titlesec}
\usepackage{cite}
\usepackage{comment}
\usepackage{caption}
\usepackage{subcaption}
\usepackage[normalem]{ulem}

\RequirePackage[colorlinks,citecolor=blue,urlcolor=blue]{hyperref}

\usepackage{amsmath}
\usepackage{amsthm}
\usepackage{amssymb}
\usepackage{amsfonts}
\usepackage{bbm}
\usepackage{bm}
\usepackage{nicefrac}
\usepackage{mathtools}

\usepackage{graphicx}
\usepackage{float}
\usepackage{wrapfig}
\usepackage{caption}

\RequirePackage{algorithm}
\RequirePackage{algpseudocode}

\makeatother

\newcommand{\bbeta}{\bm{\beta}}
\newcommand{\bgamma}{\bm{\gamma}} 

\newcommand{\dkl}{\mathrm{D_{KL} }}
\newcommand{\R}{\mathbb{R}}
\newcommand{\E}{\mathbb{E}}
\renewcommand{\P}{\mathbb{P}}
\newcommand{\Po}{\mathbb{P}_\mu}

\newcommand{\Q}{\mathcal{Q}}
\newcommand{\M}{\mathcal{M}}
\newcommand{\ML}{\mathrm{ML}}

\newcommand{\MF}{\mathrm{MF}}
\newcommand{\ELBO}{\mathrm{ELBO}}

\newcommand{\er}{\mathrm{Err}_{p,K}}

\theoremstyle{plain}\newtheorem{lemma}{\textbf{Lemma}}\newtheorem{theorem}{\textbf{Theorem}}\newtheorem{corollary}{\textbf{Corollary}}\newtheorem{assumption}{\textbf{Assumption}}\newtheorem{example}{\textbf{Example}}\newtheorem{definition}{\textbf{Definition}}

\theoremstyle{definition}

\theoremstyle{definition}\newtheorem{remark}{\textbf{Remark}}

\makeatletter
\renewenvironment{proof}[1][\proofname] {
	\par\pushQED{\qed}\normalfont
	\topsep6\p@\@plus6\p@\relax
	\trivlist\item[\hskip\labelsep\bfseries#1\@addpunct{:}]
 	\ignorespaces
} {
	\popQED\endtrivlist\@endpefalse
}
\makeatother

\begin{document}
\begin{frontmatter}

\title{A Mean Field Approach to Empirical Bayes Estimation in High-dimensional Linear Regression}
\runtitle{Empirical Bayes High-dimensional Linear Regression}

\begin{aug}
\author[A]{\fnms{Sumit}~\snm{Mukherjee\thanks{Supported by NSF DMS-2113414.}}\ead[label=e1]{sm3949@columbia.edu}}\and
\author[B]{\fnms{Bodhisattva}~\snm{Sen\thanks{Supported by NSF grants DMS-2311062 and DMS-2515520.}}\ead[label=e2]{b.sen@columbia.edu}} \and
\author[C]{\fnms{Subhabrata}~\snm{Sen\thanks{Supported by NSF DMS-CAREER 2239234, ONR N00014-23-1-2489 and AFOSR FA9950-23-1-0429.}}\ead[label=e3]
{subhabratasen@fas.harvard.edu}}
\address[A]{Department of Statistics, Columbia University\printead[presep={,\ }]{e1}}
\address[B]{Department of Statistics, Columbia University\printead[presep={,\ }]{e2}}
\address[C]{Department of Statistics, Harvard University\printead[presep={,\ }]{e3}}

\end{aug}



\begin{abstract} We study empirical Bayes estimation in high-dimensional linear regression. To facilitate  computationally efficient estimation of the underlying prior, we adopt a variational empirical Bayes approach, introduced originally in~\citet{carbonetto2012scalable} and~\citet{kim2022flexible}. We establish asymptotic consistency of the nonparametric maximum likelihood estimator (NPMLE) and its (computable) naive mean field variational surrogate under mild assumptions on the design and the prior. Assuming, in addition, that the naive mean field approximation has a dominant optimizer, we develop a computationally tractable approximation to the oracle posterior distribution, and establish its accuracy under the 1-Wasserstein metric. This enables computationally feasible Bayesian inference; e.g.,~construction of posterior credible intervals with an average coverage guarantee, Bayes optimal estimation for the regression coefficients, estimation of the proportion of non-nulls, etc. Our analysis covers both deterministic and random designs, and accommodates correlations among the features. To the best of our knowledge, this provides the first rigorous nonparametric empirical Bayes method in a high-dimensional regression setting without sparsity. 
\end{abstract}

\begin{keyword}[class=MSC]
\kwd[Primary ]{62C12}
\kwd{62G20}
\kwd[; secondary ]{62J05}
\end{keyword}

\begin{keyword}
\kwd{Consistency of nonparametric maximum likelihood, evidence lower bound, posterior inference, variational approximation}
\end{keyword}

\end{frontmatter}
	
	
	
	
	
	
	
	
	
	


\section{Introduction}\label{sec:Intro}
We consider the following high-dimensional linear regression model: 
\begin{align}\label{eq:lin_reg}
{\mathbf y}={\mathbf X}{\bm \beta}+{\bm \varepsilon}
\end{align} 
where ${\mathbf y}\in \R^n$ is the response vector, ${\mathbf X}\in \R^{n\times p}$ is the design matrix, ${\bm \beta} = (\beta_1,\ldots, \beta_p) \in \R^p$ is the vector of unknown regression coefficients, and ${\bm \varepsilon} \in \R^n$ is the (unobserved) error vector distributed as $N_n({\bf 0},\sigma^2{\bf I})$ with $\sigma^2 >0$ assumed known. In the Bayesian linear regression model, one further assumes a prior on ${\bm \beta}$; a simple starting assumption in this regard is that \begin{equation}\label{eq:iid-Beta}
\beta_1,\ldots,\beta_p\stackrel{iid}{\sim}\mu,
\end{equation}
for some prior distribution $\mu$ on $\R$. In the high-dimensional regime, which we study in this paper, we assume that the number of predictors $p$ changes with the sample size $n$, and both can be large. Throughout our analysis we work with a sequence of fixed design matrices; for notational convenience, we suppress the dependence of $\mathbf{X}$ on $n$ and $p$ throughout.

Although the traditional Bayesian approach to inference in high-dimensional linear regression is a powerful tool for tackling many scientific problems (see e.g.,~\citet{ george1997approaches},~\citet{MR2797133},~\citet{Guan-Stephens-2011}) for applications in genomics, finance, marketing, etc.) it suffers from three main drawbacks. Inference in such Bayesian models is usually implemented using Markov chain Monte Carlo (MCMC) methods (see e.g.,~\citet{metropolis1953equation},~\citet{Hastings-1970},~\citet{ george1993variable},~\citet{ Ishwaran-Rao-2005}) which are not typically scalable and impose a high computational cost in studies with a large number of variables. Further, in many settings, the assumption of a known prior distribution $\mu$ cannot be fully justified and a more flexible approach that can handle both ``dense'' (where many predictors have non-zero effects) and ``sparse'' (where most of the predictors have no effect) regimes with a variety of signal distributions, is called-for. For example, signals are expected to be sparse in signal processing \citep{candes2014mathematics}
while dense signals are expected in polygenic risk score prediction \citep{lewis2020polygenic}, GWAS data \citep{yang2010common} and econometrics \citep{angrist2004control}. Finally, prior specification in Bayesian analysis typically involves the specification of some hyperparameters, which can be challenging in applications.

In the recent papers~\citet{carbonetto2012scalable} and~\citet{kim2022flexible} the authors propose and study a ``variational empirical Bayes'' approach --- it combines {\it empirical Bayes} (EB) inference for problem~\eqref{eq:lin_reg} via a {\it variational approximation} (VA) for the class of {\it scale mixture of normal priors} (on the regression coefficients). Variational Bayesian inference (see e.g.,~\citet{jordan1999introduction},~\citet{wainwright2008graphical},~\citet{Blei-Et-Al-VI}) is an appealing approach to approximating complex posterior distributions offering computational advantages. The basic idea is to recast the problem of computing posterior probabilities --- which is inherently a high-dimensional integration problem --- as an optimization problem by introducing a class of approximating distributions, and then optimizing the Kullback-Leibler (KL) divergence to find the distribution within this class that best matches the posterior. As in usual EB inference (see e.g.,~\citet{kiefer1956consistency},~\citet{Robbins-56},~\citet{efron1973stein}), the prior $\mu$ is assumed to be unknown, but belonging to a class of distributions, which is then estimated from the data. Although the literature on EB methods is very long, relatively few papers have addressed the problem of EB inference in the context of multiple linear regression and have either focused on relatively inflexible priors, or have been met with
considerable computational challenges (see e.g.,~\citet{Nebebe-Stroud-1986},~\citet{George-Foster-2000},~\citet{Yuan-Lin-2005}).

The approach advocated in~\citet{carbonetto2012scalable} and~\citet{kim2022flexible} has several advantages over other high-dimensional linear regression techniques, namely: (i) it is fast (as opposed to modern MCMC techniques for Bayesian high-dimensional regression), (ii) it is capable of adapting to both sparse and dense settings as it estimates the prior $\mu$ from the class of all scale mixtures of normal priors; and (iii) it does not require any tuning as it does not involve the choice of (hyper)-parameters. 

In this paper we study the theoretical properties of the above variational EB procedure, in the high-dimensional regime where both $n$ and $p$ grow. In fact, as far as we are aware, ours is the first paper that provides rigorous statistical guarantees for the method, under suitable assumptions.

Let us briefly describe our setup. In contrast to~\citet{carbonetto2012scalable} and~\citet{kim2022flexible} we consider a different class of priors, namely, (\textcolor{black}{any closed subset of}) the class $\mathcal{P}$ of all bounded probability distributions on $\R$. It can be argued that our choice of priors is more flexible and allows for both ``parametric'' and ``nonparametric'' model assumptions. We estimate $\mu$ from the observed data as a first step in our EB methodology (see e.g.,~\citet{robbins1950generalization},~\citet{Robbins-56},~\citet{Jiang-Zhang-2009}). In Section~\ref{sec:NPMLE} we study the nonparametric maximum likelihood estimator (NPMLE) of $\mu$ which is obtained by maximizing the marginal likelihood of the data $\mathbf{y}$ (or an approximate maximizer thereof) over this parameter space $\mathcal{P}$, i.e.,  
\begin{equation}\label{eq:mu-NPMLE}
\hat \mu_\ML \,\in \,\arg \max_{\mu \in \mathcal{P}} \; \log m_\mu({\bf y}) 
\end{equation}
where
\begin{align}\label{eq:marginal}
m_\mu({\bf y}):=&\int\Big(\frac{1}{2\pi\sigma^2}\Big)^{\frac{n}{2}} \exp\Big(-\frac{1}{2\sigma^2}\|{\mathbf y}-{\mathbf X}{\bm \beta}\|_2^2\Big)d \mu^{\otimes p}({\bm \beta}),
\end{align}
and $\mu^{\otimes p}$ denotes the $p$-fold product distribution with each marginal being $\mu$. Here, for any $\mathbf{a} = (a_1,\ldots, a_p) \in \R^p$, $\|\mathbf{a}\|_2^2 := \sum_{i=1}^p a_i^2$.  Let $\hat \mu_\ML$ be the nonparametric maximum likelihood estimator (NPMLE) of $\mu^*$ --- the true unknown prior of the $\beta_i$'s --- obtained by maximizing~\eqref{eq:marginal}. Under minimal conditions on the design matrix $\mathbf{X}$ and the class $\mathcal{P}$, we derive a non-asymptotic convergence rate for $\hat \mu_\ML$ to $\mu^*$ under a smoothed Hellinger metric when $p$ and $n$ are both large (see Theorem~\ref{thm:ml_rate}).

However, the above integral (in~\eqref{eq:marginal}) is difficult  to compute and/or approximate via MCMC methods, especially when $p$ is large. This motivates the use of a variational approximation, as indicated above. We now introduce our key \textit{methodological contribution} (also see~\citet{carbonetto2012scalable},~\citet{kim2022flexible}): (i) we replace the marginal log-likelihood of the data $\log m_\mu(\mathbf{y})$, also known as the {\it evidence}, with the {evidence lower bound} (ELBO) obtained from a {\it naive mean field} (NMF) approximation (see e.g.,~\citet{Blei-Et-Al-VI}), and (ii) maximize this lower bound with respect to (w.r.t.) $\mu$ (belonging to a suitable class of priors). Formally, we set
\begin{equation}\label{eq:pi-hat}
\hat{\mu}_{\MF} \,\in \,\arg \max \limits_{\mu \in \mathcal{P}} \; \mathrm{NMF}(\log m_\mu(\mathbf{y});\mu),
\end{equation} 
where 
\begin{equation}\label{eq:Max-ELBO}
\mathrm{NMF}(\log m_\mu(\mathbf{y});\mu) := \sup_{\nu \in \Q} \; \ELBO(\nu;\mu)
\end{equation}
with $\Q$ being the class of all product distributions on $\R^p$ (i.e., $\nu \in \Q$ if and only if $\nu = \Pi_{i=1}^p \nu_i$, for $\nu_i$ being a distribution on $\R$, for $i=1,\ldots, p$), and 
{\small \begin{align}\label{eq:ELBO}
\ELBO(\nu;\mu)
=-\frac{n}{2}\log(2\pi\sigma^2)+\left\{-\frac{1}{2\sigma^2}\E_{\nu}\Big[\|{\bf y}-{\bf X}{\bm \beta}\|_2^2\Big]-\sum_{i=1}^p \E_{ \nu_i}\log \Big(\frac{d\nu_i}{d\mu}(\beta_i)\Big)  \right\}.
\end{align}}
We can think of $\hat{\mu}_{\MF}$ as a (naive) mean field MLE; cf.~\eqref{eq:mu-NPMLE} and~\eqref{eq:pi-hat}. We refer the reader to the discussion around~\eqref{eq:kl_expression} for the derivation of the ELBO.

The quadratic nature of the log-likelihood for the linear model dictates that for any $\mu \in \mathcal{P}$ and any maximizer $\bar{\nu} = \prod_{i=1}^{p}\bar{\nu}_i$ of~\eqref{eq:Max-ELBO}, $\bar{\nu}_i$ is a quadratic tilt of $\mu$ (see \citet[Lemma 6]{mukherjee2022variational} and~\citet[Section 5.3]{wainwright2008graphical}); Section \ref{sec:MF} has more details. Thus in this specific case, \eqref{eq:Max-ELBO} reduces to a (high-dimensional)  optimization problem over real parameters, where one maximizes over the choice of the tilts. In our subsequent discussion, we re-parametrize the $\ELBO$ objective in terms of the tilt parameters, and denote this objective function as $\widetilde{M}_p(\cdot, \mathbf{X}^{\top}\mathbf{y}/\sigma^2, \mu)$ (see \eqref{eq:Tilde-M_p}).

\begin{figure}
\centering
\begin{subfigure}[b]{0.32\textwidth}
\includegraphics[width=\textwidth]{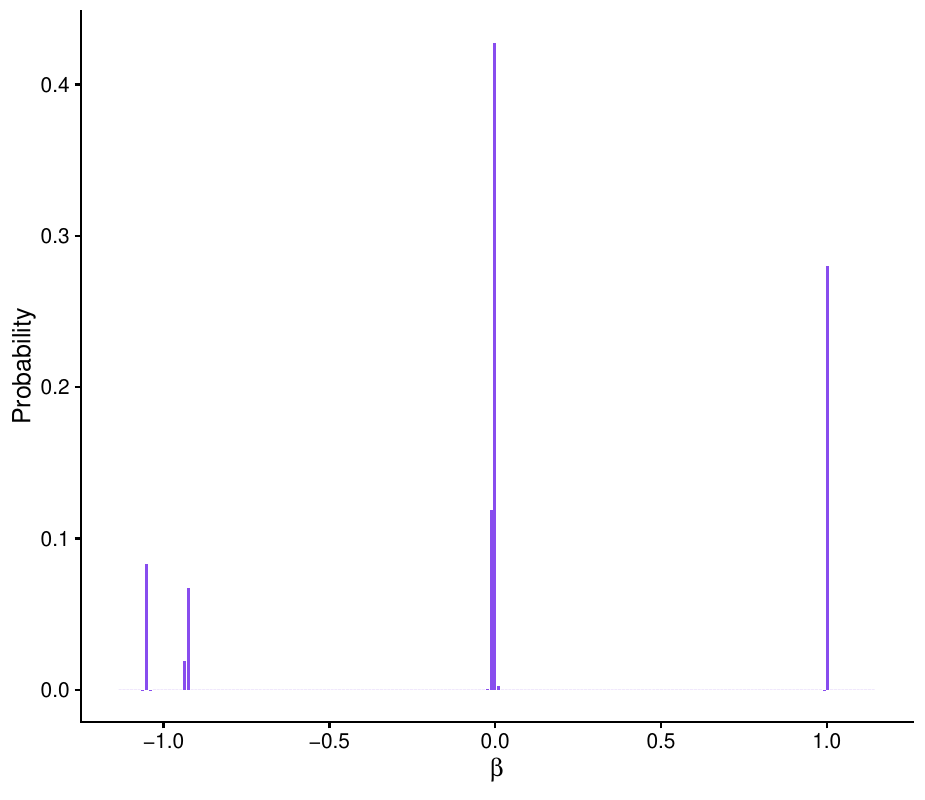}
\end{subfigure}
\hfill
\begin{subfigure}[b]{0.32\textwidth}
\includegraphics[width=\textwidth]{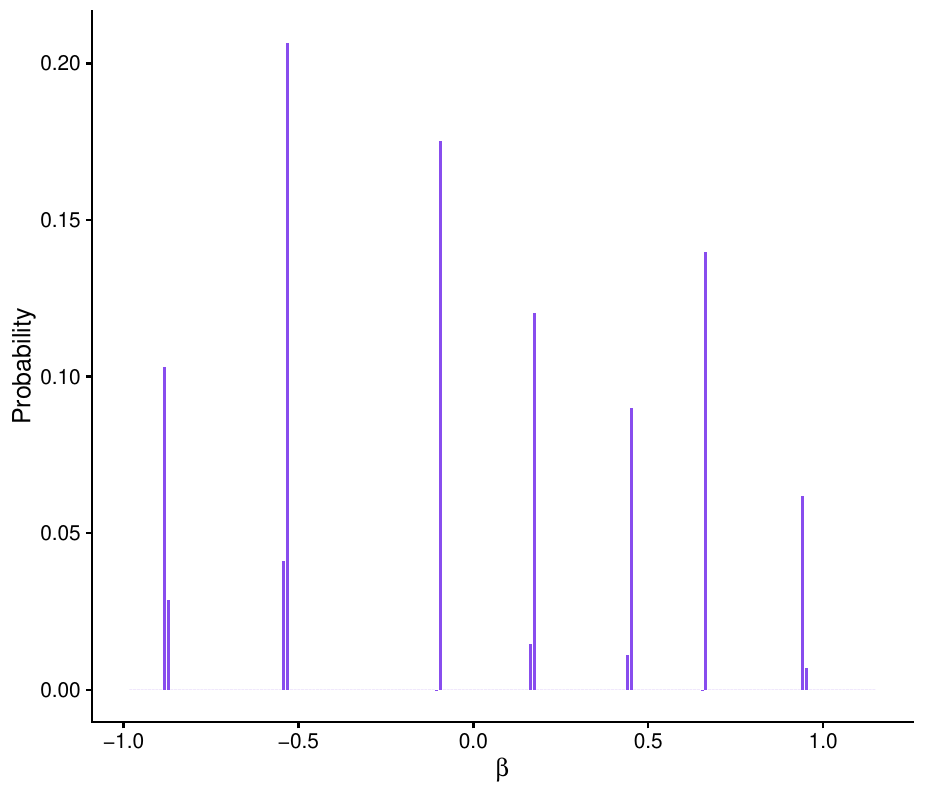}
\end{subfigure}
\hfill
\begin{subfigure}[b]{0.32\textwidth} 
\includegraphics[width=\textwidth]{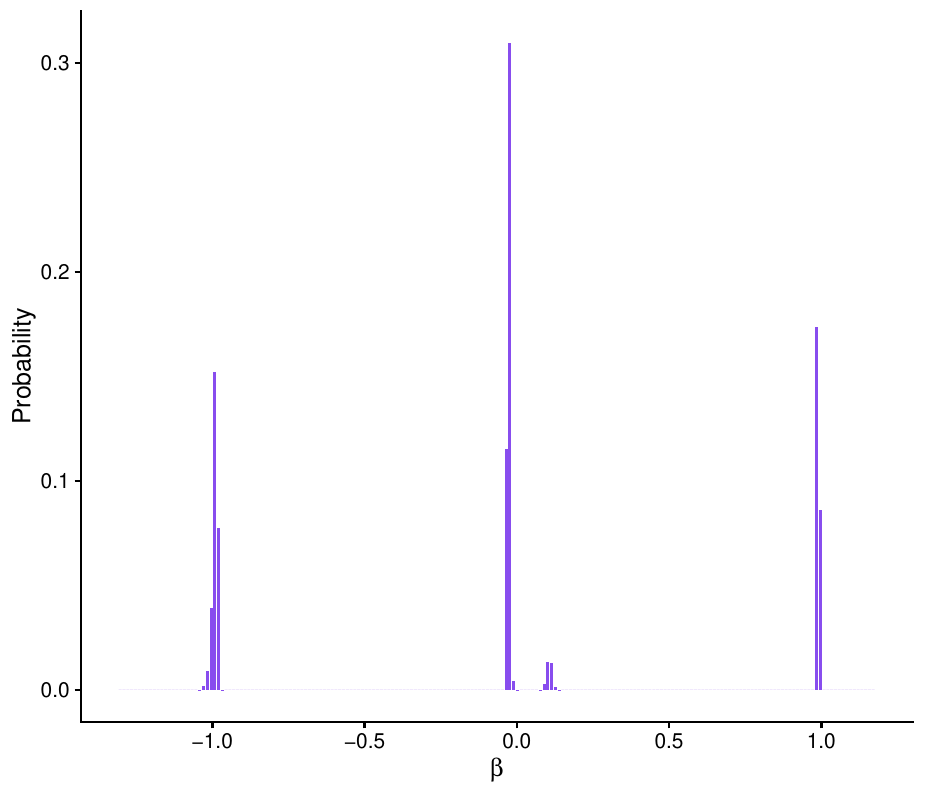}
\end{subfigure} 
\vspace{-0.25cm}
\caption{\small{We plot $\hat{\mu}_\MF$ obtained using the NMF-EB method (see~\eqref{eq:pi-hat}) from 3 different simulation experiments with $n=500$, $p=100$ and $\sigma = 0.1$. (a) {\bf Left plot}: we choose an i.i.d.~Gaussian design (suitably normalized) and the true prior $\mu^*$ is {\it discrete} with $\mu^*(\{-1\}) = \mu^*(\{1\}) = 1/4$, $\mu^*(\{0\})= 1/2$. We see that $\hat{\mu}_\MF$ is close to $\mu^*$ and also has three main regions where mass is concentrated, namely around $-1,0$, and $1$. (b) {\bf Center plot}: we consider $\mu^* = \mathrm{Unif}([-1,1])$ and i.i.d.~Gaussian design (suitably normalized). Observe here that the mass points of the estimated prior $\hat{\mu}_\MF$ are fairly evenly distributed between $[-1,1]$. (c) {\bf Right plot}: we consider a setting where the rows of the design matrix are i.i.d.~samples from $\mathcal{N}(0, \Sigma)$ (suitably normalized), where $\Sigma_{ij} = (0.5)^{|i-j|}$ (thus the predictors are dependent); $\mu^*$ is the same as in (a). Even in this setting the prior is well approximated.}}
\label{fig1} 
\vspace{-0.3cm}
\end{figure}

In Theorem~\ref{lem:approx} we show that, under suitable assumptions, the marginal log-likelihood function $\log m_\mu(\mathbf{y})$ can be approximated by its mean field counterpart $\mathrm{NMF}(\log m_\mu(\mathbf{y});\mu)$, uniformly in $\mu$.  Let $\hat \mu_\MF$ be the mean field MLE obtained by maximizing the above approximation in $\mu$. We derive a non-asymptotic convergence rate of $\hat \mu_\MF$ to $\mu^*$ when $p$ and $n$ are both large. This gives a computationally feasible alternative to estimating $\mu^*$ (via~\eqref{eq:pi-hat}); Figure~\ref{fig1} illustrates the satisfactory finite sample performance of $\hat \mu_\MF$ for different $\mu^*$'s.

Having obtained a consistent estimator of $\mu^*$, we turn to the task of approximating the  posterior distribution $\mathbb{P}_{\mu^{*}}(\cdot \mid \mathbf{y})$ --- the true posterior of $\bbeta$ given the data $\mathbf{y}$.
Note that the NMF scheme~\eqref{eq:Max-ELBO} equivalently solves the optimization problem
\begin{equation}\label{eq:Min-KL-Div} 
\inf_{\nu \in \mathcal{Q}} \dkl( \nu \| \mathbb{P}_{\mu}(\cdot\mid \mathbf{y})),
\end{equation}
where  $\dkl( P \| Q)$ denotes the Kullback–Leibler divergence between the distribution $P$ and the reference probability distribution $Q$; i.e., the optimization problems \eqref{eq:Max-ELBO} and \eqref{eq:Min-KL-Div} have the same set of optimizers (see~\eqref{eq:kl_expression} and Remark~\ref{rem:kl_minimization} for a discussion). Further, recall that for any optimizer $\bar{\nu} = \prod_{i=1}^{p} \bar{\nu}_i$ of \eqref{eq:Min-KL-Div}, $\bar{\nu}_i$ are quadratic tilts of the prior $\mu$. This suggests a natural strategy to approximate the posterior distribution $\mathbb{P}_{\mu^*}(\cdot \mid \mathbf{y})$: given any consistent estimator $\hat{\mu}$ of $\mu^*$, one can estimate the relevant quadratic tilts by optimizing $\widetilde{M}_p(\cdot, \mathbf{X}^{\top}\mathbf{y}/\sigma^2, \hat{\mu})$; the posterior $\mathbb{P}_{\mu^*}(\cdot \mid \mathbf{y})$ can then be approximated using a product measure $\hat{\bar{\nu}} = \prod_{i=1}^{p} \hat{\bar{\nu}}_i$,  where the components $\hat{\bar{\nu}}_i$ are constructed via the estimated quadratic tilts of the (estimated) prior $\hat{\mu}$. In Theorem~\ref{thm:inference}, we establish that this strategy consistently approximates the true posterior $\mathbb{P}_{\mu^*}(\cdot \mid \mathbf{y})$ in 1-Wasserstein distance, provided the objective $\widetilde{M}_p(\cdot, \mathbf{X}^{\top}\mathbf{y}/\sigma^2,\mu^*)$ (suitably reparametrized) has a sequence of ``well-separated" optimizers (see Definition \ref{defn:separation_property} for the precise details). This `well-separated' condition holds under reasonable assumptions on the prior $\mu^*$ and the additive variance $\sigma^2$, as established in \citet[Lemma 5]{mukherjee2022variational}. Indeed, we expect this requirement to be crucial for practical, computationally feasible downstream posterior inference --- see Remark \ref{rem:well_separated_req} for further discussions on this condition. In particular, Theorem~\ref{thm:inference} allows us to consistently estimate the posterior variances, at least in an averaged sense (see Remark~\ref{rem:Post-Var} for a detailed discussion); this is in sharp contrast to the folklore that VI usually underestimates variances.

\begin{algorithm}\caption{Computing NMF-EB estimate $\hat{\mu}_\MF$}\label{alg:npmle}
        \begin{algorithmic}
    \item[]   \textbf{Input}: ${\bf y} \in \R^n$, ${\bf X} \in \R^{n \times p}$. 
    
     \item[(1)] Assume $\mu$ is supported on 
     $[-1,1]$ (any bounded interval will work). 
     
   \item[(2)] Let $\{a_1, \ldots, a_k\}$ be a fixed grid on $[-1,1]$ (e.g., $k \approx $\;100-300). 
   
     \item[(3)] Approximate the prior as $\mu \approx \sum_{r=1}^{k} p_r \delta_{a_r}$, i.e., a discrete measure supported on $\{a_r\}_{r=1}^k$ with unknown probability vector $\{p_r\}_{r=1}^{k}$~\citep{koenker2014convex}. 
     \item[(4)] To compute $\hat{\mu}_\MF$, optimize the probability weights $\{p_r: 1\leq r \leq k\}$ of the NMF approximation by gradient descent (or any optimization technique). This yields 
     \[
(\hat{\mathbf p},\hat{\bgamma})= \arg\max_{\mathbf p\in\Delta_k,\,\bgamma\in\R^p}
\widetilde M_p(\bgamma,\mathbf w,\nu(\mathbf p)),
\qquad
\hat\mu_{\MF}=\nu(\hat{\mathbf p}),
\]
where $\Delta_k$ denotes the $k$-dimensional probability simplex and $\widetilde M_p$ is the NMF objective.
 \end{algorithmic}
      \end{algorithm}


The conclusions of  Theorems~\ref{lem:approx} and~\ref{thm:inference} can be directly used to recover characteristics of $\mu^*$ and provide valid posterior inference for $\bbeta$ (see Corollary~\ref{cor:takeaways}): 
\begin{itemize}
    \item[(i)] \textit{Credible intervals for effect sizes}: 
the estimated posterior $\hat{\bar{\nu}}$ yields Bayesian credible intervals for the effect sizes $\{\beta_i: 1\leq i \leq p\}$. Using Theorem \ref{thm:inference}, we will derive rigorous guarantees on the coverage of these credible sets.

    \item[(ii)] \textit{Optimal estimation of the regression coefficients}:  under an $L^2$ loss, the optimal estimator of $\bbeta$ is given by $\mathbb{E}_{\mu^*}[\bbeta \mid \mathbf{y}]$. A tractable, data driven surrogate of this posterior mean is given by $\mathbb{E}_{\hat{\bar{\nu}}}[\bbeta]$. We will establish that this estimator is Bayes optimal in terms of $L^2$ risk. In particular, this establishes that it is possible to attain Bayes optimal performance with respect to the oracle Bayes procedure using the estimated posterior $\hat{\bar{\nu}}$.  Figure~\ref{fig2} shows that the estimated posterior means $\mathbb{E}_{\hat{\bar{\nu}}}[\bbeta]$ obtained via the NMF-EB approach  are quite close to true posterior means $\mathbb{E}_{\mu^*}[\bbeta \mid \mathbf{y}]$ (approximated via Gibbs sampling).

    \item[(iii)]\textit{Estimate the proportion of non-nulls}: our framework can accommodate $\mu^*(\{0\}) >0$ which implies that the underlying signal $\bbeta$ has order $p$ zero entries. In this sparse setting, the proportion of non-nulls $ \sum_{i=1}^{p} \mathbf{1}(\beta_i \neq 0)/p \approx 1- \mu^*(\{0\})$ is of paramount importance in multiple testing applications~\citep{efron2012large}. The consistency of $\hat{\mu}_{\MF}$ implies that $\hat{\mu}_{\MF}([-\epsilon, \epsilon]) \to \mu^*(\{0\})$  (in probability, as $p \to \infty$ followed by $\epsilon \to 0$). In words, the mean field MLE allows us to consistently recover the proportion of non-null coordinates.

\end{itemize}

\begin{figure}
\centering
\begin{subfigure}[b]{0.325\textwidth}
\includegraphics[width=\textwidth]{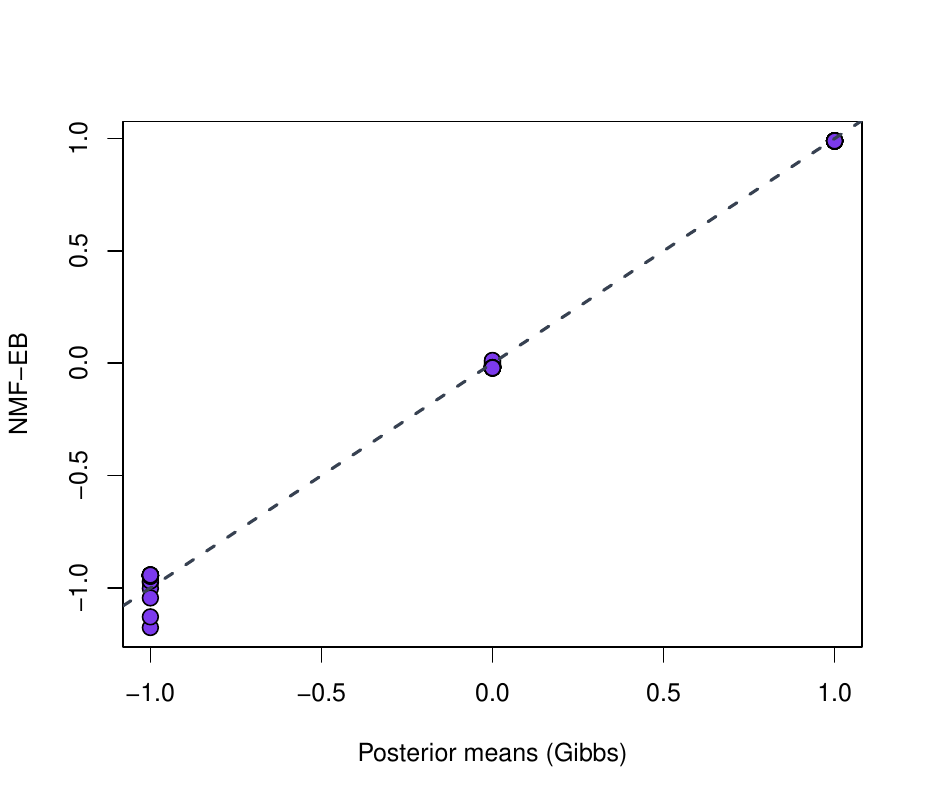}
\end{subfigure}
\hfill
\begin{subfigure}[b]{0.325\textwidth}
\includegraphics[width=\textwidth,]{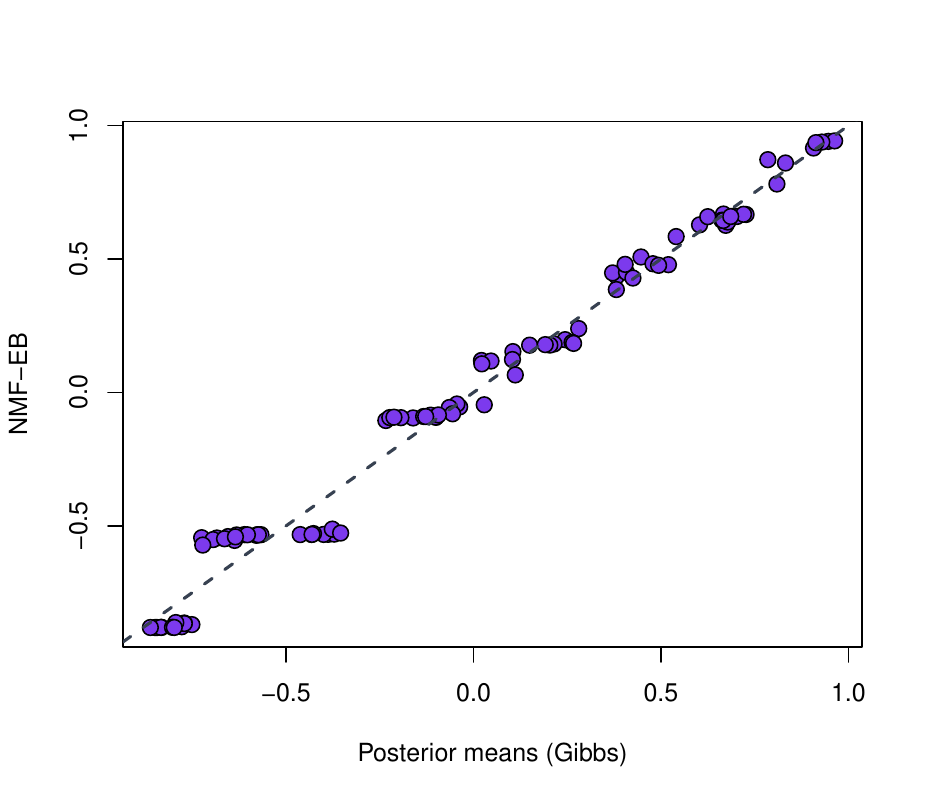}
\end{subfigure}
\hfill
\begin{subfigure}[b]{0.325\textwidth} 
\includegraphics[width=\textwidth]{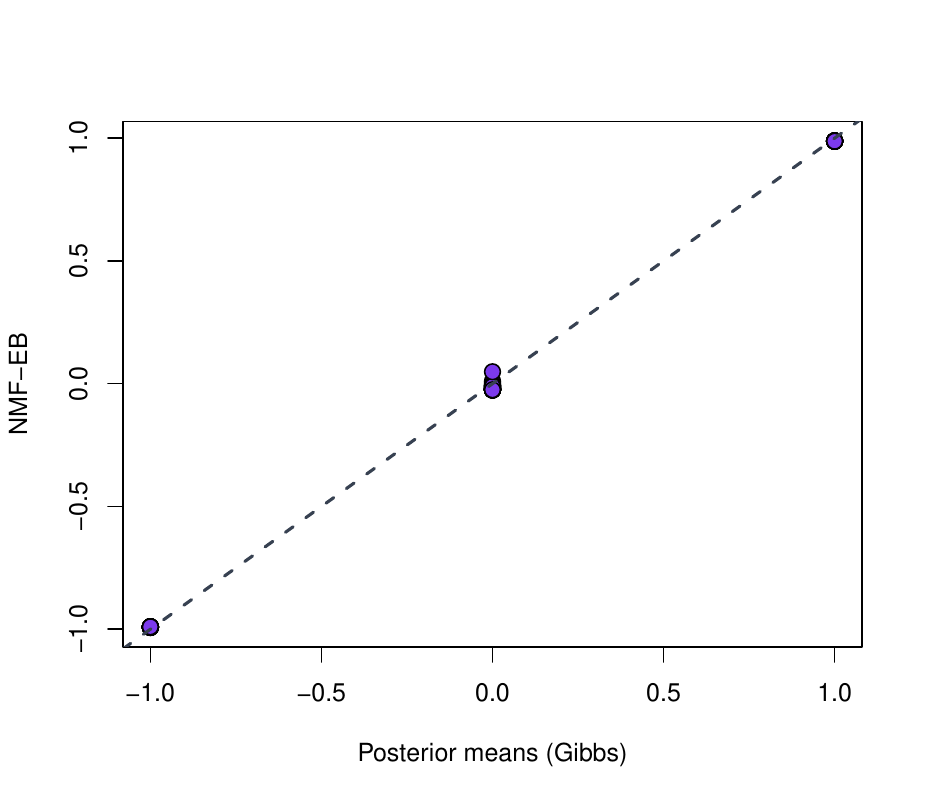}
\end{subfigure} 
\vspace{-0.25cm}
\caption{\small{The posterior means obtained from the NMF-EB approach are plotted against the posterior means estimated via Gibbs sampling (the gold standard approximation of the true posterior means), for the 3 simulation examples shown in Figure~\ref{fig1}. In all settings we see that our NMF-EB approach leads to posterior mean estimates that closely match the ones obtained via MCMC. }}
\label{fig2} 
\vspace{-0.3cm}
\end{figure}

One of the advantages of our current analysis is that we allow the design matrix $\mathbf{X}$ to be either deterministic or random. Moreover, we allow arbitrary dependence between the entries of $\mathbf{X}$; our only requirement on $\mathbf{X}$ (for consistent recovery of $\mu^*)$ is that Assumption~\ref{assump:On-Design} holds.  We further illustrate that without Assumption~\ref{assump:On-Design} we cannot expect consistent estimation of $\mu^*$ (see Example~\ref{eg:eigen}).


\subsection{Prior art} 
We provide an overview of prior related research in this section. 

\noindent \textbf{Empirical Bayes methods}: EB methods for sequence models have been long incorporated in the statistician's toolkit, and are used extensively for deconvolution/denoising problems (see e.g.,~\citet{MR0133191}, \citet{efron1972empirical,efron1972limiting,efron1973combining,efron1973stein},~\citet{MR0696849}, \citet{MR2724758}, \citet{MR3264543}, \citet{johnstone2019gaussian} and the references therein). Developed originally for univariate, homoscedastic sequence models, the methodology and associated theory has been extended to multivariate, heteroscedastic settings in recent years; see e.g.,~\citet{MR2847987}, \citet{MR3036408}, \citet{MR3559950}, \citet{MR3832220}, \citet{banerjee2020nonparametric}, \citet{MR4109006}, \citet{soloff2021multivariate}. In particular, the use of the nonparametric maximum likelihood estimator (NPMLE) for the prior, originally due to \citet{robbins1950generalization}, has witnessed a resurgence in recent years \citet{Jiang-Zhang-2009},~\citet{MR2798524},~\citet{MR3465819},~\citet{gu2016problem},~\citet{koenker2017rebayes},~\citet{feng2018approximate},~\citet{MR3983318},~\citet{kim2020fast},~\citet{polyanskiy2020self},~\citet{MR4102674},~\citet{MR4528473} leading to the widespread application of EB methods in diverse  applications \citet{ver1996parametric},~\citet{MR1946571},~\citet{brown2008season},~\citet{persaud2010comparison}.  EB methods have also been extended for the Poisson sequence model and other associated non-Gaussian analogues (see~\citet{MR0212927},~\citet{MR0258180},~\citet{MR3174656},~\citet{polyanskiy2021sharp},~\citet{jana2022optimal},~\citet{jana2023empirical} and the references therein). The literature on EB methods is substantial --- we refer the interested reader to \citet{MR0696849},~\citet{casella1985introduction},~\citet{zhang2003compound},~\citet{MR3264543},~\citet{maritz2018empirical},~\citet{efron2021empirical} for several excellent surveys and monographs on this subject. 

    On the theoretical front, the consistency of the NPMLE was established in \citet{MR0086464},~\citet{jewell1982mixtures},~\citet{heckman1984method},~\citet{lambert1984asymptotic},~\citet{pfanzagl1988consistency},~\citet[Section 4.2]{groeneboom1992information}; see also the recent survey paper~\citet{chen2017consistency}. For more recent theoretical results for the nonparametric empirical Bayes estimator (in the compound setting) see,~\citet{Ghosal-vdVaart-2001},~\citet{Brown-Greenshtein-2009},~\citet{Jiang-Zhang-2009} and the references therein. The theoretical properties of spike and slab based EB posteriors have been analyzed recently in \citet{MR3885271,zhang2020convergence}.
    Complementing these theoretical advances, there has been significant progress on the properties of the NPMLE in deconvolution problems \citep{lindsay1995mixture}. Several algorithms have been proposed to compute this estimate in practice \citet{MR0212927},~\citet{MR0258180},~\citet{lemon1969empirical},~\citet{bennett1972continuous},~\citet{laird1978nonparametric},~\citet{MR0783372},~\citet{lesperance1992algorithm},~\citet{liu2007partially},~\citet{MR2325271},~\citet{maritz2018empirical}. Exploiting the structural properties of the NPMLE in the Gaussian deconvolution problem,  \citet{koenker2014convex} have developed a particularly efficient convex optimization framework for computing the NPMLE --- this facilitates the use of out-of-the-box efficient convex solvers in this setting. In turn, this facilitates the fast computation of the NPMLE in practice. These insights have been extended to the heteroscedastic case in \citet{soloff2021multivariate}.

    Despite this rapid progress in EB methods over the past decades, most methods are almost completely restricted to the sequence model setting. \citet{carbonetto2012scalable},~\citet{wang2020simple},~\citet{kim2022flexible} use variational approximations (VA) to implement the EB approach in regression problems, and observe superior performance in genetic fine mapping applications. Unfortunately, these methods have lacked theoretical guarantees so far, which prevents more widespread use in other applications. In fact, our paper provides the theoretical justification for the VA-EB approach outlined in the above papers. 

    A notable extension in this regard is the recent article \citet{zhong2022empirical}. In this paper, the authors investigated EB approaches in the context of the PCA problem. Assuming a low rank spike matrix model, the authors combine approximate message passing (AMP) ideas with the NPMLE to derive consistent estimators for the priors on the elements of the low rank spike vectors. This methodology has rigorous guarantees, and exhibits excellent performance on simulated and real data. However, this approach has a crucial shortcoming --- the guarantees are derived under a very specific generative model on the observed data. As a result, the algorithms and the associated theory do not generalize directly to settings beyond PCA. \newline

\noindent \textbf{Variational inference}: Variational approximations provide a general purpose method to approximate intractable probability distributions arising frequently in Bayesian inference (see e.g.,~\citet{wainwright2008graphical}). We refer the interested reader to \citet{Blei-Et-Al-VI} for an introduction to this general area. The naive mean field (NMF) approximation, where the probability distribution of interest is approximated (in KL divergence) by a product distribution represents the simplest form of variational inference. The NMF optimization problem is numerically scalable, which makes it an attractive option for practitioners. Despite the popularity of these methods in diverse applications, the supporting theoretical evidence is surprisingly recent. Rigorous results for the performance of variational methods under classical large $n$ fixed $p$ asymptotics are derived in \citet{wang2019variational,MR4011769} (also see~\citet{MR2221291}). Variational inference methods for the low dimensional linear model were initially explored in \citet{neville2014mean} and~\citet{MR3709863}, while variational approximation based variable selection methods were studied in \citet{carbonetto2012scalable}. 

    Our understanding of variational methods in high-dimensions under sparsity has progressed rapidly in the last few years. \citet{MR3595173},~\citet{han2019statistical},~\citet{ray2020spike},~\citet{MR4102680},~\citet{zhang2020convergence},~\citet{Zhang-Gao-2020},~\citet{Ray-Szabo-2022} study the asymptotic contraction properties of approximate variational methods in high-dimensions. In particular, these works derive explicit conditions for the minimax optimal contraction properties of variational posteriors. At a high-level, this line of research establishes that if the prior puts sufficient mass near the true data model, and the signal strength is sufficiently high, the variational posterior often contracts to the truth at the same rate as the true posterior. In turn, the true posterior usually contracts to the truth at the frequentist optimal (minimax) rate. Consequently, in these regimes, variational methods provide the best of both worlds --- they achieve optimal statistical performance and are simultaneously computationally efficient. On the other hand, the NMF approximation can be grossly inaccurate in high dimensions in the absence of sparsity~\citet{ghorbani2019instability},~\citet{MR4203332},~\citet{MR4600991}. In some of these settings, advanced mean-field approximations, e.g., the Thouless-Anderson-Palmer approximation from statistical physics can often address the shortcomings of the NMF approximation \citet{MR4203332},~\citet{MR4600991}.  We also refer to \citet{katsevich2023approximation} for recent progress on the accuracy of Gaussian variational inference in high-dimensions. 
    
    Beyond the realm of regression models, VI methods have been theoretically analyzed in several settings (see e.g., \citet{MR2796867},~\citet{MR2906876},~\citet{MR3127853},~\citet{MR4152113} and references therein). We emphasize that our techniques are independent of the ones used in these prior works.

\noindent \textbf{Naive mean field and optimal transport}: The optimization problem corresponding to the NMF approximation is generally non-convex. It is typically solved by iterating the fixed point equations corresponding to the mean parameters (\citet{wainwright2008graphical}). Under special conditions, it is well understood that this system is contractive, and the fixed point corresponds to the mean of the distribution. 

We note that there has been very interesting recent progress connecting variational inference with the advances in optimal transport ( \citet{lambert2022variational}, \citet{yao2022mean}, \citet{diao2023forward}). In a nutshell, this line of work leverages advances in computational optimal transport to solve the optimization problem over probability distributions arising in the NMF scheme. \newline

\noindent 
\textbf{Outline:} 
The rest of the paper is structured as follows. In Section~\ref{sec:NPMLE}, we discuss the consistency of the NPMLE in the Bayesian linear regression model. In Section~\ref{sec:MF}, we discuss the NMF approximation, its theoretical properties and explore the associated optimization algorithm. In Section~\ref{sec:post_inf}, we turn to posterior inference using the estimated prior. We establish that under suitable conditions on the model and prior, the estimated posterior attains the same performance as the oracle posterior distribution. We survey some directions for future research and connect our results with the existing literature in Section \ref{sec:discussion}. Finally, the main results are proved in Section \ref{sec:Proofs} and Section \ref{sec:thm_inf}. We defer some technical proofs to the Appendix. 
\vspace{6pt}

\noindent 
\textbf{Notation:} Throughout the paper, we use the usual Landau notation $O(\cdot)$, $o(\cdot)$, $\Theta(\cdot)$ for sequences of real numbers.   $\P_{\mu^*}$ denotes the joint law of $(\mathbf{y},\bbeta)$ under model \eqref{eq:lin_reg}, where  $(\beta_1,\ldots,\beta_p)\stackrel{iid}{\sim}\mu^*$. For sequences of random variables,  we  use the analogous notation $O_{\P_\mu^*}(\cdot)$, $o_{\P_{\mu^*}}(\cdot)$ and $\Theta_{\P_{\mu^*}}(\cdot)$. We use the notation $a\lesssim b (a\lesssim_K b)$ to denote the existence of a finite positive constant $C$ which is universal (depends only on  $K$) such that $a\le Cb$. For any natural number $n \in \mathbb{N}$, $[n]=\{1, \ldots, n\}$. Finally, for a symmetric $n \times n$ matrix $\mathbf{A}$, we use $\|\mathbf{A}\|_2$ to denote the usual $L^2$ operator norm.  Throughout, $d_{\mathcal{H}}$ will denote the Hellinger distance between probability distributions. Finally, for any probability measure $\mu$, $\mu \star N(0,\sigma^2)$ will denote the usual convolution between the measures $\mu$ and $N(0,\sigma^2)$. 
\vspace{6pt}

\section{Nonparametric maximum likelihood}\label{sec:NPMLE}
Recall the Bayesian linear model in~\eqref{eq:lin_reg} and~\eqref{eq:iid-Beta}.
Then the joint distribution $\P_\mu$ of $({\bf y}, {\bm \beta})$ is given by
\begin{align}\label{eq:post}
\frac{d \P_\mu({\mathbf y},{\bm \beta})}{d\lambda^{\otimes n}d\mu^{\otimes p}}:=\Big(\frac{1}{2\pi\sigma^2}\Big)^{\frac{n}{2}} \exp\Big(-\frac{1}{2\sigma^2}\|{\mathbf y}-{\mathbf X}{\bm \beta}\|_2^2\Big),
\end{align}
where $\lambda^{\otimes n}$ refers to the Lebesgue measure in $\R^n$.
Further, we use $\Po(\cdot \mid \mathbf{y})$ to denote the posterior distribution of $\bm \beta$ given $\mathbf{y}$. 
The marginal distribution of ${\bf y}$ has a density with respect to $\lambda^{\otimes n}$, given by~\eqref{eq:marginal}.

In this section, our goal is to consistently estimate the prior $\mu$, having observed the vector ${\bf y}\in \R^n$, in the high dimensional setting when $n,p\to\infty$. To this end, throughout the paper, we will assume that the design matrix ${\mathbf X}$ satisfies the following condition.
\begin{assumption}\label{assump:On-Design}
For a matrix $A$ let $\lambda_{\rm min}(A)$ and $\lambda_{\rm max}(A)$ denote the minimum and maximum eigenvalue of $A$. Suppose
\begin{align}\label{eq:eigen}
C_1\le \frac{\lambda_{\rm min}({\mathbf X}^\top {\mathbf X})}{\sigma^2}\le  \frac{\lambda_{\rm max}({\mathbf X}^\top {\mathbf X})}{\sigma^2}\le C_2,
\end{align}
for some finite positive reals $C_1\in (0,1]$ and $C_2\in [1,\infty)$ (free of $n,p$). 
\end{assumption}
In particular, the above assumption ensures that the matrix ${\mathbf X}$ has full column rank, and thus $n\ge p$. We show that consistent estimation of $\mu$ may not be possible if the lower bound assumption in~\eqref{eq:eigen} is violated; see Example~\ref{eg:eigen} below. Thus the requirement \eqref{eq:eigen} is essentially tight for consistent recovery of $\mu$.

\begin{example}[On the necessity of Assumption~\ref{assump:On-Design}]\label{eg:eigen}
Let $\{{\bf e}_1,\ldots,{\bf e}_{p-1}\}$ be an orthonormal basis of the $(p-1)$ dimensional subspace $\{{\bf z} = (z_1,\ldots, z_p) \in \R^p:\sum_{i=1}^pz_i=0\}$. Let ${\bf X}$ be a $(p-1)\times p$ design matrix with rows given by ${\bf x}_i^\top={\bf e}_i^\top$ for $1\le i\le p-1$. Then $\mu$ is not identifiable and hence there does not exist any consistent estimator for $\mu$. Note that, in this case the eigenvalues of ${\bf X}^\top {\bf X}$ are $\{1,1,\ldots,1,0\}$. A formal proof of this is deferred to Appendix~\ref{Appendix-A}.
\end{example} 

\begin{remark}[Upper bound on the operator norm]\label{rem:norm} 
The upper bound in~\eqref{eq:eigen}  should be interpreted as a low signal-to-noise ratio (SNR) assumption. In fact, this SNR regime represents the most challenging and subtle setting for the prior recovery problem. As an example, if the entries of $\mathbf{X}$ are i.i.d.~from a centered distribution on $\mathbb{R}$ with finite fourth moment, using the Bai-Yin law we have $\lambda_{\max}(\mathbf{X}^{\top} \mathbf{X}) = \Theta_P(n)$~\citep{MR3837109}. The upper bound in~\eqref{eq:eigen} can then be achieved by suitably rescaling (by $1/\sqrt{n}$) the rows of the data matrix $\mathbf{X}$. Indeed, under this scaling, assuming $n \geq p$, the covariance matrix of the least squares estimator of $\bbeta$ is  $\sigma^2 (\mathbf{X}^{\top} \mathbf{X})^{-1}$. Hence, invoking~\eqref{eq:eigen}, each coordinate of the least squares estimate has non-trivial variance/fluctuation in the limit, and consequently $\bbeta$ cannot be consistently estimated under the average $L^2$-loss. However, as our results will demonstrate, consistent estimation of $\mu$ is still possible in this low SNR regime. 

We expect the analysis to be substantially easier in higher SNR regimes, where consistent estimation of $\bbeta$ is possible. Indeed, it is straightforward to consistently estimate the prior $\mu$ when we already have a good estimator of $\bbeta$ (by just using its empirical distribution).
\end{remark}

\subsection{Consistency of the NPMLE}
In this section, we derive the  consistency of the NPMLE $\hat \mu_\ML$. Formally, having observed ${\bf y}\in \R^n$, we now want to estimate the unknown prior distribution $\mu$. To this end, we assume that $\mu \in \mathcal{P}$ where $\mathcal{P}$ is  a  known (possibly infinite dimensional) subset of probability measures on $\R$. In particular, in this paper, we make the following assumption about $\mathcal{P}$:
\begin{assumption}\label{assump:On-P}
$\mathcal{P}$ is a closed subset of the set of all probability measures supported on $[-1,1]$.
\end{assumption}

\begin{remark}[On the compactness of the support] In Assumption~\ref{assump:On-P} we have assumed that the probability measures are compactly supported. 
See Section~\ref{sec:discussion} for an in-depth discussion of this assumption and possible extensions. 
\end{remark}

Note that it is possible that $\mathcal{P}$ is the set of all probability measures on $[-1,1]$. Further, the interval $[-1,1]$ is for definiteness, and could be replaced by any compact interval. We estimate $\mu$  from the observed data via the NPMLE of $\mu$ which is obtained by maximizing the marginal likelihood $m_\mu({\bf y})$ of the data $\mathbf{y}$ with respect to the unknown prior $\mu \in \mathcal{P}$ (see e.g.,~\citet{robbins1950generalization},~\citet{kiefer1956consistency},~\citet{Jiang-Zhang-2009}); in particular, we consider the following optimization problem: 
\begin{align}\label{eq:npmle}
\sup_{\mu\in \mathcal{P}}\log m_\mu({\bf y}).
\end{align}
Let $\{\varepsilon_p\}_{p\ge 1}$ be a (possibly random) sequence of non-negative reals, and let $\hat{\mu}_{\ML} \equiv \hat{\mu}_{\ML}^{(p)}$ denote any sequence of  `$\varepsilon_p$-approximate' optimizers of \eqref{eq:npmle}, i.e.,~$\hat{\mu}_{\ML}$ satisfies
\begin{equation}\label{eq:Approx-MLE}\sup_{\mu\in \mathcal{P}}\log m_\mu({\bf y})\le \log m_{\hat{\mu}_{\ML}}({\bf y})+\varepsilon_p.
\end{equation}
We will assume that $\varepsilon_p=o_{\mathbb{P}_{\mu^*}}(p)$ i.e., $\varepsilon_p/p \stackrel{P}{\to} 0$ as $p\to\infty$, under $\P_{\mu^*}$, where $\mu^*$ is the true (unknown) prior. In particular, if there is a global maximizer of~\eqref{eq:npmle}, we can take $\hat{\mu}_{\ML}$ to be any global maximizer (and choose $\varepsilon_p=0$ for all $p\ge 1$). However, it is possible that this optimization problem has no global maximizers, or has multiple global maximizers, in which case we can work with any optimizer $\hat{\mu}_{\ML}$ satisfying~\eqref{eq:Approx-MLE}. 

\begin{theorem}
\label{thm:ml_rate}
Fix $\mu^* \in \mathcal{P}$. Suppose that Assumptions \ref{assump:On-Design} and \ref{assump:On-P} hold. Let $\hat{\mu}_{\mathrm{ML}}$ denote a sequence of $\varepsilon_p$-approximate optimizers of \eqref{eq:npmle}. If $\varepsilon_p \le p^{5/6}$ a.s., then there exists  constants $c,C$ depending only on $C_1,C_2$ such that  
    \begin{align}
        \mathbb{P}_{\mu^*}\left( d_{\mathcal{H}}^2(\hat{\mu}_{\mathrm{ML}} \star N(0, C_1^{-1}), \mu^* \star N(0, C_1^{-1})) \geq  cp^{-1/6} \right) \le C \exp\Big(- \frac{p^{2/3}}{C}\Big). 
    \end{align}
\end{theorem}
Theorem~\ref{thm:ml_rate} (proved in Section~\ref{sec:Proofs}) establishes a non-asymptotic estimation rate for the NPMLE after smoothing with a Gaussian density, in the Hellinger metric. This result is similar in spirit to the consistency result of the marginal density in the Gaussian sequence model (see e.g.,~\citet[Theorem 4]{Jiang-Zhang-2009}). Note that to establish a polynomial convergence rate, one expects to need some degree of Gaussian smoothing. In fact, without Gaussian smoothing, deconvolution suffers from the slow logarithmic rate, which is known to be minimax optimal (see e.g.,~\citet{Meister-2009} and the references therein). Finally, we do not expect the rate in Theorem~\ref{thm:ml_rate} to be optimal. For the traditional Gaussian sequence model, \citet{Jiang-Zhang-2009} establishes near parametric rates of convergence. We believe characterizing the optimal rate of convergence is an interesting direction for future research.

\begin{remark}[On the proof of Theorem~\ref{thm:ml_rate}]
    We note that existing proofs for rates of convergence of the NPMLE rely heavily on the independence of the observations. Under independence, one can employ classical tools from empirical process theory. In contrast, the observations are dependent in our setting, and standard techniques are not immediately applicable. Consequently, we develop a more hands-on approach to determine the rate of convergence of $\hat{\mu}_{\mathrm{ML}}$ in our problem. In particular, we introduce a novel pseudo-metric $d_{\upsilon^2}(\cdot, \cdot)$ such that if two probability measures $\mu_1, \mu_2 \in \mathcal{P}$ are close w.r.t.~$d_{\upsilon^2}$, the corresponding marginal densities of $\mathbf{y}$ are deterministically close (see Lemma~\ref{lem:non-asymptotic-stability}). This pseudo-metric allows us to overcome the challenge caused by the dependence in the marginal likelihood of $\mathbf{y}$. Next, for an appropriate constant $c>0$, we cover the set $\mathcal{P}_p=\{\mu: d_{\mathcal{H}}(\mu\star N(0, C_1^{-1}), \mu^* \star N(0, C_1^{-1}) \geq c p^{-1/6}\}$ in the pseudo-metric $d_{\upsilon^2}$ (see Lemma~\ref{lemma:covering}). The $d_{\upsilon^2}$ pseudo-metric ensures that within a ball in this cover, the likelihood at any point can be approximated by that at its center. Consequently, to show that $\hat{\mu}_{\mathrm{ML}}$ does not belong to the set $\mathcal{P}_p$ with high-probability, it suffices to show that the likelihood at the truth $\mu^*$ is greater than that at the centers of the cover. We accomplish this using concentration (see Lemma~\ref{lem:exp_concentration}) and a data-processing inequality (see Lemma~\ref{lemma:chain_rule}). We believe that the pseudo-metric $d_{\upsilon^2}$ is one of our main conceptual contributions in this context, and believe that the proof strategy introduced here might be useful in other problems with dependent marginals.  
\end{remark}

\section{Variational approximation: naive mean field}\label{sec:MF}

Even though the NPMLE $\hat{\mu}_{\ML}$ is consistent under minimal assumptions, computing the NPMLE involves optimizing  $m_\mu(\mathbf{y})$, which is (usually) an intractable integral in $\R^p$. In this section we study a computationally feasible alternative that first approximates (lower bounds) $\log m_\mu({\bf y})$ via the ELBO~\eqref{eq:ELBO}, and then optimizes the lower bound with respect to $\mu \in \mathcal{P}$. To this end, we introduce some notation.

\begin{definition}[Quadratic tilt]\label{defn:Quad-Tilt}
Let $\mu$ be a probability distribution on $\R$. For any $(\gamma_1,\gamma_2)\in \R \times (0,\infty)$ set
\begin{align}\label{eq:cgamma}
c_\mu(\gamma_1,\gamma_2):=\log\Big[ \int  \exp\Big(\gamma_1 \beta-\frac{\gamma_2}{2}\beta^2\Big)d\mu(\beta)\Big],
\end{align}
and define the probability distribution $\mu_{\gamma_1,\gamma_2}$ on $\R$ by setting
\begin{equation}\label{eq:Tilt-mu}
\frac{d\mu_{\gamma_1,\gamma_2}(\beta)}{d\mu} := \exp\Big(\gamma_1 \beta-\frac{\gamma_2}{2}\beta^2 -c_\mu(\gamma_1,\gamma_2)\Big), \qquad \mathrm{for} \;\; \beta \in \R.
\end{equation}
Note that $\mu_{\gamma_1,\gamma_2}$ is an exponential family w.r.t.~the base measure $\mu$ and natural sufficient statistic $(\beta,\beta^2)$; we refer to $\mu_{\gamma_1,\gamma_2}$ as a quadratic tilt of the probability measure $\mu$. Fixing $\gamma_2=d>0$, it follows from standard properties of exponential families that the mean of the measure $\mu_{\gamma_1,\gamma_2}$ is $\dot{c}_\mu(\gamma_1,d):=\frac{\partial c_\mu(\gamma_1,d)}{\partial \gamma_1}$, and the function $\dot{c}_\mu(\cdot,d)$   is strictly increasing on $\R$ with range $I_{\mu} := (\inf \mathrm{Supp}(\mu), \sup \mathrm{Supp}(\mu)).$
Thus it has an inverse
$h_{\mu,d} : I_{\mu} \to \mathbb{R}$ which satisfies 
\begin{align}
\dot{c}_{\mu} (h_{\mu,d}(u), d)  = u, \qquad \mbox{for } u \in I_\mu. \label{h:defn} 
\end{align} 
Note that for any $u$ in the mean parameter space $I_\mu$, $h_{\mu,d}(u)$ gives the natural parameter $\gamma_1 \in \R$ of the tilted measure $\mu_{\gamma_1,d}$ such that $\dot{c}_\mu(\gamma_1,d) = u$. Finally, define the function $G_{\mu, d}: I_{\mu} \to \R$ as
    \begin{align}
        G_{\mu, d}(u) := \dkl(\mu_{h_{\mu,d}(u),d} \| \mu_{0,d} )=u \,h_{\mu,d} (u) - c_{\mu}(h_{\mu,d}(u),d) + c_{\mu}(0,d). \label{eq:G_defn}  
    \end{align}
Here the second equality follows from a direct calculation. 
\end{definition}
 It is known that the set of optimizers in the NMF optimization problem~\eqref{eq:Min-KL-Div} always lie in this class of tilts~\citet{mukherjee2022variational},~\citet{yan2020nonlinear}. Note that, for any probability measure $\mu$, $c_\mu(\gamma_1,\gamma_2) < \infty$, for any $(\gamma_1,\gamma_2)\in \R \times (0,\infty)$.
\begin{definition}\label{def:Z}
Let $A\in \R^{p\times p}$ be the off-diagonal part of the matrix $\sigma^{-2}{\mathbf X}^\top{\mathbf X}$ with diagonal entries 0, 
\begin{equation}\label{eq:notation}
{\mathbf w}:=\sigma^{-2}{\mathbf X}^\top{\mathbf y}, \qquad \mbox{and} \qquad d_i:=\sigma^{-2}({\mathbf X}^\top{\mathbf X})_{ii} \quad \mathrm{for} \;\;  i=1,\ldots, p.
\end{equation}
\end{definition}
Armed with this notation, $m_\mu({\mathbf y})$ in~\eqref{eq:marginal} can be expressed as
{\small \begin{align}\label{eq:mpiy}
\notag &\Big(\frac{1}{2\pi\sigma^2}\Big)^{\frac{n}{2}} \exp\Big(-\frac{\|{\mathbf y}\|_2^2}{2\sigma^2}\Big) \int \exp\Big(-\frac{1}{2}{\bm \beta}^\top A{\bm \beta}+{\mathbf w}^\top {\bm \beta}\Big)\exp\Big(-\frac{1}{2}\sum_{i=1}^p d_i \beta_i^2\Big)d\mu^{\otimes p}({\bm \beta})\\
\notag=&\;\Big(\frac{1}{2\pi\sigma^2}\Big)^{\frac{n}{2}} \exp\left(-\frac{\|{\mathbf y}\|_2^2}{2\sigma^2}+ \sum_{i=1}^p c_\mu\Big(0,d_i\Big)\right) \int  \exp\Big(-\frac{1}{2}{\bm \beta}^\top A{\bm \beta}+{\mathbf w}^\top {\bm \beta}\Big)\prod_{i=1}^pd\mu_{0, d_i}(\beta_i)\\
=&\;\Big(\frac{1}{2\pi\sigma^2}\Big)^{\frac{n}{2}} \exp\left(-\frac{\|{\mathbf y}\|_2^2}{2\sigma^2}+ \sum_{i=1}^p c_\mu\Big(0,d_i\Big)\right) Z_p({\bf w},\mu),
\end{align}}
where in the second display we have replaced the base measure from $\mu$ to $\mu_{0, d_i}$ for $i \in [p]$, and we define
\begin{align}\label{eq:zp}
Z_p({\bf v},\mu):= \int_{\R^p}  \exp\Big(-\frac{1}{2}{\bm \beta}^\top A{\bm \beta}+{\mathbf v}^\top {\bm \beta}\Big)\prod_{i=1}^pd\mu_{0, d_i}(\beta_i), \qquad \mbox{for}\;\;  {\bf v} \in \R^p, \; \mu \in \mathcal{P}.
\end{align} 
Note that $Z_p(\mathbf{w},\mu)$ is the only challenging term in the computation of $m_\mu(\mathbf{y})$. To bypass the computational challenge arising in this context, we will utilize the ELBO. To this end, we introduce some notation that will be critical in our subsequent discussion.

\begin{definition}
For any ${\bm \gamma} = (\gamma_1,\ldots, \gamma_p)\in \R^p$ and a probability measure $\mu$ on $\R$, let ${\bf u}^{(\mu)}({\bm \gamma}) \equiv (u^{(\mu)}_{i}(\gamma_i))_{i\in [p]}\in \R^p$ be defined as
\begin{equation}\label{eq:u_i}
u^{(\mu)}_{i}(\gamma_i) :=\dot{c}_\mu\Big(\gamma_i,  d_i\Big)=\E_{\mu_{\gamma_i, d_i}}[\beta_i],
\end{equation}
where the second equality uses Definition \ref{defn:Quad-Tilt} above.
Also, for $ {\bf v} \in \R^p$ (and $\bgamma,\mu$ as above) set
{\small \begin{align}\label{eq:M_p}
\notag M_p(\bm \gamma, {\bf v}, \mu):&= -\frac{1}{2} {\bf u}^{(\mu)}({\bm \gamma})^\top A \, {\bf u}^{(\mu)}({\bm \gamma})+{\bf u}^{(\mu)}({\bm \gamma})^\top {\mathbf v}-\sum_{i=1}^p\dkl(\mu_{\gamma_i,d_i}\|  \mu_{0,d_i})\\
&=-\frac{1}{2} {\bf u}^{(\mu)}({\bm \gamma})^\top A \, {\bf u}^{(\mu)}({\bm \gamma})+{\bf u}^{(\mu)}({\bm \gamma})^\top {\mathbf v}-{\bf u}^{(\mu)}({\bm \gamma})^\top {\bm \gamma}+\sum_{i=1}^pc_\mu\Big(\gamma_i, d_i\Big)-\sum_{i=1}^pc_\mu\Big(0, d_i\Big),
\end{align}}
where the second equality uses~\eqref{eq:G_defn}.
\end{definition}

Let us first motivate the term $M_p(\bm \gamma, {\bf v}, \mu)$. By direct computation (see e.g.,~\citet{wainwright2008graphical}), for any distribution $\nu$ absolutely continuous w.r.t.~$\prod_{i=1}^{p} \mu_{0,d_i}$, i.e., $\nu \ll \prod_{i=1}^{p} \mu_{0,d_i}$, we have 
{\small \begin{align}
    \log Z_p(\mathbf{w},\mu) =  \Big[ \mathbb{E}_{\bar{\nu}}\Big( - \frac{1}{2} \bbeta^{\top} A \bbeta + \mathbf{w}^{\top}\bbeta   \Big) - \mathrm{D}_{\mathrm{KL}}\Big(\bar{\nu} \| \prod_{i=1}^{p} \mu_{0,d_i} \Big) \Big] + \mathrm{D}_{\mathrm{KL}} (\bar{\nu} \| \mathbb{P}_{\mu}(\cdot \mid \mathbf{y})) \label{eq:kl_expression} 
\end{align}}
where $\bar{\nu} = \prod_{i=1}^p \nu_i$ is a product distribution on $\R^p$.
The non-negativity of KL-divergence immediately implies the lower bound 
\begin{align}
    \log Z_p(\mathbf{w}, \mu) \geq \sup_{\bar{\nu} \ll \prod_{i=1}^{p} \mu_{0,d_i}} \Big[ \mathbb{E}_{\bar{\nu}}\Big( - \frac{1}{2} \bbeta^{\top} A \bbeta + \mathbf{w}^{\top}\bbeta   \Big) - \mathrm{D}_{\mathrm{KL}}\Big(\bar{\nu} \| \prod_{i=1}^{p} \mu_{0,d_i} \Big) \Big]. \label{eq:elbo} 
\end{align}
It is well-known (see for e.g.,\citet{wainwright2008graphical},~\citet{mukherjee2022variational}) that any global optimizer of the right hand side above has the form $\bar{\nu} = \prod_{i=1}^{p}\mu_{\gamma_i,d_i}$ for some $\bgamma = (\gamma_1, \ldots, \gamma_p) \in \mathbb{R}^p$. Thus one simply needs to optimize over the scalar parameters $\bgamma \in \mathbb{R}^p$. The quantity $M_p(\bgamma,\mathbf{w}, \mu)$, defined in~\eqref{eq:M_p}, represents the right hand side of the above display evaluated at $\bar{\nu} = \prod_{i=1}^{p} \mu_{\gamma_i,d_i}$; this directly yields the ELBO, which states that 
\begin{align}
    \log Z_p(\mathbf{w},\mu) \geq \sup_{\bgamma \in \mathbb{R}^p} M_p(\bgamma, \mathbf{w}, \mu). \label{eq:elbo_lower} 
\end{align}
Plugging this back into \eqref{eq:mpiy}, one immediately obtains the ELBO for $\log m_{\mu}(\mathbf{y})$:
\begin{align}
    \log m_{\mu}(\mathbf{y}) \geq -\frac{n}{2} \log (2 \pi \sigma^2) - \frac{\|\mathbf{y}\|_2^2 }{2 \sigma^2 } + \sum_{i=1}^{p} c_{\mu} (0,d_i) + \sup_{\bgamma \in \mathbb{R}^p} M_p(\bgamma, \mathbf{w},\mu). \label{eq:elbo_explicit}
\end{align}

Our next result (proved in Section~\ref{pf:lem-approx}) gives a quantitative bound to the approximations in \eqref{eq:elbo_lower}
and \eqref{eq:elbo_explicit}. Utilizing this, it identifies sufficient conditions for which these inequalities are \emph{tight} to leading order.


\begin{theorem}\label{lem:approx}
Suppose that Assumptions~\ref{assump:On-Design} and~\ref{assump:On-P} hold.

\begin{enumerate}
\item[(a)]
With $Z_p(\mathbf{v}, \mu)$ as in \eqref{eq:zp}, there exists a universal constant $\kappa$ such that
$$\sup_{\mathbf{v} \in \R^p} \sup_{\mu \in \mathcal{P}} \; \Big|\log Z_p(\mathbf{v}, \mu) - \sup_{{\bm \gamma}\in \R^p} M_p({\bm \gamma},{\bf v}, \mu)\Big| \le \kappa p^{2/3} {\rm Tr}(A^2)^{1/3}.$$

\item[(b)]
Set, for ${\bm \gamma}, {\bf v} \in \R^p$ and $\mu \in \mathcal{P}$,
\begin{align}
\widetilde{M}_p({\bm \gamma},\mathbf{v},\mu) &:=  M_p({\bm \gamma},\mathbf{v},\mu)+\sum_{i=1}^pc_\mu(0,d_i) \nonumber \\
& \;\,= \; -\frac{1}{2} {\bf u}^{(\mu)}({\bm \gamma})^\top A \, {\bf u}^{(\mu)}({\bm \gamma})+{\bf u}^{(\mu)}({\bm \gamma})^\top {\mathbf v}-{\bf u}^{(\mu)}({\bm \gamma})^\top {\bm \gamma}+\sum_{i=1}^pc_\mu\Big(\gamma_i, d_i\Big), \label{eq:Tilde-M_p}
\end{align}
and consider the following optimization problem:
\begin{equation}\label{eq:mu-MF}
\sup_{\mu\in \mathcal{P}} \sup_{{\bm \gamma}\in \R^p} \widetilde{M}_p({\bm \gamma},{\bf w},\mu),
\end{equation}
where, as before, ${\bf w} = \mathbf{X}^\top \mathbf{y}/\sigma^2$.
Let $\hat{\mu}_{\MF}$ be any sequence of $\varepsilon_p$-approximate maximizers of the above optimization problem, i.e.,
$$\sup_{\mu\in \mathcal{P}}\sup_{{\bm \gamma}\in \R^p}\widetilde{M}_p({\bm \gamma},{\bf w}, {\mu})\le \sup_{{\bm \gamma}\in \R^p}\widetilde{M}_p({\bm \gamma},{\bf w}, \hat{\mu}_{\MF})+\varepsilon_p. $$
If  $\varepsilon_p \le p^{5/6}$ a.s., then there exists constants $c,C$ depending only on $C_1,C_2$, such that  
    {\small \begin{align}
        \mathbb{P}_{\mu^*}\Big( d_{\mathcal{H}}^2(\hat{\mu}_{\mathrm{MF}} \star N(0, C_1^{-1}), \mu^* \star N(0, C_1^{-1})) > c \Big( \frac{1}{p^{1/6}} + \frac{\mathrm{Tr}(A^2)^{1/3}}{p^{1/3}} \Big) \Big) \le C e^{-\frac{p^{2/3}}{C}}. \label{eq:mf_rate}
    \end{align}}
    \end{enumerate}
\end{theorem}
In words, part (a) of the above theorem establishes that the lower bound in \eqref{eq:elbo_lower} is tight up to $O(p^{2/3} \mathrm{Tr}(A^2)^{1/3})$ corrections, uniformly in $\mathbf{v} \in \mathbb{R}^p$ and $\mu \in \mathcal{P}$. We further emphasize that the optimization problem \eqref{eq:mu-MF} is equivalent to optimizing the right hand side in \eqref{eq:elbo_explicit}. Since the NPMLE (see~\eqref{eq:npmle}) is obtained by maximizing the left hand side of~\eqref{eq:elbo_explicit}), one would hope that any approximate optimizer $\hat{\mu}_{\mathrm{MF}}$ of \eqref{eq:mu-MF} would also be consistent for $\mu^*$, given that the NPMLE is consistent, by Theorem \ref{thm:ml_rate}. This is the main takeaway from \eqref{eq:mf_rate}, which also derives a non-asymptotic convergence rate for $\hat{\mu}_{\mathrm{MF}}$.

\begin{remark}[On the assumptions of Theorem~\ref{lem:approx}]\label{rem:np} Note that in Theorem~\ref{lem:approx} we have an extra condition on the design matrix (${\rm Tr}(A^2)=o(p)$), compared to Theorem~\ref{thm:ml_rate}. Upon noting that for $i\neq j$, $A_{ij}$ denotes the inner-product between the $i^{th}$ and $j^{th}$-columns of $\mathbf{X}$, we see that this condition enforces that the correlations among the features are relatively weak. 
Equivalently, this condition ensures that the typical eigenvalues of $A$ are small, and hence the matrix $A$ can be ``approximated'' by a ``low-rank'' matrix $\hat A$ (i.e., rank$(\hat A) = o(p)$). 
In fact, the condition $\mathrm{Tr}(A^2) = o(p)$ cannot be relaxed, as illustrated in the following example. 
Consider the Gaussian design setting, where the entries of $\mathbf{X}$ are i.i.d.~mean-zero Gaussian with variance $1/n$ (this ensures that the operator norm of $\mathbf{X}^{\top}\mathbf{X}$ is $\Theta_P(1)$; see e.g.,~\citet{MR3837109}). A straightforward calculation shows that $\mathrm{Tr}(A^2)=o(p)$ if and only if $p = o(n)$. Our result thus guarantees the accuracy of the NMF approximation in the entire regime $p=o(n)$. In sharp contrast, the NMF approximation is conjectured to fail when $n$ and $p$ are both large and comparable (see Section~\ref{sec:discussion} for an in-depth discussion). Finally, we emphasize that although the NMF approximation fails in the regime $p/n \to \kappa \in (0,1)$, Theorem~\ref{thm:ml_rate} continues to hold, and thus the NPMLE remains consistent.

The condition $\mathrm{Tr}(A^2)=o(p)$ holds trivially if the matrix $\mathbf{X}^{\top} \mathbf{X}$ is diagonal. In this case, the posterior is itself a product measure. One might naturally presume that under the condition $\mathrm{Tr}(A^2)= o(p)$, our results go through as the posterior is ``close" to a product measure. We caution the reader that the actual picture is significantly more complicated --- in fact, even in the special case of $A_{ij} = 1/p$ (i.e., there is small but non-trivial correlations among the features) and $\mathbf{y}=0$, it is well-known that for $\sigma^2$ small enough, the posterior distribution can be approximated as a mixture of two distinct product distributions \citep{MR0503332}. In general, under the assumption $\mathrm{Tr}(A^2)=o(p)$, the posterior can only be approximated as a mixture of product distributions with ``not too many components" \citep{MR4038047}. We emphasize that Theorem~\ref{lem:approx} holds even if the posterior distribution is a mixture, rather than a simple product distribution. In our subsequent discussion (see Theorem \ref{thm:inference}) we identify additional conditions on the design which ensures that the posterior can be approximated by a unique product measure.
\end{remark}

\begin{remark}[Approximating the posterior distribution]\label{rem:kl_minimization}
When combined with \eqref{eq:kl_expression}, part (a) of Theorem \ref{lem:approx}  yields that     $$ \sup_{{\bf y}\in \R^n} \sup_{\mu \in \mathcal{P}} \left\{\inf_{\bar{\nu} = \prod_{i=1}^p \nu_i} D_{\mathrm{KL}}\big( \bar{\nu} \| \P_{\mu}(\cdot \mid {\bf y}) \big)\right\}=O(p^{2/3} \mathrm{Tr}(A^2)^{1/3}),$$
i.e.,~the posterior is (uniformly in $\mu \in \mathcal{P}$) mean field up to $O(p^{2/3} \mathrm{Tr}(A^2)^{1/3})$ corrections. This suggests an algorithmic route to approximating the posterior distribution $\mathbb{P}_{\mu}(\cdot \mid \mathbf{y})$. Given any consistent estimator $\hat{\mu}$ of $\mu$, the  optimizer of the right hand side of~\eqref{eq:Tilde-M_p} can be utilized to identify the optimal quadratic tilts $\hat{\bm \gamma}$ (of $\hat \mu$). Suppose that $\hat{\bm \gamma}$ and $\hat{\mu} \equiv \hat \mu_\MF$ are the global maximizers of \eqref{eq:mu-MF}. Then the product measure
$\prod_{i=1}^p\hat{\mu}_{\hat{\gamma}_i,d_i}$
    is a ``good'' approximation of the true posterior $\P_{\mu}(\cdot \mid {\bf y})$. This intuition is formalized in Theorem \ref{thm:inference}-(b) below.    
\end{remark}

\subsection{Numerical computations}
\label{sec:numerical}

Suppose that the true unknown prior of the $\beta_i$'s is $\mu^* \in \mathcal{P}$. In this subsection we describe the computations involved in estimating the unknown prior $\mu^*$ by (an approximation of) $\hat{\mu}_{\MF}$ (see~\eqref{eq:mu-MF}). In the sequel we assume that $\mathcal{P}$ is the class of all probability distributions on $[-1,1]$, although similar computations can be done for other classes --- e.g., the class of all scale mixture of normals (see~\citet{kim2022flexible}). Even though $\mathcal{P}$ is infinite dimensional, we can readily approximate $\mathcal{P}$ by a class of discrete distributions; this is a common strategy employed in computing NPMLEs in mixture models (see e.g.,~\citet{koenker2014convex, kim2022flexible}). For simplicity let $-1 = a_1 < a_2 < \ldots < a_{k-1} < a_k =1$ be a fine discretization of the interval $[-1,1]$;~\citet{koenker2014convex} recommends taking $k \approx 300$ for typical problems.  
For any ${\bf p} = (p_1,\ldots,p_k) \in \Delta_k$, the probability simplex in $\R^k$, define the probability measure $\nu({\bf p})$ supported on $\{a_1, \ldots, a_k\}$ as $$\nu({\bf p}):=\sum_{r=1}^k p_r\delta_{a_r}.$$ Then, the natural (discrete) analogue to the optimization problem~\eqref{eq:mu-MF} is to solve:
\begin{equation}\label{eq:Opt-M_p}
\max\limits_{ {\bf p} \in \Delta_k} \;\;\sup_{{\bm \gamma}\in \R^p}  \widetilde{M}_p({\bm \gamma}, \mathbf{w},\nu({\bf p})),
\end{equation}
where $\widetilde{M}_p({\bm \gamma}, \mathbf{w} ,\nu({\bf p}))$ is defined in~\eqref{eq:Tilde-M_p}. Note that~\eqref{eq:Opt-M_p} is a finite dimensional smooth optimization problem, with variables ${\bf p} \in \Delta_k$  and ${\bm \gamma}\in \R^p$, which can be solved via a variety of first order optimization methods. In our numerical studies, we used an iterative approach to solving~\eqref{eq:Opt-M_p}, where we fix ${\bm \gamma}$ and optimize over ${\bf p} \in \Delta_k$, and then fix ${\bf p}$ and optimize over ${\bm \gamma}$. 

For a fixed ${\bm \gamma}$ to optimize over ${\bf p}$, we used projected gradient descent (see e.g.,~\citet[Chapter 8]{Beck-2017}) as ${\bf p}$ is constrained to lie in the unit simplex $\Delta_k$.  On the other hand, for a fixed ${\bf p} \in \Delta_k$ to optimize over ${\bm \gamma}$, we used the Broyden–Fletcher–Goldfarb–Shanno (BFGS) algorithm  --- an iterative method for solving nonlinear optimization problems (see e.g.,~\citet[Chapter 6]{Nocedal-Wright-2006}) --- implemented via the \texttt{optim} function in the \texttt{stats} package in the \texttt{R} programming language. To initialize the iterations, we start from an estimate $\hat{\bbeta}$ of $\bbeta$ --- we advocate using the ordinary least squares estimator (when $p$ is relatively small compared to $n$) or the  Lasso with a small penalty parameter (when $n$ and $p$ are comparable). Given $\hat{\bbeta}$, we set $\bgamma = \mathbf{w} - A \hat{\bbeta}$ (see Appendix~\ref{sec:Opt-Algo} for an intuition behind this initialization). The initialization for $\mathbf{p}$ is easier; we just set $\mathbf{p} = (1/k,\ldots, 1/k) \in \R^k$. In Appendix~\ref{sec:Opt-Algo} we give the explicit expressions for the gradient of the objective $\widetilde{M}_p({\bm \gamma}, \mathbf{w} ,\nu({\bf p}))$ w.r.t.~the parameters ${\bm \gamma}$ and ${\bf p}$. We note that our BFGS algorithm is generally guaranteed to converge to a local optimum rather than a global maximum. However, the algorithm exhibits excellent performance in practice. It would be interesting to know if the BFGS algorithm converges to a global optimizer under some conditions on the design; we leave this for future work.  

\section{Posterior inference}
\label{sec:post_inf}

In this section we study the consistency of the approximate posterior of $\bbeta$ given ${\bf y}$ obtained from the NMF approximation via the solution to the optimization problem~\eqref{eq:mu-MF}.

From an implementational perspective this posterior approximation can be computed by solving the optimization problem~\eqref{eq:Opt-M_p}. Suppose $(\hat {\bf p},\hat {\bm \gamma})$ are the unique global maximizers of this optimization problem. Then the posterior distribution can be approximated via
\begin{align} 
\label{eq:posterior_intuition}
\mathbb{P}_{\mu}(\cdot \mid {\bf y})\approx \prod_{i=1}^p \nu(\hat {\bf p})_{\hat \gamma_i,d_i} = \prod_{i=1}^p \left[\sum_{r=1}^k \hat p_r \exp\Big(a_r {\hat \gamma}_i-\frac{a_r^2 d_i}{2}-c_{\nu({\hat{\bf p}})}\Big(\hat \gamma_i, d_i\Big)\Big)\delta_{a_r} \right].
\end{align}
The above approximation will be rigorized in Theorem~\ref{thm:inference} below.

 For the subsequent discussion, it will be helpful to reparametrize the function $M_p({\bm \gamma}, {\bf v}, \mu)$ introduced in~\eqref{eq:M_p} 
in terms of the mean vector ${\bf u} \in \R^p$ instead of the natural parameter $\bm \gamma\in \R^p$. With ${\bf d} := (d_1,\ldots, d_p) \in (0,\infty)^p$ as in~\eqref{eq:notation}, 
 for ${\bf u} := (u_1,\ldots, u_p), {\bf v} \in \R^p$, define
\begin{align}
\notag\M_p( \mathbf{u},{\bf v}, \mu):&=-\frac{1}{2} {\mathbf{u}}^\top A \, {\mathbf{u}}+{\bf u}^\top {\mathbf v}-\sum_{i=1}^p\dkl(\mu_{h_{\mu,d_i}(u_i),d_i},\mu_{0,d_i}) \\
&=-\frac{1}{2} {\mathbf{u}}^\top A \, {\mathbf{u}}+{\bf u}^\top {\mathbf v}-{\bf u}^\top h_{\mu, \mathbf{d}}({\bf u})+\sum_{i=1}^p c_\mu\Big(h_{\mu, d_i}(u_i), d_i\Big) -\sum_{i=1}^pc_\mu\Big(0, d_i\Big), \label{eq:M_mean_param}
\end{align} 
where by $h_{\mu,\mathbf{d}}({\bf u})$ we mean the vector in $\R^p$ with the $i$-th coordinate $h_{\mu, d_i}(u_i)$. We note the analogy of the above definition with that of $M_p(\cdot, \cdot, \mu)$ (see~\eqref{eq:M_p}) which is parametrized in terms of the natural parameter $\bm \gamma\in \R^p$.
The notion of a well-separated optimum of $\M_p(\cdot, {\bf w}, \mu)$, introduced next,  will be crucial in the sequel. 
\begin{definition}[Well-separated optimizer]
\label{defn:separation_property}
Fix a probability measure $\mu^*$ on $\R$, and let ${\bf w} := \mathbf{X}^\top \mathbf{y}/\sigma^2$ as in \eqref{eq:notation}.  We say that $\M_p(\cdot, {\bf w}, \mu^*)$ has a well-separated optimizer if there exists a sequence of random vectors $\mathbf{u}^* \in I_{\mu^*}^p$ such that for any $\delta>0$, there exists $\varepsilon>0$ such that 
 \begin{align}
\mathbb{P}_{\mu^*}\Big( \sup_{\| \mathbf{u} -{\mathbf{u}}^*\|_2^2 > p \delta} \M_p(\mathbf{u} , {\bf w}, \mu^*) < \sup_{\mathbf{u} \in I_{\mu^*}^p} \M_p(\mathbf{u}, {\bf w}, \mu^*) - p \varepsilon \Big) \to 1 \qquad \mbox{as } p \to \infty. \nonumber 
 \end{align} 
\end{definition} 
Note that we keep the $p$-dependence of $\mathbf{u}^*$ implicit in our definition. Further, note that $\mathbf{u}^*$ is random, as $\mathbf{w}$ is random.  In \citet[Lemma 5]{mukherjee2022variational}, the authors discuss some easily verifiable conditions for the existence of a well-separated optimizer. We present an analogue here to keep the discussion self-contained; we sketch a proof in Section~\ref{sec:thm_inf}. 
\begin{lemma}
\label{lemma:separation_sufficient} 
The function $\mathcal{M}_p(\cdot, \mathbf{w},\mu^*)$ has a well-separated optimizer if either (i) $\| A \|_2 \le 1-\eta$ for some $\eta>0$ free of $p$, or (ii)  the prior $\mu^*$ has a density proportional to $\exp(-V(\cdot))$ on $[-1,1]$, with $V(\cdot)$ being an even function such that $V(\cdot)$ is increasing on $[0,1]$ and $V'(\cdot)$ is increasing on $[0,1)$. 
\end{lemma}

\begin{remark}[Ubiquity of well-separated optimizers]
Definition~\ref{defn:separation_property} might appear somewhat arbitrary at first glance. In particular, one might naturally wonder if this condition is satisfied for a broad class of priors $\mu^*$ and design matrices $A$. Lemma~\ref{lemma:separation_sufficient} provides some simple sufficient conditions for the existence of well-separated optimizers. We note that part (i) of Lemma~\ref{lemma:separation_sufficient} applies to any prior supported on $[-1,1]$, and ensures the existence of well-separated optimizers as long as $\|A\|_2$ is sufficiently small. Part (ii) of Lemma~\ref{lemma:separation_sufficient} requires additional conditions on the prior $\mu^*$, but does not require any condition on the matrix $A$. 
\end{remark}

\begin{remark}[Natural versus mean parametrization] 
One might wonder why we used the natural parameter $\bgamma$ (and the function $\widetilde{M}_p(\bgamma, \mathbf{w},\mu)$) for numerical computations in section \ref{sec:numerical}, and the mean parametrized function $\mathcal{M}_p(\mathbf{u},\mathbf{w}, \mu)$ in the current section. Our choice is governed by the following consideration: the gradient of $\widetilde{M}_p(\bgamma, \mathbf{w},\mu)$ involves the function $c_\mu(\cdot,\cdot)$ and its derivative, which are easier to compute. On the other hand, the derivative of $\mathcal{M}_p(\mathbf{u},\mathbf{w}, \mu)$ involves the inverse function $h_\mu(\cdot,\cdot)$, which is considerably harder to evaluate. However, on the theoretical side, we expect the condition of well separated optimizers to be violated in most settings for the natural parametrized function $\widetilde{M}_p(\bgamma, \mathbf{w},\mu)$, whereas sufficient conditions for existence of well separated optimizers for the mean parametrized function $\mathcal{M}_p(\mathbf{u},\mathbf{w}, \mu)$ are known in the literature (as indicated above).   
\end{remark}

To describe the structure of the posterior distribution we will need the notion of Wasserstein distance, introduced below (see e.g. \citet{MR2459454}). 

\begin{definition}[Wasserstein distance]
Let $P$, $Q$ be two distributions on $\mathbb{R}^k$ (for some integer $k$) with finite first moments. Define the $\ell$-Wasserstein distance between $P$ and $Q$ as 
\begin{align}
    d_{W_\ell}(P, Q) :=  \inf_{\Pi} \mathbb{E}_{(\mathbf{X},\mathbf{Y}) \sim \Pi}[\|\mathbf{X} - \mathbf{Y}\|_\ell],\nonumber 
\end{align}
where the infimum ranges over all couplings $\Pi$ with marginals $P$ and $Q$. Here, for $\mathbf{a} = (a_1,\ldots, a_k),$ $\|\mathbf{a}\|_\ell := \Big(\sum_{i=1}^k |a_i|^\ell\Big)^{\frac{1}{\ell}}$. 
\end{definition}

We first provide some informal overview for our next result. Fix a true (unknown) prior $\mu^*$. 
Given $\widetilde{\mathbf{u}} \equiv (\widetilde{u}_1, \ldots, \widetilde{u}_p):= \mathbb{E}_{\mu^*}[\bbeta \mid \mathbf{y}]$, one can compute the tilts $\widetilde{\tau}_i$ such that $\mathbb{E}_{\mu^*_{\widetilde{\tau}_i,d_i}}[\beta_i] = \widetilde{u}_i$, and approximate the true posterior of $\bbeta$ given  $\mathbf{y}$ by $\prod_{i=1}^{p} \mu^*_{\widetilde{\tau}_i,d_i}$. However, this is not immediately feasible --- the true prior $\mu^*$ is unknown, and thus $\widetilde{\mathbf{u}}$ is unavailable to the statistician. Given the consistency result for $\hat{\mu}_{\mathrm{MF}}$ derived in Theorem~\ref{lem:approx}, the natural idea at this point is to use the proxy $\hat{\mu}_{\mathrm{MF}} $ in place of $\mu^*$. In addition, assuming $\mathcal{M}_p(\cdot, \mathbf{w},\mu^*)$ has a well-separated optimizer $\mathbf{u}^*$, one can show that $\mathbf{u}^* \approx \widetilde{\mathbf{u}}$. Thus to estimate $\widetilde{\mathbf{u}}$, the natural proxy is a (near) maximizer $\hat{\mathbf{u}}_{\mathrm{MF}}$ for $\mathcal{M}_p(\cdot, \mathbf{w},\hat{\mu}_{\mathrm{MF}})$ (observe from~\eqref{eq:Tilde-M_p} that given ${\hat \mu}_\MF$, ${M}_p(\cdot, \mathbf{w}, \hat{\mu}_{\mathrm{MF}})$ and $\widetilde{{M}}_p (\cdot, \mathbf{w}, \hat{\mu}_{\mathrm{MF}})$ have the same optimizer). To derive a completely data-driven approximation of the oracle posterior, one can now follow the recipe outlined above, with $\hat{\mu}_{\mathrm{MF}}$ and $\hat{\mathbf{u}}_\mathrm{MF}$ in place of $\mu^*$ and $\mathbf{u}^*$ respectively. 
Our next result establishes the validity of this intuitive data-driven procedure. To the best of our knowledge, this is the first result regarding the approximation of the oracle posterior using an EB method in the context of high-dimensional linear regression.

\begin{theorem}
\label{thm:inference} 
\begin{itemize}
\item[(i)] Let $\mu^*$ be a non-degenerate probability distribution on $[-1,1]$ (i.e., $\mu^*$ is not a point mass). Suppose that Assumptions~\ref{assump:On-Design}  and~\ref{assump:On-P} hold, and $\mathrm{Tr}(A^2)=o(p)$.  Let $\{\mu^{(p)} : p \geq 1\} \subset \mathcal{P}$ be a sequence of probability distributions converging weakly to $\mu^*$, as $p \to \infty$. 

Assume that $\M_p(\cdot, {\bf w}, \mu^*)$ has a well-separated optimizer $\mathbf{u}^*$. Let ${\mathbf{u}}^{(p)} := ({u}_1^{(p)}, \ldots, {u}_p^{(p)}) \in [-1,1]^p$ be any global optimizer of $\M_p(\cdot, {\bf w}, \mu^{(p)})$ on $I_{\mu^{(p)}}^p$, and let $$\widetilde{\bf u}:=\E_{\P_{\mu^*}}[\bbeta \mid \mathbf{y}]\in [-1,1]^p,\qquad \widetilde{\bm \tau} = (\widetilde{\tau}_1,\ldots, \widetilde{\tau}_p) :=h_{\mu^*,{\bf d}}(\widetilde{\bf u})\in \R^p.$$ Then the following conclusions hold, as $p\to\infty$:
\begin{enumerate}
\item[(i)] $$\frac{1}{p} \| {\mathbf{u}}^* - {\mathbf{u}}^{(p)} \|^2_2 \stackrel{\P_{\mu^*}}{\to} 0\qquad \mbox{and}\qquad    \frac{1}{p} \| {\mathbf{u}}^*-\widetilde{\mathbf{u}} \|^2_2 \stackrel{\P_{\mu^*}}{\to} 0;$$ 

\item[(ii)]   $\frac{1}{p} \dkl \left(\mathbb{P}_{\mu^*}(\cdot \mid \mathbf{y} )\| \prod_{i=1}^{p} \mu^*_{\widetilde{\tau}_i,d_i} \right) \stackrel{\mathbb{P}_{\mu^*}}{\to} 0;$

\item[(iii)]  $$ \frac{1}{p} \, d_{W_1} \left(\mathbb{P}_{\mu^*}(\cdot \mid \mathbf{y}) , \prod_{i=1}^p \mu^*_{\widetilde{\tau}_i,d_i} \right) \stackrel{\mathbb{P}_{\mu^*}}{\to} 0 \qquad \mbox{and}\qquad  \frac{1}{p}\, d_{W_1}\left(\mathbb{P}_{\mu^*}(\cdot \mid  \mathbf{y}) , \prod_{i=1}^p \mu^{(p)}_{\tau_i^{(p)},d_i}\right) \stackrel{\mathbb{P}_{\mu^*}}{\to} 0,$$ where \begin{align}\label{eq:tau}
\tau_i^{(p)} := h_{\mu^{(p)},d_i}(u_i^{(p)}), \qquad \mbox{for } \; i \in [p].
\end{align}

\end{enumerate} 
\end{itemize}
\end{theorem}

\begin{remark}[Estimation of posterior variance]\label{rem:Post-Var}
Theorem~\ref{thm:inference} part $(iii)$ shows that we are able to recover the one dimensional marginal distributions of the posterior distribution in Wasserstein distance, at least in an averaged sense. In particular, this means that we are able to recover both the means and variances of the posterior distribution (but not the covariances, of course). This is somewhat in contrast with the existing literature, which has reported underestimation of the variance using mean field variational inference (see for e.g.,~\citet{MR2894235},~\citet{MR4011769},~\citet{qiu2023sub} and the references therein). We stress that in all these examples, our mean field assumption ${\rm Tr}(A^2)=o(p)$ is violated, which is the main reason for this discrepancy. As an example, \citet[Section 1.3.3]{MR2894235} works with the auto-regressive model which corresponds to choosing $A$ as a tridiagonal matrix with $0$ on the diagonal  and $\lambda$ on the non-zero off-diagonals, and reports underestimation of variances.  In this case ${\rm Tr}(A^2)=2(p-1)\lambda$ which is not $o(p)$ for $\lambda$ fixed. Similarly, \citet[Section 3]{qiu2023sub} points out sub-optimality of mean field when working with a Gaussian design matrix under the assumption $n\propto p$. Again in this case ${\rm Tr}(A^2)\propto p$ (see Remark \ref{rem:np} for details). As a final example, \citet[Section 3.3]{MR4011769} points out underestimation of variances using mean field methods. However, in their setting $p$ is fixed, so the condition ${\rm Tr}(A^2)=o(p)$ is not meaningful. Also, they work with a design matrix which has a growing spectral norm, whereas we work under a bounded spectral norm assumption (see Remark \ref{rem:norm} for a discussion on the scaling for the spectral norm). To summarize, we expect mean field based inference to correctly recover the variances if the design matrix $\mathbf{X}$ has a bounded spectral norm, and the mean field condition ${\rm Tr}(A^2)=o(p)$ holds. Further, we expect these conditions to be ``almost" tight.
\end{remark}

Theorem~\ref{thm:inference} is proved in Section~\ref{sec:thm_inf}. We first collect some observations regarding Theorem \ref{thm:inference}. We state the theorem for any sequence of probability distributions $\mu^{(p)}$ converging weakly to $\mu^*$. The result carries over immediately if $\mu^{(p)} \equiv \hat{\mu}$ is a sequence of consistent estimators for $\mu^*$. We will apply the result to the consistent estimator $\hat{\mu}_{\mathrm{MF}}$.

Next, we elaborate on conclusions obtained above. Part (i) concludes that the optimizer $\mathbf{u}^{(p)}$ is approximately close to $\mathbf{u}^*$ --- this is intuitive, given that $\mu^{(p)}$ is converging weakly to $\mu^*$. However, we emphasize that the functionals $\mathcal{M}_p(\cdot, \mathbf{w},\mu^*)$ and $\mathcal{M}_p(\cdot, \mathbf{w},\mu^{(p)})$ are \emph{not} pointwise close. Our proof crucially exploits the accuracy of the NMF approximation, the stability of $ Z_p(\mathbf{w},\mu)$ (see~\eqref{eq:zp}) w.r.t.~$\mu$, and the existence of a well-separated optimizer $\mathbf{u}^*$. In addition, part (i) establishes that $\mathbf{u}^*$ is close to the vector of 1-dimensional means $\widetilde{\mathbf{u}}$. Part (ii) provides an approximation of the posterior $\mathbb{P}_{\mu^*}(\cdot \mid  \mathbf{y})$ in KL-divergence in terms of quadratic tilts of the true prior $\mu^*$ --- note that this is distinct from the NMF approximation, due to the reverse nature of the KL divergence employed. This intermediate result is crucial for establishing part (iii), which provides the desired approximation to the posterior distribution in terms of observable quantities. Finally, we note that part (iii) presents two alternative approximations to the true posterior --- the first is an oracle product approximation, and involves $\mu^*$ and the tilt parameter $\widetilde{\boldsymbol{\tau}}$. We collect this version as it provides significant  insight on the structure of the posterior distribution. The second version, involving $\mu^{(p)}$ and $\boldsymbol{\tau}^{(p)} = (\tau_1^{(p)},\ldots, \tau_p^{(p)})$ is actually one that is useful from a practical perspective. Note that this data-driven product approximation, formulated in terms of $\mu^{(p)}$ and the tilt parameters $\boldsymbol{\tau}^{(p)}$, formalizes the intuition introduced in \eqref{eq:posterior_intuition}.

In the following corollary we investigate  statistical applications of Theorem~\ref{thm:inference}.
This will facilitate downstream Bayesian inference, upon combination with the consistency of $\hat{\mu}_{\mathrm{MF}}$ established in Theorem \ref{lem:approx}. The proof of the corollary is given in Section~\ref{sec:thm_inf}.

\begin{corollary}
\label{cor:takeaways} 
Suppose we are in the setting of Theorem \ref{thm:inference}.
Then we have the following conclusions.
\begin{itemize}
\item[(i)] 
\label{item2_inference} 
Let $\{f_{p,i}: 1\leq i \leq p\}$ be 1-Lipschitz functions on $[-1,1]^{L}$, for any positive integer $L$, i.e., $$|f_{p,i}(s_1,\ldots, s_L) - f_{p,i}(t_1,\ldots, t_L)| \le \sum_{i=1}^L |s_i - t_i|, \qquad \mbox{and} \qquad \|f_{p,i} \|_{\infty} \leq 1.$$ 
For any $\varepsilon>0$, as $p \to \infty$,  we have
\begin{align}
    \mathbb{P}_{\mu^*}\bigg( \Big|  \frac{1}{p} \sum_{i=1}^{p} f_{p,i}(\beta_i^{(1)}, \cdots, \beta_i^{(L)}) - \frac{1}{p} \sum_{i=1}^p \mathbb{E}_{\mu^{(p)}_{\tau_i^{(p)},d_i}}[f_{p,i}(\zeta^{(1)}_i,\cdots, \zeta^{(L)}_i)] \Big|  > \varepsilon  \;\Big | \;  \mathbf{y} \bigg)  \stackrel{\mathbb{P}_{\mu^*}}{\to} 0. \nonumber 
\end{align}
In the display above, $\bbeta^{(1)}, \cdots, \bbeta^{(L)}$ are i.i.d.~samples from the posterior $\mathbb{P}_{\mu^*}(\cdot \mid \mathbf{y})$ and $\zeta^{(1)}_i,\cdots, \zeta^{(L)}_i$ are i.i.d.~samples from $\mu^{(p)}_{\tau_i^{(p)},d_i}$.

\item[(ii)] For $i\in [p]$ and $\varepsilon>0$, let $$\mathcal{I}_i^{\varepsilon}:= (q_i^{(\alpha/2)} - \varepsilon, q_i^{(1-\alpha/2)} + \varepsilon ),$$ where $q_i^{(\alpha/2)}$ and $q_i^{(1-\alpha/2)}$ are the $\alpha/2$ and $1-\alpha/2$ quantiles of $\mu^{(p)}_{\tau_i^{(p)},d_i}$. 
Then we have
\begin{align}
\mathbb{P}_{\mu^*} \bigg( \frac{1}{p} \sum_{i=1}^{p} \mathbf{1}(\beta_i \in \mathcal{I}_i^{\varepsilon}) \geq (1-\alpha - \varepsilon) \; \Big |\;\mathbf{y}\bigg) \stackrel{\mathbb{P}_{\mu^*}}{\to} 1. \nonumber 
\end{align} 

\item[(iii)] For $i\in [p]$, let $\hat{\beta}_i := \mathbb{E}_{\mu^{(p)}_{\tau_i^{(p)},d_i}}[\beta_i]$ be an estimate of the true posterior mean $\E_{\P_{\mu^*}} [\beta_i \mid {\bm y}]$.  Then, $\hat{\bbeta} := (\hat{\beta}_1, \cdots, \hat{\beta}_p)$ is asymptotically Bayes optimal under the $L^2$ loss, i.e., $$\frac{1}{p}\E_{\P_{\mu^*}}  \Big[\|{\bm \beta} - \E_{\P_{\mu^*}} [{\bm \beta} \mid {\bf y}] \|_2^2 \Big] - \frac{1}{p}\E_{\P_{\mu^*}}  \Big[\|{\bm \beta} - \hat{\bbeta}\|_2^2 \Big] \to 0.$$
\end{itemize} 
\end{corollary} 

\begin{remark}[On the well-separated optimizer assumption in Theorem~\ref{thm:inference}]
\label{rem:well_separated_req}
The reader might naturally wonder about the necessity of the well-separated optimizer assumption in Theorem \ref{thm:inference}. As established in Theorem \ref{thm:inference}, in this setting, the true posterior can be approximated (in KL and Wasserstein sense) by a product measure. In fact, it is well-understood that even when the mean-field approximation is correct (i.e., Theorem~\ref{lem:approx} part (a) holds), if the well-separated optimizer condition is violated, the true posterior can actually behave like a \emph{mixture} of product distributions, rather than a simple product distribution. The most well-understood example in this regard is the Curie-Weiss Ising model for Markov Random Fields. This model is studied extensively in statistical physics as a model for ferromagnetism and in statistics/machine learning as a simple model for dependent data. It is well understood that at low temperatures, the associated Gibbs distribution can be approximated as a mixture of two product distributions. We refer the interested reader to~\citet{MR3752129} and~\citet{MR0503332} for an in-depth discussion of related models. It would be interesting to explore statistical methodology in settings when the well-separated optimizer condition does not hold, and instead the posterior is approximated by a mixture of product measures. 
We hope to explore this further in follow up investigations. 
\end{remark}
\section{Discussion}
\label{sec:discussion}
In the following we discuss some limitations of the current approach, and collect some directions for future enquiry. We also discuss some additional indirect connections with current research. 
\begin{itemize}

\item[(1)] \textbf{Extensions}: 
One of the main limitations of our current analysis is the bounded prior assumption (Assumption \ref{assump:On-P}). Unbounded priors are of natural interest in this problem, e.g.,~\citet{kim2022flexible} propose to use scale mixture of Gaussians for the parameter space $\mathcal{P}$. The use of heavy-tailed spike and slab priors are also ubiquitous in modern Bayesian variable selection \citep{george1997approaches}. We believe that with significant technical work, it should be possible to extend our results to cover unbounded priors with \emph{very} light tails (e.g.,~faster than sub-Gaussian). However, extending these ideas to cover truly unbounded priors (e.g., with only bounded second moment) should require deep new ideas, and is significantly beyond the scope of this work.  

We establish the asymptotic consistency of the NPMLE ${\hat \mu}_\ML$ and its mean field surrogate ${\hat \mu}_\MF$ in Theorems~\ref{thm:ml_rate} and~\ref{lem:approx}, respectively. One potential criticism of these results is that the convergence rates are probably sub-optimal. We emphasize that this is (to the best of our knowledge) the very first non-asymptotic rates for the NPMLE and the mean-field surrogate in a high-dimensional regression setting. A characterization of the optimal convergence rates for these estimators is the next natural step---we expect that such a characterization will likely require the development of several new ideas, and leave this as an open question.

Theorem~\ref{thm:inference}-(i) implies that under the existence of a well-separated optimizer $\mathbf{u}^*$, we have $\| \mathbf{u}^{(p)} - \widetilde{\mathbf{u}}\|^2 = o(p)$ under $\mathbb{P}_{\mu^*}$. In words, this implies the consistency of the empirical Bayes estimator for the means. For the Gaussian sequence model, \citet{Jiang-Zhang-2009} establish nearly parametric rates of convergence for the empirical Bayes estimator. This proof technique is not directly applicable in our setting due to the dependence between the marginals in the regression setting. Establishing analogous non-asymptotic convergence rates for the mean estimator $\mathbf{u}^{(p)}$ is an interesting open question. We believe this will require new ideas significantly beyond the present paper, and leave this for future work.

Finally, we have assumed throughout the paper that the noise variance $\sigma^2$ is known to the statistician. If $\sigma^2$ is instead unknown, the prior $\mu$ might not be identifiable. This arises even in the context of the Gaussian sequence model (where $n=p$ and $\mathbf{X} = I_n$); in this case if $\mathcal{P}$ is the class of all normal distributions with mean 0 and variance $A >0$, the $y_i$'s are i.i.d.~$N(0, A + \sigma^2)$, and it is not possible to estimate $A$ and $\sigma^2$ separately based on the $y_i$'s.

    \item[(2)] \textbf{CAVI}: The EB methodology introduced in this paper relies crucially on the numerical optimization of the function $\widetilde{M}_p({\bm \gamma}, \mathbf{w} ,\nu({\bf p}))$ (recall e.g.,~\eqref{eq:Opt-M_p}). In our implementation (described in Section \ref{sec:numerical}), we adopt an alternating update strategy, where we iteratively update $\bgamma$ (with $\mathbf{p}$ held constant) and then update $\mathbf{p}$ (with $\bgamma$ held constant). We note that \emph{Coordinate Ascent Variational Inference} (CAVI) provides a natural alternative to optimize  $\widetilde{M}_p(\bgamma, \mathbf{w}, \nu(\mathbf{p}))$ in $\bgamma$ (for a fixed $\mathbf{p}$) (see~\citet{Blei-Et-Al-VI} and the references therein). CAVI is an intuitive iterative procedure --- it cycles over the coordinates of $\bgamma$, and optimizes the function $\widetilde{M}_p(\bgamma, \mathbf{w}, \nu(\mathbf{p}))$ in the chosen coordinate of $\bgamma$ when the others are held fixed. The single coordinate updates are usually very fast, which leads to the popularity of CAVI in many applications. Our BFGS based approach and CAVI are both generally guaranteed to converge to local optima, rather than to a global optimizer of $\widetilde{M}_p(\cdot, \mathbf{w} ,\nu({\bf p}))$. We expect that given a reasonable initialization, the two approaches should have comparable performance in our setting. We believe that an in-depth study and comparison of diverse optimization algorithms (including the aforementioned ones) for the optimization problem \eqref{eq:Opt-M_p} is an important direction for future research. 

    \item[(3)] \textbf{Settings beyond NMF}: The ELBO \eqref{eq:elbo} is crucial in our analysis. In fact, we focus exclusively on settings where the lower bound in~\eqref{eq:elbo} is tight up to $o(p)$ errors. This prompts two natural questions: (a) when is the lower bound in~\eqref{eq:elbo} \emph{not} tight (i.e.,~there exists a $\Theta(p)$ gap between $\log Z_p(\cdot,\cdot)$ and the NMF approximation)? And (b) what can one do in such settings?

    It seems challenging to derive a general answer to the above two questions. However, one can provide specific examples where \eqref{eq:elbo} is not tight (thus answering question (a) above) and in some of these instances, one can hope to go beyond the NMF based approximation scheme outlined in this paper (thus providing a partial answer to question (b) above). A prominent example in this regard is a linear model with i.i.d.~Gaussian design (every entry of the design has mean zero and variance $1/n$) under a proportional asymptotic regime, i.e.,~the number of features $p$ and data points $n$ both diverge to infinity such that $p/n \to \kappa \in (0,\infty)$. Deep but non-rigorous heuristics from spin glass theory \citep{MR1026102,MR2518205} suggest that the lower bound \eqref{eq:elbo} is not tight in this regime. Instead, the Thouless-Anderson-Palmer (TAP) approximation --- an alternative mean field approximation derived from the Bethe approximation \citep{wainwright2008graphical} --- is conjectured to yield a tight estimate to $\log Z_p$ (see e.g.,~\citet{krzakala2014variational} for a precise statement). 
    
    The third author, in joint work with Jiaze Qiu,  studied the linear model with i.i.d.~Gaussian design in a proportional asymptotic regime and established the TAP approximation to $\log Z_p$ under a uniform spherical prior (under additional \emph{low} Signal-to-Noise (SNR) assumptions); see~\citet{MR4612652}. However, this approximation relies crucially on the symmetry of the uniform spherical prior, and exploits estimates from random matrix theory which are no longer applicable under the product prior setting studied in this paper. The TAP approximation was recently extended to the product prior setting in \citet{celentano2023mean}.
    Given the accuracy of this approximation, another natural direction for future enquiry concerns the performance of EB procedures based on the aforementioned TAP approximation.      

    We note that the accuracy of the TAP approximation to the marginal density has been rigorously established recently for the low rank matrix recovery problem \citet{MR4203332},~\citet{celentano2022sudakov},~\citet{MR4600991}. Building on this approximation and related ideas, \citet{zhong2022empirical} develops EB methods to estimate the prior on the spike entries in the low rank matrix estimation problem.

    \item[(4)] \textbf{The condition $n \geq p$}: Our theoretical results are derived under Assumption \ref{assump:On-Design}. In particular, as discussed in Remark~\ref{eg:eigen}, this assumption necessitates that $n \geq p$. In light of this requirement (and in view of Example~\ref{eg:eigen}), the reader might end up concluding that the EB approach outlined in this paper is of limited utility in modern big data applications, where the number of features $p$ can often be much larger than the sample size $n$. However, we believe that such a conclusion can be overly pessimistic. In fact, if one imposes strong additional structure on the design $\mathbf{X}$, e.g.,~the rows are i.i.d.~from a Gaussian distribution, consistent estimation of $\mu$ might be statistically possible, even if $p \geq n$. A precise characterization of the design properties which enable consistent prior recovery even if $p \geq n$ is beyond our current paper. 
    We leave this for future work.  
\end{itemize}
\noindent
\textbf{Acknowledgments:} The authors thank Zhou Fan, Pragya Sur, and Yihong Wu for helpful discussions at various stages of this work. They also thank the anonymous referees for constructive comments that substantially improved the manuscript, and Samprit Banerjee for assistance with the implementation of the NMF-EB approach.

\section{Proof of Theorem \ref{thm:ml_rate} and Theorem \ref{lem:approx}}\label{sec:Proofs}

\subsection{Reduction to sequence model}\label{sec:Reduction}

In the following lemma we first reduce the problem of inference for~\eqref{eq:lin_reg} and~\eqref{eq:iid-Beta} by sufficiency.
\begin{definition}[Correlated sequence model]
Consider the model
\begin{align}\label{eq:reg_z}
{\bf z}={\bm \beta}+\Sigma^{1/2}{\bm \varepsilon},
\end{align}
where $${\bf z}:=({\mathbf X}^\top {\mathbf X})^{-1}{\mathbf X}^\top {\bf y}\in \R^p,\qquad \Sigma:=\sigma^2 ({\mathbf X}^\top{\mathbf X})^{-1}\in \R^{p\times p},\qquad{\bm  \varepsilon}\sim N({\bf 0}, {\bf I}_p).$$
Note that $\mathbf{z}$ is well-defined, as ${\mathbf X}$ has full column rank (which follows from Assumption~\ref{assump:On-Design}; see~\eqref{eq:eigen}), and hence ${\mathbf X}^\top {\mathbf X}$ is invertible. Then ${\bf z}\sim N({\bm \beta}, \Sigma)$. Also, Assumption~\ref{assump:On-Design} (i.e.~\eqref{eq:eigen}) translates to
\begin{align}\label{eq:eigen2}
\frac{1}{C_2}\le \lambda_{\rm min}(\Sigma)\le  \lambda_{\rm max}(\Sigma)\le \frac{1}{C_1}.
\end{align}
Let $\widetilde{m}_\mu(\cdot)$ denote the marginal likelihood of ${\bf z}$, given by
\begin{equation}\label{eq:marg-z}
\widetilde{m}_\mu({\bf z}) :=(2\pi)^{-p/2}|\Sigma|^{-1/2}\int \exp\Big[-\frac{1}{2}({\bf z}-{\bm \theta})^\top \Sigma^{-1}({\bf z}-{\bm \theta})\Big]d\mu^{\otimes p}(\bm\theta).
\end{equation}

\end{definition}
The following proposition shows that the NPMLE of the above model is the same as the NPMLE of the original model \eqref{eq:lin_reg}. Its proof is straight forward, and hence omitted.

\begin{lemma}\label{lem:suff}
Using the above notation, $m_\mu({\bf y})\propto \widetilde{m}_\mu({\bf z})$, and consequently $\arg \sup_{\mu\in \mathcal{P}} m_\mu({\bf y})=\arg \sup_{\mu\in \mathcal{P}}  \widetilde{m}_\mu({\bf z})$. Here the constant of proportionality is independent of $\mu$.
\end{lemma}

Henceforth we will work with model~\eqref{eq:reg_z}.  
Throughout the rest of the paper, we fix $\upsilon^2=(2C_2)^{-1}\in (0, C_2^{-1})$. Note that \eqref{eq:eigen2} implies that the matrix $\bar{\Sigma}:=\Sigma-\upsilon^2{\bf I}$ is positive definite, as
\begin{align}\label{eq:eigen33}
\frac{1}{2C_2} \le \lambda_{\min}(\bar{\Sigma})\le  \lambda_{\max}(\bar{\Sigma})\le \frac{1}{C_1}.
\end{align}
Consequently
we can decompose $\Sigma^{1/2} {\bm \varepsilon} = {\bm \varepsilon_1 } + {\bm \varepsilon_2}$, where ${\bm \varepsilon}_1 \sim N(0, \upsilon^2 \mathbf{I})$ and ${\bm \varepsilon}_2 =\Sigma^{1/2} {\bm \varepsilon} - {\bm \varepsilon}_1 \sim N(0, \Sigma - \upsilon^2 \mathbf{I})$ are independent. Plugging this back into \eqref{eq:reg_z}, we have, 
\begin{align}
\label{eq:modified_sequence}
    \mathbf{z} = \boldsymbol{\beta} + \Sigma^{1/2} {\bm {\varepsilon} } := \bar{\boldsymbol{\beta}} + {\bm \varepsilon}_2 
\end{align}
where $\bar{\beta}_i = \beta_i + \varepsilon_{1,i} \sim \mu \star N(0, \upsilon^2)$ are i.i.d., and ${\bm \varepsilon}_2 \sim N(0, \bar{\Sigma} )$ (here $\bar{\Sigma} = \Sigma - \upsilon^2 \mathbf{I}$)  are independent. 
We will switch to this representation in some of  our proofs.

 \subsection{Auxiliary lemmas for Theorem \ref{thm:ml_rate} and Theorem \ref{lem:approx}}
 We will now state two supporting lemmas, which will be used to verify Theorem \ref{thm:ml_rate} and Theorem \ref{lem:approx}. Their proofs are deferred to Appendix \ref{Appendix-A}. For stating the lemmas, we require some notation, which we introduce below. With $\widetilde{m}_{\mu,\Sigma}(\cdot)$ denoting the joint density of ${\bf z}$ w.r.t.~Lebesgue measure, a direct calculation gives
\begin{align}\label{eq:gg}
\log \widetilde{m}_{\mu,\Sigma}({\bf v})
 =\log \Big[(2\pi)^{-p/2}|\Sigma|^{-1/2}\Big]+\widetilde{g}({\bf v},\mu, \Sigma),
 \end{align}
 where
 \begin{align}\label{eq:gh}
\widetilde{g}({\bf v},\mu, \Sigma):= \log \int e^{-\widetilde{h}({\bf v}, {\bm \theta},\Sigma)}d\mu^{\otimes p}({\bm \theta}),\qquad 
\widetilde{h}({\bf v},{\bm \theta}, \Sigma):=\frac{1}{2}({\bf v}-{\bm \theta})^\top \Sigma^{-1}({\bf v}-{\bm \theta}).
 \end{align}
Thus, the marginal log-density \(\log\widetilde m_{\mu,\Sigma}({\bf v})\)
decomposes into a Gaussian normalizing constant, which is independent of
\(\mu\), and the prior-dependent log-partition term
\(\widetilde g({\bf v},\mu,\Sigma)\). Since the former does not affect
optimization over \(\mu\), the term \(\widetilde g({\bf v},\mu,\Sigma)\)
will be the object of primary interest in what follows.
 
 Note that the posterior distribution of ${\bm \beta}$ given ${\bf z}$, denoted henceforth by $\widetilde{H}({\bf z},\mu,\Sigma)$, 
 is given by
\begin{align}\label{eq:post1}
 \frac{d\widetilde{H}({\bf z},\mu,\Sigma)}{d\mu^{\otimes p}}({\bm \theta})= e^{-\widetilde{h}({\bf z}, {\bm \theta},\Sigma)-\widetilde{g}({\bf z},\mu,\Sigma)}.
 \end{align}

Our first lemma gives a stability result for the log-partition term $\widetilde{g}(\cdot,\cdot)$ in terms of the prior distribution $\mu$. We will derive the stability in terms of a new pseudo-metric defined below. 

\begin{definition}\label{def:metric}
 For any probability measure  $\mu$ on $[-1,1]$, let $\Psi_{\mu,\upsilon^2}:\R\mapsto (0,\infty)$ denote the density of $\mu \star N(0, \upsilon^2)$, given by
    \begin{align}
    \label{eq:Upsilon}
        \Psi_{\mu,\upsilon^2}(\theta) :=\int_{[-1,1]} \frac{1}{\sqrt{2\pi \upsilon^2}}e^{-\frac{(\theta-\beta)^2}{2\upsilon^2}} \mu(d\beta).
    \end{align}
 For any two probability measures $\mu_1$, $\mu_2$ supported on $[-1,1]$, define a discrepancy measure
$$d_{\upsilon^2}(\mu_1,\mu_2) :=\sup_{\theta\in \R}\frac{1}{1+\theta^2}\Big|\log \frac{\Psi_{\mu_1,\upsilon^2}(\theta)}{\Psi_{\mu_2,\upsilon^2}(\theta)}\Big|. $$
In particular, this gives
\begin{align}\label{eq:Psi}\Psi_{\mu_1,\upsilon^2}(\theta)\le \Psi_{\mu_2,\upsilon^2}(\theta) e^{(1+\theta^2)d_{\upsilon^2}(\mu_1,\mu_2)},\text{ for all }\theta\in \R.
\end{align}
\end{definition}

\begin{lemma}
    \label{lem:non-asymptotic-stability}
    Suppose $\mathbf{X}$ satisfies Assumption~\ref{assump:On-Design} (i.e.,~\eqref{eq:eigen}). Let $\mu_1 , \mu_2 \in \mathcal{P}$. Then for every $K>0$, there exists $C:= C(K,C_1, C_2)$ such that  
    \begin{align}
        \sup_{\mathbf{v}: \| \mathbf{v}\|_2^2 \leq Kp} \frac{1}{p} \Big| \widetilde{g}(\mathbf{v}, \mu_1,\Sigma) - \widetilde{g}(\mathbf{v}, \mu_2, \Sigma) \Big| \leq C d_{\upsilon^2}(\mu_2, \mu_1) + e^{-p}. 
    \end{align}
\end{lemma}
We now introduce some notation that will be used in the following lemma, and in the sequel. Using the relation ${\bf z}={\bm \beta}+\Sigma^{1/2}{\bm \varepsilon}$ (see \eqref{eq:reg_z}) we also have another form for $\widetilde{m}_{\mu,\Sigma}({\bf v})$, given by:
\begin{align}
\begin{split}\label{eq:gg2}
 \log \widetilde{m}_{\mu,\Sigma}({\bf z}) & =\log \Big[(2\pi)^{-p/2}|\Sigma|^{-1/2}\Big] \\ & \hspace{0.2in} +\log \int \exp\Big[-\frac{1}{2}({\bm \beta}+\Sigma^{1/2}{\bm \varepsilon}-{\bm \theta})^\top \Sigma^{-1}({\bm \beta}+\Sigma^{1/2}{\bm \varepsilon}-{\bm \theta})\Big]d\mu^{\otimes p}({\bm \theta})\\
 &=\log \Big[(2\pi)^{-p/2}|\Sigma|^{-1/2}\Big]+{g}({\bm \beta},{\bm \varepsilon},\mu,\Sigma),
 \end{split}
\end{align}
where
\begin{align}\label{eq:gh2}
{g}({\bm \beta},\boldsymbol{\varepsilon},\mu,\Sigma) := \log \int e^{-{h}({\bm \beta}, {\bm \varepsilon},{\bm \theta})}d\mu^{\otimes p}({\bm \theta})
\end{align}
with 
\begin{align*}
{h}({\bm \beta},{\bm \varepsilon},{\bm \theta},\Sigma)=\frac{1}{2}({\bm \beta}+\Sigma^{1/2}{\bm \varepsilon}-{\bm \theta})^\top \Sigma^{-1}({\bm \beta}+\Sigma^{1/2}{\bm \varepsilon}-{\bm \theta}).
\end{align*}
Comparing \eqref{eq:gg} and \eqref{eq:gg2} it follows that
\begin{align}\label{eq:gg3}
    \widetilde{g}({\bf z},\mu, \Sigma)={g}({\bm \beta},\boldsymbol{\varepsilon},\mu,\Sigma).
\end{align}
Throughout the rest of the paper, we will omit the dependence on $\Sigma$ in $\widetilde{m}_{\mu,\Sigma}(\cdot),   {g}({\bm \beta},\boldsymbol{\varepsilon},\mu,\Sigma)$, $\widetilde{g}({\bf z},\mu, \Sigma)$, barring when we will need to track $\Sigma$ changes, where we will explicitly point out this dependence.

Our second lemma derives the concentration of the function $g$ defined above.

\begin{lemma}
\label{lem:exp_concentration}
There exists constants $C, c >0$ depending only on $C_2$ such that for any $\varepsilon>0$ and $\mu \in \mathcal{P}$,  
\begin{align}
    \mathbb{P}_{\mu^*}\Big( |g(\boldsymbol{\beta}, {\bm \varepsilon}, \mu) - g(\boldsymbol{\beta}, {\bm \varepsilon}, \mu^*) - \mathbb{E}_{\mathbb{P}_{\mu^*}}[g(\boldsymbol{\beta}, {\bm \varepsilon}, \mu) - g(\boldsymbol{\beta}, {\bm \varepsilon}, \mu^*)] | > p \varepsilon \Big)\leq  C\exp(-c p \varepsilon^2). \nonumber 
\end{align}
\end{lemma}

\begin{remark}
For any two measures $\mu_1$ and $\mu_2$,  the discrepancy measure $d_{\upsilon^2}(\cdot, \cdot)$ compares the log-density of $\mu_1 \star N(0,\upsilon^2)$ and $\mu_2 \star N(0,\upsilon^2)$. As  $|\theta|\to \infty,$ the log-ratio of the two smoothed densities linearly in $|\theta|$, and so the factor $1+ \theta^2$ ensures that the measure $d_{\upsilon^2}(\cdot, \cdot)$ effectively focuses on the density ratios in a compact set. It is not hard to show that $d_{\upsilon^2}(\mu_n, \mu) \to 0$ as $n \to \infty$ , iff the sequence $\mu_n \to \mu$ weakly. Since we don't need this result in the current paper, we omit the formal statement and proof of such a result.  
\end{remark}

Let $\mathcal{N}(\mathcal{P},d, \varepsilon)$ denote the $\varepsilon$-covering number of the parameter space $\mathcal{P}$ under the pseudo-metric $d_{\upsilon^2}(\cdot,\cdot)$. The following lemma relates the metric $d_{\upsilon^2}(\mu_1, \mu_2)$ to the $L^{\infty}$ distance on $\mathcal{P}$, defined by
\begin{equation}\label{eq:L_infty}
L^\infty(\mu_1,\mu_2):=\|\Psi_{\mu_1, \upsilon^2} - \Psi_{\mu_2,\upsilon^2}\|_{\infty}.
\end{equation}
\begin{lemma}
\label{lemma:covering}
 There exists a constant $c$ depending only on $C_2$, such that for any $\varepsilon>0$, setting
$   \delta = \exp\Big(-\frac{c}{\varepsilon^2}\Big)$
we have $\mathcal{N}(\mathcal{P}, d_{\upsilon^2}, \varepsilon) \leq \mathcal{N}(\mathcal{P}, L^{\infty}, \delta)$.
\end{lemma}
\begin{lemma}
    \label{lemma:chain_rule}
    For any two measures $\mu_1, \mu_2 \in \mathcal{P}$ we have 
    \begin{align*}
        \mathrm{D}_{\mathrm{KL}}(\widetilde{m}_{\mu_1} \| \widetilde{m}_{\mu_2}) \geq p\,\, \mathrm{D}_{\mathrm{KL}} \left(\mu_1 \star N(0, C_1^{-1}) \| \mu_2 \star N(0, C_1^{-1}) \right). \nonumber 
    \end{align*}
\end{lemma}
\begin{lemma}\label{lem:2026}
   For any $C>0$, there exists finite positive constants $c_1,c_2$ depending only on $C,C_1,C_2$, such that for any sequence of non-negative reals $\{\Delta_p\}_{p\ge 1}$ with $\Delta_p\ge p^{-1/6}$, setting
   \begin{align*}
       \mathcal{P}_p:=\left\{ \mu \in \mathcal{P}: d_{\mathcal{H}}^2(\mu^* \star N(0, C_1^{-1}) , \mu \star N(0, C_1^{-1}) ) >c_1\Delta_p\right\}
       \end{align*}
       we have 
       \begin{align*}
           \P_{\mu^*}\left(\log m_{\mu^*}({\bf y}) \le \sup_{\mu\in \mathcal{P}_p}\log m_{\mu}({\bf y})+ Cp\Delta_p\right)\le c_2 \exp\Big(-\frac{p\Delta_p^2}{c_2}\Big) \le c_2 \exp\Big(-\frac{p^{2/3}}{c_2}\Big).
       \end{align*}
       \end{lemma}
      \begin{proof}
      The second bound is immediate from the first, on noting that $\Delta_p\ge p^{-1/6}$. For the first bound,
          use the fact that ${\bf z}={\bm \beta}+\Sigma^{1/2}{\bm \varepsilon}$ (see \eqref{eq:modified_sequence}) along with the eigenvalue bounds on $\Sigma$ (see \eqref{eq:eigen2}) to note the existence of a constant $K$ depending only on $C_1$, such that
\begin{align}\label{eq:ztail}
    \P_{\mu^*}(\|{\bf z}\|_2^2\ge Kp)\le e^{-p}.
\end{align}
Working under the sequence model \eqref{eq:reg_z}, where $g$ is as in \eqref{eq:gh2}, we can write
\begin{small}
\begin{align}\label{eq:conl2}
    &\P_{\mu^*} \left(\log m_{\mu^*}({\bf y})\le \sup_{\mu \in \mathcal{P}_p}\log m_\mu({\bf y})+Cp\Delta_p \right) \nonumber \\
   =&\mathbb{P}_{\mu^*} \left(g(\boldsymbol{\beta}, {\bm \varepsilon}, \mu^*) \le \sup_{\mu \in \mathcal{P}_p}g(\boldsymbol{\beta}, {\bm \varepsilon}, \mu)+ Cp\Delta_p \right) \nonumber \\
   &\leq \mathbb{P}_{\mu^*} \left(g(\boldsymbol{\beta}, {\bm \varepsilon}, \mu^*) \le \sup_{\mu \in \mathcal{P}_p}g(\boldsymbol{\beta}, {\bm \varepsilon}, \mu)+ Cp\Delta_p , \| \mathbf{z}\|_2^2 \leq Kp \right) + e^{-p}. 
\end{align}
\end{small}
 Let $ \mathcal{N}\subseteq \mathcal{P}_p$ be an $\eta$-net of $\mathcal{P}_p$ with respect to $d_{\upsilon^2}$ (introduced in Definition \ref{def:metric}), where  $\eta =  c_1'\Delta_p$. Here $c_1'$ is a positive constant depending only on $C, C_1,C_2$, but to be specified later.  For any $\mu_0\in \mathcal{P}_p$, let $\mu_1\in \mathcal{N}$ be such that $d_{\upsilon^2}(\mu_0,\mu_1)<\eta.$ On the set $\|{\bf z}\|_2^2\le Kp$, invoking Lemma \ref{lem:non-asymptotic-stability} and the fact that $\Delta_p\ge p^{-1/6}$ we have the existence of a constant $c_2'$ (depending only on $K, C_1,C_2$), such that $$|\widetilde{g}({\bf z},\mu_1)-\widetilde{g}({\bf z},\mu_0)|\le c_2' p\eta=c_1'c_2' p\Delta_p.$$
Since $\widetilde{g}({\bf z},\mu)=g(\boldsymbol{\beta}, {\bm \varepsilon}, \mu)$ (see \eqref{eq:gg3}), we get
\begin{align}\label{eq:conl4}
\notag&\mathbb{P}_{\mu^*}\left(g(\boldsymbol{\beta}, {\bm \varepsilon}, \mu^*) \le \sup_{\mu \in \mathcal{P}_p}g(\boldsymbol{\beta}, {\bm \varepsilon}, \mu)+ Cp\Delta_p, \|\mathbf{z}\|_2^2 \leq Kp\right)\\
    \notag\le&\; \mathbb{P}_{\mu^*}\left(g(\boldsymbol{\beta}, {\bm \varepsilon}, \mu^*) \le \max_{\mu \in \mathcal{N}}g(\boldsymbol{\beta}, {\bm \varepsilon}, \mu)+ p\Delta_p(C+c_1'c_2') , \| \mathbf{z}\|_2^2 \leq Kp\right)\\
    \le &\;|\mathcal{N}|\max_{\mu\in \mathcal{N}}\mathbb{P}_{\mu^*}\left(g(\boldsymbol{\beta}, {\bm \varepsilon}, \mu^*) \le g(\boldsymbol{\beta}, {\bm \varepsilon}, \mu)+ (C+c_1'c_2')p\Delta_p\right).
\end{align}

To bound the cardinality of $\mathcal{N}$, using Lemma~\ref{lemma:covering} gives the existence of a constant $c_3'$ (depending only on $C_2$), such that with $\delta=e^{-\frac{c_3'}{\eta^2}}$ we have 
\begin{align}\label{eq:conl3}
    |\mathcal{N}|\le \mathcal{N}(\mathcal{P}_p,L^\infty,\delta)\le  \exp \Big( \Big(\log \frac{1}{\delta} \Big)^2\Big) =\exp\Big(\frac{(c_3')^2}{\eta^4}\Big)= \exp\Big(\frac{(c_3')^2}{(c_1')^4\Delta_p^4}\Big), 
\end{align}
where the second inequality uses \citet[Theorem 3.1]{Ghosal-vdVaart-2001}.

To bound the probability in the RHS of \eqref{eq:conl4}, observe that
\begin{align}
    \mathbb{E}_{\mathbb{P}_{\mu^*}}[\widetilde{g}(\mathbf{z}, \mu^*) - \widetilde{g}(\mathbf{z}, \mu)] &= \mathrm{D}_{\mathrm{KL}}(\widetilde{m}_{\mu^*} \| \widetilde{m}_{\mu}) \nonumber \\
    &\geq p \,\, \mathrm{D}_{\mathrm{KL}} (\mu^* \star N(0, C_1^{-1}) \| \mu \star N(0, C_1^{-1})) \nonumber \\
    &\geq 2 p \,\, d_{\mathcal{H}}^2( \mu^* \star N(0, C_1^{-1}) , \mu \star N(0, C_1^{-1})) \nonumber \\
    &\geq 2c_1  p\Delta_p. \nonumber 
\end{align}
The first inequality above follows Lemma~\ref{lemma:chain_rule} and the second inequality follows from the relations between KL divergence and the Hellinger distance \citep{polyanskiy2025information}. The last inequality follows from the fact that $\mu\in \mathcal{P}_p$. Choosing $c_1=C+c_1'c_2'$ (where the choice of $c_1'$ has not yet been specified), for any $\mu\in \mathcal{N}\subseteq \mathcal{P}_p$ Lemma~\ref{lem:exp_concentration}
gives the existence of constants $c_4',c_5'$ (depending only on $C_1,C_2$), such that
\begin{align}\label{eq:conl5}
\mathbb{P}_{\mu^*}\left(g(\boldsymbol{\beta}, {\bm \varepsilon}, \mu^*) \le g(\boldsymbol{\beta}, {\bm \varepsilon}, \mu)+ (C+c_1'c_2')p\Delta_p\right)
\le c_4' e^{-c_5'c_1^2 p\Delta_p^2}.
\end{align}

Combining \eqref{eq:conl2}, \eqref{eq:conl4}, \eqref{eq:conl3} and \eqref{eq:conl5}, we get 
\begin{align*}
   \P_{\mu^*}\left(\log m_{\mu^*}({\bf y}) \le \sup_{\mu \in \mathcal{P}_p} \log m_{\mu}({\bf y})+ c_0p\Delta_p\right)\le c_4' \exp\left[\frac{(c_3')^2}{(c_1')^4\Delta_p^4}-c_5'c_1^2p\Delta_p^2 \right].
\end{align*}
Recalling that $c_1=C+c_1'c_2'$, the desired conclusion follows on taking $c_1'$ large enough to ensure that $(C+c_1'c_2')^2c_5'\ge 2 \frac{(c_3')^2}{(c_1')^4},$ and again using the fact that $\Delta_p\ge p^{-1/6}.$
\end{proof}

\subsection{Proof of Theorem~\ref{thm:ml_rate}}
\label{pf:ml_rate}
Setting \[\mathcal{P}_p:=\{ \mu \in \mathcal{P}: d_{\mathcal{H}}^2(\mu^* \star N(0, C_1^{-1}) , \mu \star N(0, C_1^{-1}) ) > cp^{-1/6} \},\]
the conclusion of theorem is equivalent to showing that \begin{align}\label{eq:conl1}
\mathbb{P}_{\mu^*}(\hat{\mu}_{\mathrm{ML}}\in \mathcal{P}_p)\le C \exp\Big(-\frac{p^{2/3}}{C}\Big).
\end{align}
Here $c,C$ are positive constants which can only depend only on $C_1,C_2$.
Recalling that $\hat{\mu}_{\mathrm{ML}}$ is a $\varepsilon_p$-optimizer for the map $\mu\mapsto m_\mu({\bf y})$ with $\varepsilon_p\le p^{5/6}$ a.s., we get
\begin{align*}
  \mathbb{P}_{\mu^*}(\hat{\mu}_{\mathrm{ML}}\in \mathcal{P}_p)\le & \; \P_{\mu^*}\Big(\log m_{\mu^*}({\bf y})\le \sup_{\mu \in \mathcal{P}_p}\log m_\mu({\bf y})+p^{5/6}\Big).
\end{align*}
The desired conclusion \eqref{eq:conl1} then follows immediately on invoking Lemma \ref{lem:2026} with $\Delta_p=p^{-1/6}$ and $C=1$ to bound the RHS above.

 \subsection{Proof of Theorem \ref{lem:approx}}\label{pf:lem-approx}

 \begin{enumerate}
 \item[(a)] To begin, use \eqref{eq:zp} along with \citet[Cor 1.2]{augeri2020nonlinear} to note the existence of a universal constant $\kappa$ such that
\begin{align*}
    \log Z_p({\bf v},\mu)=& \log \int_{[-1,1]^p}\exp\Big(-\frac{1}{2}{\bm \beta}^\top A{\bm \beta}+{\bf v}^\top {\bm \beta}\Big)\prod_{i=1}^p \mu_{(0,d_i)}(\beta_i)\\
    \le& \sup_{{\bm \gamma}\in \R^p} M_p({\bm \gamma},{\bf v}, \mu)+\kappa p^{1/3} \left(\E\sup_{{\bf u}\in [-1,1]^p} {\bf u}^\top A {\bf Z} \right)^{2/3},
\end{align*}
where ${\bf Z}=(Z_1,\cdots,Z_p)\stackrel{i.i.d.}{\sim}N(0,1)$. Proceeding to bound the last term in the RHS above, Cauchy-Schwarz inequality gives
\begin{align*}
\E\sup_{{\bf u}\in [-1,1]^p} {\bf u}^\top A {\bf Z}=\E \|A{\bf Z}\|_1\le \sqrt{p} \sqrt{\E \|A{\bf Z}\|_2^2}=\sqrt{p{\rm Tr}(A^2)}.
\end{align*}
Combining the above two displays we get
$$\log Z_p({\bf v},\mu)\le \sup_{{\bm \gamma}\in \R^p} M_p({\bm \gamma},{\bf v}, \mu)+\kappa p^{2/3} {\rm Tr}(A^2)^{\frac{1}{3}},$$
which gives the desired uniform upper bound of part (a). The corresponding lower bound follows from \eqref{eq:elbo_lower}, and so the proof of part (a) is complete. 
 
 \item[(b)] 
As in the proof of Theorem~\ref{thm:ml_rate}, setting \[\mathcal{Q}_p:=\left\{ \mu \in \mathcal{P}: d_{\mathcal{H}}^2(\mu^* \star N(0, C_1^{-1}) , \mu \star N(0, C_1^{-1}) ) > c\Big[p^{-1/6}+\Big(\frac{{\rm Tr}(A^2)}{p}\Big)^{1/3} \Big]\right\},\]
Then the conclusion of theorem is equivalent to showing that \begin{align}\label{eq:conl21}
\mathbb{P}_{\mu^*}(\hat{\mu}_{\mathrm MF}\in \mathcal{Q}_p)\le C \exp\Big(-\frac{p^{2/3}}{C}\Big).
\end{align}
Here $c,C$ are positive constants which can only depend on $C_1,C_2$. Again recalling that $\hat{\mu}_{MF}$ is a $\varepsilon_p$ optimizer for the map $\mu\mapsto \sup_{{\bm \gamma}\in \R^p} \widetilde{M}_p({\bm \gamma},{\bf w},\mu)$ with $\varepsilon_p\le p^{\frac{5}{6}}$ (see \eqref{eq:Tilde-M_p} and \eqref{eq:mu-MF}), 
we have
\begin{align}\label{eq:conl22}
  \notag \mathbb{P}_{\mu^*}(\hat{\mu}_{\mathrm MF}\in \mathcal{Q}_p)\le &\mathbb{P}_{\mu^*}\left(\sup_{{\bm \gamma}\in \R^p} \widetilde{M}_p({\bm \gamma},{\bf w},\mu^*)\le  \sup_{\mu \in \mathcal{Q}_p}\sup_{{\bm \gamma}\in \R^p} \widetilde{M}_p({\bm \gamma},{\bf w},\mu)+ p^{\frac{5}{6}}\right)\\
   \le &\mathbb{P}_{\mu^*}\left(\log m_{\mu^*}({\bf y})
   \le  \sup_{\mu \in \mathcal{Q}_p} \log m_\mu({\bf y})+p^{5/6}+2\kappa p^{2/3}{\rm Tr}(A^2)^{1/3}\right)
\end{align}
where the last inequality uses  part (a) along with \eqref{eq:elbo_explicit} to conclude that 
$$\sup_{\mu\in \mathcal{P}}\Big|-\frac{n}{2} \log (2 \pi \sigma^2) - \frac{\|\mathbf{y}\|_2^2 }{2 \sigma^2 }+\sup_{{\bm \gamma}\in \R^p} \widetilde{M}_p({\bm \gamma},{\bf w},\mu)- \log m_\mu({\bf y})\Big|\le \kappa p^{2/3} {\rm Tr}(A^2)^{1/3}.$$ 

Invoking Lemma~\ref{lem:2026} with the choices $C=\max(1,2\kappa)$ and $\Delta_p=p^{-1/6}+\Big(\frac{{\mathrm Tr}(A^2)}{p}\Big)^{1/3}$ to bound the RHS of \eqref{eq:conl22} , the desired conclusion \eqref{eq:conl21} follows.
 \end{enumerate}
 \qed

\section{Proof of Lemma~\ref{lemma:separation_sufficient}, Theorem \ref{thm:inference} and Corollary \ref{cor:takeaways}}
\label{sec:thm_inf} 

We prove Lemma~\ref{lemma:separation_sufficient}, Theorem \ref{thm:inference} and Corollary \ref{cor:takeaways} in this section. We start with the proof of Lemma~\ref{lemma:separation_sufficient}, the proof of which is analogous to \citet[Lemma 25]{mukherjee2022variational}.

\begin{proof}[Proof of Lemma~\ref{lemma:separation_sufficient}] 
To begin, use the expression of $\mathcal{M}_p(\cdot,\mathbf{w},\mu^*)$ from \eqref{eq:M_mean_param} along with the lower semi continuity of the Kullback-Leibler divergence to conclude that $\mathcal{M}_p(\cdot,\mathbf{w},\mu^*)$ is upper semi continuous, and so existence of a maximizer $\mathbf{u}^*$ on the compact set $[-1,1]^p$ is immediate. To conclude existence of a well separated optimizer, it suffices to show strong concavity of the function $\mathcal{M}_p(\cdot,\mathbf{w},\mu^*)$, i.e.~there exists a constant $\delta>0$ (free of $p$) such that
\begin{align}\label{eq:concave}
    \sup_{\mathbf{u}\in (\inf {\rm supp}(\mu^*),\sup {\rm supp}(\mu^*))^p}\lambda_{\max}(H_p(\mathbf{u},\mathbf{w},\mu^*))\le -\delta,
\end{align}
where $H_p(\cdot,\mathbf{w},\mu^*)$ is the Hessian of $\mathcal{M}_p(\cdot,\mathbf{w},\mu^*)$, given by
\begin{align*}
 H_p(\mathbf{u},\mathbf{w},\mu^*)=-A-\frac{1}{\ddot{c}_{\mu^*}\Big(h_{\mu^*,d_i}((u_i),d_i\Big)}.
\end{align*}
We now verify this under the  assumptions of the lemma.
\begin{itemize}
    \item[(i)] 
In this case, since $\ddot{c}(\cdot,d_i)$ is the variance of a random variable supported on $[-1,1]$, we can use the trivial bound $\ddot{c}_{\mu^*}\Big(h_{\mu^*,d_i}((u_i),d_i\Big)\le 1$, giving
\[H_p(\mathbf{u},\mathbf{w},\mu^*)=-A-\frac{1}{\ddot{c}_{\mu^*}\Big(h_{\mu^*,d_i}((u_i),d_i\Big)}\preceq -A-{\bf I}\preceq -\eta {\bf I},\]
thus verifying \eqref{eq:concave} with $\delta=\eta$.
\\
 \item[(ii)] 
 In this case, the assumptions on the prior $\mu^*$ ensures that it satisfies the GHS inequality (see \citet{ellis1976ghs}[Theorem 1.1]{}) and (\citet{ellis1976ghs}[Theorem 1.2 part (c)]), 
    which in particular gives the inequality
    $\ddot{c}_{\mu^*}(\theta,d_i)\le \ddot{c}_{\mu^*}(0,d_i)$ for any $\theta\in \R$. We further claim that for any $\theta\in \R$,
\begin{align}\label{eq:ghs_claim}
        \sup_{d\ge 0}d\ddot{c}_{\mu^*}(\theta,d)\le 1.
    \end{align}
    Given this claim, we can  use the  bound
    \[\ddot{c}_{\mu*}\Big(h_{\mu^*,d_i}(u_i),d_i\Big)\le \ddot{c}_{\mu*}\Big(0,d_i\Big)\le \frac{1}{d_i},\]
along with the observation that $d_i=\sigma^{-2}(\mathbf{X}^\top\mathbf{X})_{ii}$ (see \eqref{eq:notation}) to conclude 
    \[H_p(\mathbf{u},\mathbf{w},\mu^*)\preceq -\sigma^{-2}\mathbf{X}^\top \mathbf{X}\preceq -C_1{\bf I}, \]
    where the last step uses \eqref{eq:eigen}. Thus we have again verified \eqref{eq:concave} with $\delta=C_1$.
    \\
    It thus remains to verify the claim \eqref{eq:ghs_claim}. To this effect, consider a scale family of parametric distributions 
\(\{P_\theta:\theta\ge 0\}\) with
\[
\frac{dP_\theta}{dx}
\propto
\exp\{-\theta V(x)\}\exp\left\{-\frac{d x^2}{2}\right\}.
\]
Since \(V(\cdot)\) is even, it follows that the first moment under \(P_\theta\) is
\(0\) for all \(\theta\). Also, since \(V(\cdot)\) is increasing, this family has
monotone likelihood ratio in \(T(x)=|x|\), and thus the second moment under the law
\(P_\theta\) is decreasing in \(\theta\) for \(\theta>0\). Therefore,
\[
\ddot c(0,d)
=
\operatorname{Var}_{\theta=1}(Z)
\le
\operatorname{Var}_{\theta=0}(Z).
\]
Proceeding to bound the right-hand side of the above display, for \(d>0\) we have
\[
d\,\operatorname{Var}_{\theta=0}(Z)
=
\frac{
d\displaystyle\int_{-1}^{1} z^2 \exp\left\{-\frac{d z^2}{2}\right\}\,dz
}{
\displaystyle\int_{-1}^{1} \exp\left\{-\frac{d z^2}{2}\right\}\,dz
}.
\]
We substitute \(t=\sqrt d\,z\), so that
\[
d\,\operatorname{Var}_{\theta=0}(Z)
=
\frac{
\displaystyle\int_{-\sqrt d}^{\sqrt d} t^2 \exp\left\{-\frac{t^2}{2}\right\}\,dt
}{
\displaystyle\int_{-\sqrt d}^{\sqrt d} \exp\left\{-\frac{t^2}{2}\right\}\,dt
}.
\]
This is exactly the variance of a truncated standard Gaussian distribution, truncated to the
interval \([-\sqrt d,\sqrt d]\). Indeed, the truncated variance is
\[
1-
\frac{2\sqrt d\,\phi(\sqrt d)}{2\Phi(\sqrt d)-1}
<1,
\]
where \(\phi(\cdot)\) and \(\Phi(\cdot)\) represent the pdf and cdf of the standard Gaussian
distribution, respectively (see \citet{johnson1994continuous}).
\end{itemize} 
\end{proof}


We now turn to the proof of Theorem \ref{thm:inference} and Corollary \ref{cor:takeaways}. 
To this end, we make the following definition.


\begin{definition}
Let $\{a_p : p \geq 1\}$ be any deterministic sequence of positive reals.  
Let $\{\Gamma(p,K): p\geq 1, K \geq 1\}$ be an array of random variables  satisfying 
$$\limsup_{K\to\infty}\limsup_{p\to\infty}\P_{\mu^*}(|\Gamma(p,K)|>\delta a_p)= 0$$ for every $\delta >0$. 
We will denote such a sequence as $\er(a_p)$.  

\end{definition}

In our subsequent discussions, the $\er(\cdot)$ variable may change from line to line. Throughout, we suppress the dependence of $\er(\cdot)$ on $\mu^*$. 

 \subsection{Auxiliary lemmas for Theorem \ref{thm:inference} and Corollary \ref{cor:takeaways}}
 
We now state three lemmas that will be useful in our subsequent analysis. We defer the proofs of these results to Appendix \ref{sec:supporting}. 
\begin{lemma}\label{lem:w}
    With $\mathbf{w}$ as in \eqref{eq:notation}, we have
    $\|\mathbf{w}\|_2^2=O_{\P_{\mu^*}}(p)$.
\end{lemma}

\begin{lemma}\label{lem:unif_cont}
Recall Definition~\ref{defn:Quad-Tilt}. The following conclusions hold:

(i) The function $(\mu,\tau,d)\mapsto c_\mu(\tau,d)$ is continuous on $\mathcal{P}\times \R\times (0,\infty)$.

(ii) The function $(\mu,\tau,d)\mapsto \mu_{\tau,d}$ is continuous on $\mathcal{P}\times \R\times (0,\infty)$.

(iii) The function $(\mu,\tau,d)\mapsto \dot{c}_\mu(\tau,d)$ is continuous on $\mathcal{P}\times \R\times (0,\infty)$.

(iv) The function $(v,d)\mapsto h_{\mu,d}(v)$ is continuous on $I_\mu\times (0,\infty)$, for any $\mu\in \mathcal{P}$. Here $I_\mu=(\inf{\rm supp}(\mu), \sup {\rm supp} (\mu))$ is as in Definition \ref{defn:Quad-Tilt}.

\end{lemma}

\begin{lemma}
\label{lem:dev_bounds}
    Assume we are in the setting of Theorem \ref{thm:inference}. Let $\mathbf{v}^{(p)} = (v^{(p)}_1,\ldots, v^{(p)}_p)\in [-1,1]^p$ be any global optimizer of $\mathcal{M}_p(\cdot,\mathbf{w},\mu^{(p)})$. Setting ${ \gamma}^{(p)}_i:=h_{\mu^{(p)},d_i}(v^{(p)}_i)$, for $i \in [p]$, we have $$| \{i \in [p]:|\gamma^{(p)}_i|> K\} |=\er(p). $$
\end{lemma}

\begin{lemma}
\label{lem:approximations}
    Assume we are in the setting of Theorem \ref{thm:inference}. Set 
        \begin{align}
     & {\tau}_i^{(p,K)}:=\tau_i^{(p)} 1\{|\tau_i|\le K\}, \qquad  \bm{\tau}^{(p,K)}:=(\tau_i^{(p,K)})_{i\in [p]}, \nonumber \\
     &  {\mathbf{u}}^{(p,K)} = \dot{c}_{\mu^{(p)}}({\bm \tau}^{(p,K)},\mathbf{d}), \qquad  \mathbf{u}_*^{(p,K)} = \dot{c}_{\mu^*}({\bm \tau}^{(p,K)} ,\mathbf{d}). \label{eq:u_defns} 
    \end{align}
    Then we have, 
    \begin{align}
        \| \mathbf{u}^{(p)} - {\mathbf{u}}^{(p,K)} \|_2^2 = \er(p), \qquad \mbox{and} \qquad
          \| {\mathbf{u}}_*^{(p,K)} - {\mathbf{u}}^{(p,K)} \|_2^2 = o_{\P_{\mu^*}}(p)\quad \forall\,\, K>0.  \nonumber
    \end{align}
    Combining, we have, $  \| \mathbf{u}^{(p)} - {\mathbf{u}}_*^{(p,K)} \|_2^2 =\er(p)$. 
\end{lemma}

\subsection{Proof of Theorem \ref{thm:inference}}

(i) {\bf Proof of $\|\mathbf{u}^*-\mathbf{u}^{(p)}\|_2^2=o_{\P_{\mu^*}}(p).$}
We break the proof into the following three steps:

\begin{enumerate}
\item[Step (1).]
 $$   \Big|  \M_p(\mathbf{u}^{(p)}, \mathbf{w}, \mu^{(p)}) - \sup_{\mathbf{u} \in [-1,1]^p} \M_p(\mathbf{u}, \mathbf{w}, \mu^*)  \Big| = o_{\P_{\mu^*}}(p). $$

\begin{proof}
Since $\mathrm{Tr}(A^2) = o(p)$, and $\mathbf{u}^{(p)}$ is a global optimizer for $\M_p(\cdot, \mathbf{w}, \mu^{(p)})$, Theorem \ref{lem:approx} part (a) and \eqref{eq:M_mean_param} imply 
\begin{align*} 
\sup_{\mathbf{u} \in [-1,1]^p} \M_p(\mathbf{u}, \mathbf{w}, \mu^*)  =\log Z_p(\mathbf{w},\mu^*)+o_{\P_{\mu^*}}(p), \quad 
 \M_p(\mathbf{u}^{(p)}, \mathbf{w}, \mu^{(p)}) =\log Z_p(\mathbf{w},\mu^{(p)})+o_{\P_{\mu^*}}(p).
\end{align*}
Also, using Lemma~\ref{lem:suff}, Lemma \ref{lem:non-asymptotic-stability}, \eqref{eq:gg}, \eqref{eq:mpiy} and the assumption $\mu^{(p)}\stackrel{w}{\to}\mu^*$ we conclude that 
\begin{align*}
     \log Z_p(\mathbf{w},\mu^{(p)})=\log Z_p(\mathbf{w},\mu^*)+o_{\P_{\mu^*}}(p).
\end{align*}
The desired conclusion follows on combining the above two displays.
\end{proof}

\item[Step (2).]
      $$  \Big|\M_p({\mathbf{u}}^{(p,K)}, \mathbf{w}, \mu^{(p)}) - \M_p(\mathbf{u}^{(p)}, \mathbf{w}, \mu^{(p)})\Big|=\er(p).$$

\begin{proof}
Recall from~\eqref{eq:G_defn} that $G_{\mu,d}(u)=\dkl(\mu_{h_{\mu,d}(u),d}\|\mu_{0,d})$.  
Let $\mathcal{S}(p,K):=\{i\in [p]:|\tau^{(p)}_i|\le K\}$, and use Lemma \ref{lem:dev_bounds} to conclude that $|\mathcal{S}(p,K)^c|=\er(p)$. Define $\mathbf{u}^{(p,K)}\in [-1,1]^p$ as
\begin{align*}
    {u}^{(p,K)}_i = 
    \begin{cases}
        {u}^{(p)}_i &\mathrm{ if } \,\, i\in  \mathcal{S}(p,K), \\
        \mathbb{E}_{\mu^{(p)}_{0,d_i}}[X] & \mathrm{if} \,\, i \in \mathcal{S}(p,K)^c. 
    \end{cases}
\end{align*}
Thus  $G_{\mu^{(p)},d_i}({u}^{(p,K)}_i)$ equals $G_{\mu^{(p)},d_i}({u}^{(p)}_i)$ if $i\in \mathcal{S}(p,K)$, and equals $ 0$ if $i\in  \mathcal{S}(p,K)^c$.
Combining, for all $i\in [p]$ we have $$G_{\mu^{(p)},d_i}({u}^{(p,K)}_i) \le G_{\mu^{(p)},d_i}(u^{(p)}_i).$$ 

Since $\mathbf{u}^{(p)}$ is a global optimizer of $\mathcal{M}_
p(\cdot, \mathbf{w}, \mu^{(p)})$
using \eqref{eq:M_mean_param} and the display above we have  
\begin{align*}
0&\leq    \mathcal{M}_
p(\mathbf{u}^{(p)}, \mathbf{w}, \mu^{(p)}) -   \mathcal{M}_
p({\mathbf{u}}^{(p,K)} , \mathbf{w}, \mu^{(p)}) \\ &\leq    -\frac{1}{2} (\mathbf{u}^{(p)})^{\top} A \mathbf{u}^{(p)} + (\mathbf{u}^{(p)})^{\top} \mathbf{w} + \frac{1}{2} ({\mathbf{u}}^{(p,K)} )^{\top} A {\mathbf{u}}^{(p,K)} -  ({\mathbf{u}}^{(p,K)} )^{\top} \mathbf{w}\\
&=O_{\P_{\mu^*}}(\sqrt{p})  \| \mathbf{u}^{(p)} - \mathbf{u}^{(p,K)}\|_2,\nonumber 
 \end{align*}
 where the last line uses the fact that $\lambda_{\max}(A)=O(1)$ (from \eqref{eq:eigen}) and $\|\mathbf{w}\|_2^2=O_{\P_{\mu^*}}(p)$ (from Lemma \ref{lem:w}), along with Cauchy-Schwarz inequality. 
  The desired conclusion now follows from Lemma \ref{lem:approximations}.
 \end{proof}
    
 \item[Step (3).] 
For any $K>0$, we have
 \begin{align*}
        \Big|  \M_p({\mathbf{u}}_*^{(p,K)}, \mathbf{w}, \mu^*) -\M_p({\mathbf{u}}^{(p,K)}, \mathbf{w}, \mu^{(p)}) \Big|=o_{\P_{\mu^*}}(p).
 \end{align*}
 \begin{proof}
Using the bounds $\lambda_{\max}(A)=O(1)$ and $\|\mathbf{w}\|_2^2 = O_{\P_{\mu^*}}(p)$ (as in the previous step) and Lemma \ref{lem:approximations}, using \eqref{eq:M_mean_param} we have, 
 \begin{align}
       \M_p({\mathbf{u}}^{(p,K)}, \mathbf{w}, \mu^{(p)}) - \M_p({\mathbf{u}}_*^{(p,K)}, \mathbf{w}, \mu^*) =  \sum_{i=1}^{p} (G_{\mu^*,d_i}({u}_{*i}^{(p,K)}) - G_{\mu^{(p)},d_i}({u}^{(p,K)}_i) ) + o_{\P_{\mu^*}}(p). \label{eq:int_target_22}
 \end{align}
Now, using the definition of $G_{\mu,d}(\cdot)$ \eqref{eq:G_defn} and the definitions of ${\mathbf{u}}^{(p,K)}$ and ${\mathbf{u}}_*^{(p,K)}$, we have, 
\begin{align}
    &\sum_{i=1}^{p} \left(G_{\mu^{(p)},d_i}({u}^{(p,K)}_i) - G_{\mu^*,d_i}({u}_{*i}^{(p,K)} ) \right) \;=\;  \sum_{i=1}^{p} \tau_i^{(p,K)} ({u}^{(p,K)}_i - {u}_{*i}^{(p,K)})  \nonumber \\
   & \hspace{1in} -  \sum_{i=1}^{p}\Big(c_{\mu^{(p)}}(\tau_i^{(p,K)}, d_i) - c_{\mu^*}(\tau_i^{(p,K)}, d_i) \Big) + \sum_{i=1}^{p} (c_{\mu^{(p)}}(0,d_i) - c_{\mu^*}(0,d_i)). \nonumber 
\end{align}
Using \eqref{eq:eigen} we get $C_1 < d_i < C_2$ for all $i\in [p]$. Noting that $|\tau_i^{(K)}| \leq K$, triangle inequality implies
\begin{align*}
&\Big|\sum_{i=1}^{p} (G_{\mu^{(p)},d_i}({u}^{(p,K)}_i) - G_{\mu^*,d_i}({u}^{(p,K)}_{*i}) )   \Big|\\
&\leq K \sqrt{p}\|\mathbf{u}^{(p,K)} - {\mathbf{u}}_*^{(p,K)}\|_2  + 2 p \sup_{|a| \leq K, C_1 \leq d \leq C_2} |c_{\mu^{(p)}}(a,d)- c_{\mu^*}(a,d)|=o_{\P_{\mu^*}}(p), 
\end{align*}
where the last equality uses Lemma \ref{lem:approximations}, along with assumption $\mu^{(p)}\stackrel{w}{\to} \mu^*$. The desired result follows on using the
last display along with  \eqref{eq:int_target_22}.  
\end{proof}

Combining the three steps above we conclude 
\begin{align*}
\M_p({\mathbf{u}}_*^{(p,K)}, \mathbf{w}, \mu^*)-   \sup_{\mathbf{u} \in [-1,1]^p} \M_p(\mathbf{u}, \mathbf{w}, \mu^*)=\er(p). 
\end{align*}
The well-separated optimizer property of $\M_p(\cdot, \mathbf{w}, \mu^*)$ (see Definition~\ref{defn:separation_property}) then gives
\begin{align*}
    \| {\mathbf{u}}_*^{(p,K)} - \mathbf{u}^*\|_2^2=\er(p).  
\end{align*}
The proof of the first assertion is complete by another application of Lemma~\ref{lem:approximations}.
\end{enumerate}

 {\bf Proof of $\|\mathbf{u}^*-\widetilde{\mathbf{u}}\|_2^2=o_{\P_{\mu^*}}(p).$} Define a random vector $\mathbf{b}\in [-1,1]^p$ by setting \begin{align}\label{eq:defb}
 b_i:=\E_{\P_{\mu^*}}[ \beta_i \mid \beta_{-i}, \mathbf{y}],
 \end{align}
  and use \citet[Theorem 7 (ii)]{mukherjee2022variational} along with the well-separated assumption to conclude that
  \begin{align}\label{eq:bconc}
  \|\mathbf{b}-\mathbf{u}
^*\|_2^2=o_{\P_{\mu^*}}(p).  \end{align}
Define a random variable $\zeta_p$ by first picking $I$ uniformly from $[p]=\{1,2,\ldots,p\}$, and then set $\zeta_p:=b_I-u^*_I$. By a slightly abuse of notation, let $\P_{\mu^*}$ denote the joint distribution of $(\mathbf{y},\bbeta,I)$. Then \eqref{eq:bconc} is equivalent to $\E_{\P_{\mu^*}}[\zeta_p^2 \mid \mathbf{y}, \bbeta]\stackrel{\P_{\mu^*}}{\to}0.$
Since $|\zeta_p|\le 2$, dominated convergence theorem implies
$\E_{\P_{\mu^*}}[\zeta_p^2]\to 0$. Consequently, Jensen's inequality gives  $$\E_{\P_{\mu^*}}\Bigg[\Big(\E_{\P_{\mu^*}}[\zeta_p|\mathbf{y},I]\Big)^2\Big|\;\mathbf{y}\Bigg]\le \E_{\P_{\mu^*}}\Bigg[\Big(\E_{\P_{\mu^*}}[\zeta_p^2|\mathbf{y},I]\Big)\Big|\;\mathbf{y}\Bigg]=\E_{\P_{\mu^*}}[\zeta_p^2 \mid \mathbf{y}]\stackrel{\P_{\mu^*}}{\to}0.$$
Noting that  $$\E_{\P_{\mu^*}}[\zeta_p \mid \mathbf{y},I=i]=\E_{\P_{\mu^*}}[b_i \mid \mathbf{y}]-u^*_i=\E_{\P_{\mu^*}}[\beta_i \mid \mathbf{y}]-u^*_i=\widetilde{u}_i-u^*_i,$$
the last display is equivalent to 
$\|\widetilde{\bf u}-\mathbf{u}^*\|_2^2=o_{\P_{\mu^*}}(p),$
which is the desired conclusion. 
\\

(ii)
 To begin, set
 $f(\bbeta):=-\frac{1}{2}\bbeta^\top A\bbeta+{\bf w}^\top \bbeta$, where $A$ and ${\bf w}$ are as in
 Definition \ref{def:Z}, and note that the posterior distribution of $\bbeta$ given ${\bf y}$ is given by
 $$\frac{d\P_{\mu^*}(\cdot \mid \mathbf{y})}{\prod_{i=1}^pd\mu_{0,d_i}}(\bbeta)=\frac{1}{Z_p(\mathbf{w},\mu^*)} \exp(f(\bbeta)).$$
Then we claim that
\begin{align}\label{eq:claim11}
\E_{\P_{\mu^*}}[f(\bbeta) \mid \mathbf{y}]=f(\widetilde{\mathbf{u}}) +o_{\P_{\mu^*}}(p). 
\end{align}
We now complete the proof of (ii), deferring the proof of the claim. To this effect, set 
$\mathbb{Q}:=\prod_{i=1}^p \mu^*_{\widetilde{\tau}_i,d_i},$
and note that with $\mathcal{M}_p(\cdot,\cdot,\cdot)$ as defined in \eqref{eq:M_mean_param}, a direct calculation gives
\begin{align*}
    \dkl(\P_{\mu^*}(\cdot \mid \mathbf{y})\| \mathbb{Q})&
    =\E_{\P_{\mu^*}}[f(\bbeta) \mid \mathbf{y}]-\sum_{i=1}^p\Big[\widetilde{\tau}_i\E_{\P_{\mu^*}}[\beta_i \mid \mathbf{y}]- c_{\mu^*}({\widetilde{\tau}}_i,d_i)+c_{\mu^*}(0,d_i)\Big]-\log Z_p(\mathbf{w},\mu^*)\\
    &=f(\widetilde{\bf u})-\sum_{i=1}^p\Big[\widetilde{\tau}_i\E_{\P_{\mu^*}}[\beta_i \mid \mathbf{y}]- c_{\mu^*}({\widetilde{\tau}}_i,d_i)+c_{\mu^*}(0,d_i)\Big]-\log Z_p(\mathbf{w},\mu^*)+o_{\P_{\mu^*}}(p)\\
    &= \mathcal{M}_p(\widetilde{\bf u}, \mathbf{w},\mu^*) -\log Z_p(\mathbf{w},\mu^*)+o_{\P_{\mu^*}}(p)\\
    &\le \sup_{{\bf u}\in [-1,1]^p} \mathcal{M}_p({\bf u}, \mathbf{w},\mu^*) -\log Z_p(\mathbf{w},\mu^*)+o_{\P_{\mu^*}}(p)=o_{\P_{\mu^*}}(p).
\end{align*}
    In the above display, the second line uses \eqref{eq:claim11}, and the last equality uses Theorem \ref{lem:approx} part (a). 
    
    To complete the proof of part (ii), we need to verify \eqref{eq:claim11}. To this effect, note that
    $$|\bbeta^\top A \mathbf{b} -\bbeta^\top A\widetilde{\bf u}|\le \|\bbeta\|_2 \|A\|_2 \|\mathbf{b}-\widetilde{\bf u}\|_2\le \sqrt{p} \|A\|_2 \|\mathbf{b}-\widetilde{\bf u}\|_2= o_{\P_{\mu^*}}(p),$$
    where we use \eqref{eq:eigen}, along with the bound $\|\mathbf{b}-\widetilde{\mathbf{u}}\|_2^2=o_{\P_{\mu^*}}(p)$ (follows by combining \eqref{eq:bconc} along with part (i)). An application of the dominated convergence theorem 
    on the conditional probability space $\P_{\mu^*}(\cdot \mid \mathbf{y})$ then 
    gives $$ \E_{\mathbb{P}_{\mu^*}}[\bbeta^\top A \mathbf{b} \mid \mathbf{y}] -\E_{\mathbb{P}_{\mu^*}}[\bbeta^\top A\widetilde{\mathbf u} \mid \mathbf{y}]=o_{\P_{\mu^*}}(p).$$
    Also, a direct computation gives
    $$\mathbb{E}_{\P_{\mu^*}}[\bbeta^\top A \bbeta \mid \mathbf{y}]=\mathbb{E}_{\P_{\mu^*}}[\bbeta^\top A \mathbf{b} \mid \mathbf{y}], \quad \text{ and }\quad \mathbb{E}_{\P_{\mu^*}}[\bbeta^\top A \widetilde{\bf u} \mid \mathbf{y}]=\widetilde{\bf u}^\top A \widetilde{\bf u}.$$
   Combining the above two displays, we conclude
    $$\E_{\P_{\mu^*}}[\bbeta^\top A \bbeta \mid \mathbf{y}]-\widetilde{\bf u}^\top A \widetilde{\bf u}=o_{\P_{\mu^*}}(p),$$
    from which \eqref{eq:claim11} follows on recalling that $\E_{\P_{\mu^*}}[\bbeta \mid \mathbf{y}]=\widetilde{\bf u}.$
 \\

(iii) 
The first conclusion follows from part (ii) along with Marton's transportation-entropy inequality (see~\citet{MR0838213,MR1404531}). Proceeding to verify the second assertion, using the first conclusion along with triangle inequality, it suffices to show that
\begin{align}\label{eq:aa0} \frac{1}{p} \, d_{W_1}\Big(\prod_{i=1}^p\mu^*_{\widetilde{\tau}_i,d_i},\prod_{i=1}^p \mu^{(p)}_{\tau_i^{(p)},d_i}  \Big)=\frac{1}{p}\sum_{i=1}^{p} d_{W_1}\Big(\mu^*_{\widetilde{\tau}_i,d_i}, \mu^{(p)}_{\tau_i^{(p)},d_i}  \Big)=o_{\P_{\mu^*}}(1).
\end{align}
To this effect, note that since $\mathbf{u}^{(p)}$ is a global optimizer of $\mathcal{M}_p(\cdot, \mathbf{w},\mu^{(p)})$, an application of Lemma \ref{lem:dev_bounds} gives
\begin{align}\label{eq:claim21}
|\{i\in [p]:|\tau^{(p)}_i|>K\}|=\er(p).
\end{align}
Also, if $\mathbf{v}^* = (v_1^*,\ldots, v_p^*)$ denotes a global optimizer of the function $\mathcal{M}_p(\cdot,\mathbf{w},\mu^*)$, then another application of Lemma \ref{lem:dev_bounds} gives
\begin{align*}
    |\{i\in [p]:|h_{\mu^*,d_i}(v_i^*)|>K\}|=\er(p).
\end{align*}
By the well-separation property of $\mathcal{M}_p(\cdot,\mathbf{w},\mu^*)$ we have $\|\mathbf{v}^*-\mathbf{u}^*\|_2^2=o_{\P_{\mu^*}}(p)$, which along with part (i) and triangle inequality gives $\|\mathbf{v}^*-\widetilde{\mathbf{u}}\|_2^2=o_{\P_{\mu^*}}(p)$. Also recall that $\widetilde{\tau}_i=h_{\mu^*,d_i}(\widetilde{u}_i)$, and note that the function $h_{\mu^*,\cdot}(\cdot)$ is   continuous on the (non-compact) set
$(\inf{\rm supp}(\mu), \sup {\rm supp} (\mu))\times [C_1,C_2]$ (by part (iv) of Lemma \ref{lem:unif_cont}), and so the set $$\{(v,d)\in (\inf{\rm supp}(\mu^*), \sup {\rm supp} (\mu^*))\times [C_1\times C_2]: |h_{\mu^*,d}(v)|\le K\}$$
is compact. In particular, there exists $\eta=\eta_{K,\mu^*}$ such that $$|h_{\mu^*,d}(v)|\le K\Rightarrow  \inf{\rm supp}(\mu^*)-\eta\le v\le  \sup {\rm supp} (\mu^*)+\eta.$$
Fix $i\in [p]$ such that \[|h_{\mu^*,d_i}(v_i^*)|\le K,\text{ and }|v_i^*-\widetilde{u}_i|\le \frac{\eta}{2}.\]
Then we have \[\inf{\rm supp}(\mu^*)-\frac{\eta}{2}\le \widetilde{u}_i\le  \sup {\rm supp} (\mu^*)+\frac{\eta}{2} \]
 The above display gives, on recalling that $\widetilde{\tau}_i=h_{\mu^*,d_i}(\widetilde{u}_i)$, gives
\begin{align}\label{eq:claim22}
|\{i\in [p]:|\widetilde{\tau}_i|>K\}|=\er(p).
\end{align}
Using \eqref{eq:claim21} and \eqref{eq:claim22} and setting
$$\mathcal{S}(p,K):=\{i\in [p]:|\tau^{(p)}_i|\le K, |\widetilde{\tau}_i|\le K\},$$
we have $|\mathcal{S}(p,K)^c|=\er(p)$, and so 
\begin{align}\label{eq:aa1}\sum_{i=1}^{p} d_{W_1}\Big(\mu^*_{\widetilde{\tau}_i,d_i}, \mu^{(p)}_{\tau_i^{(p)},d_i}  \Big)\le \er(p)+\sum_{i\in \mathcal{S}(p,K)}d_{W_1}\Big(\mu^*_{\widetilde{\tau}_i,d_i}, \mu^{(p)}_{\tau_i^{(p,K)},d_i}  \Big),
\end{align}
where $\tau_i^{(p,K)}=\tau_i1\{|\tau_i|\le K\}$ is as in Lemma \ref{lem:approximations}. 
We now claim that for any $K>0$ we have
\begin{align}\label{eq:claim23}
    \sum_{i\in \mathcal{S}(p,K)}(\widetilde{\tau}_i-\tau^{(p,K)}_i)^2=o_{\P_{\mu^*}}(p).
\end{align}
We first complete the proof, given \eqref{eq:claim23}. Given \eqref{eq:claim23}, there exists a sequence of positive reals $\varepsilon_p$ satisfying $\lim_{p\to\infty}\varepsilon_p=0$, such that
$$\Big|\{i\in \mathcal{S}(p,K):|\widetilde{\tau}_i-\tau^{(p,K)}_i|>\varepsilon_p\}\Big|=o_{\P_{\mu^*}}(p).$$
Also, with $\mathcal{P}$ denoting the parameter space, the function $(\mu,\tau,d)\mapsto \mu_{\tau,d}$ is continuous on the compact set $\mathcal{P}\times [-K,K]\times [C_1,C_2]$ with respect to weak topology (invoking Lemma \ref{lem:unif_cont} part (ii)), and hence the Wasserstein metric $d_{W_1}(\cdot,\cdot)$, and so it is uniformly continuous. Consequently, the above display, along with the assumption $\mu^{(p)}\stackrel{w}{\to}\mu^*$, gives
$$\sum_{i\in \mathcal{S}(p,K)}d_{W_1}\Big(\mu^*_{\widetilde{\tau}_i,d_i}, \mu^{(p)}_{\tau_i^{(p,K)},d_i}  \Big)=o_{\P_{\mu^*}}(p).$$
Combining \eqref{eq:aa1} and the above display, the desired conclusion \eqref{eq:aa0} follows.
\\

To complete the proof, it remains to verify the claim \eqref{eq:claim23}. For this, recalling that $$\widetilde{\bm \tau}=h_{\mu^*,{\bf d}}(\widetilde{\bf u})\qquad \text{ and }\qquad {\bm \tau}^{(p,K)}= h_{\mu^*,{\bf d}}({\bf u}_*^{(p,K)}), $$
it suffices to show that $\|\widetilde{\bf u}-{\bf u}_*^{(p,K)}\|_2^2=o_{\P_{\mu^*}}(p).$ But this follows from Lemma \ref{lem:approximations} and part (i), combined with triangle inequality.

\subsection{Proof of Corollary \ref{cor:takeaways}}


\begin{itemize}
\item[(i)]
Without loss of generality assume $L=1$. 
By Theorem \ref{thm:inference} (iii), there exists a coupling of the random variables $\bbeta\sim \mathbb{P}_{\mu^*}(\cdot \mid \mathbf{y})$ and $\boldsymbol{\zeta}\sim\prod_{i=1}^{p}\mu^{(p)}_{\tau_i^{(p)},d_i}$ under which
$\|\bbeta-\boldsymbol{\zeta}\|_1=o_{\P_{\mu^*}}(p).$ Then we have 
\begin{align*}
&\Big|\frac{1}{p}\sum_{i=1}^pf_{p,i}(\beta_i)-\frac{1}{p}\sum_{i=1}^p \E_{\mu^{(p)}_{\tau_i^{(p)},d_i}}[f_{p,i}(\zeta_i)]\Big|\\
\le& \Big|\frac{1}{p}\sum_{i=1}^pf_{p,i}(\beta_i)-\frac{1}{p}\sum_{i=1}^pf_{p,i}(\zeta_i) \Big|+\Big|\frac{1}{p}\sum_{i=1}^pf_{p,i}(\zeta_i)-\frac{1}{p}\sum_{i=1}^p \E_{\mu^{(p)}_{\tau_i^{(p)},d_i}}[f_{p,i}(\zeta_i)]\Big|\\
\le &\frac{1}{p}\sum_{i=1}^p|\beta_i-\zeta_i|+\Big|\frac{1}{p}\sum_{i=1}^pf_{p,i}(\zeta_i)-\frac{1}{p}\sum_{i=1}^p \E_{\mu^{(p)}_{\tau_i^{(p)},d_i}}[f_{p,i}(\zeta_i)]\Big|.
\end{align*}
The first term above is $o_{\P_{\mu^*}}(1)$ by construction of the coupling, and the second term above is $o_{\P_{\mu^*}}(1)$ by
the bounded difference inequality. 
This completes the proof.

\item[(ii)] Fix $\alpha \in (0,1)$, and let { $\mathcal{I}_j = (q_j^{(\alpha/2)}, q_j^{(1-\alpha/2)} )$}, where $q_j^{(\alpha/2)}$ and $q_j^{(1-\alpha/2)}$ are the $\alpha/2$ and $1-\alpha/2$ quantiles of $\mu^{(p)}_{\tau_j^{(p)},d_j}$. 
For $\varepsilon>0$, let $\mathcal{I}_j^{\varepsilon}$ be a $\varepsilon$-fattening of $\mathcal{I}_j$. Then there exists a bounded, continuous Lipschitz function $f_{p,j}$ (whose Lipschitz constant depends only on $\varepsilon$) such that $f_{p,j}$ is $1$ on $\bar{\mathcal{I}}_j$, and $0$ on $(\mathcal{I}_j^{\varepsilon})^c$. 
Observe now that, by construction,  $$\sum_{j=1}^{p} \mathbf{1}(\beta_j \in \mathcal{I}_j^{\varepsilon}) \geq \sum_{j=1}^{p} f_{p,j}(\beta_j),\quad \text{ and } \quad \mathbb{E}_{\mu^{(p)}_{\tau^{(p)}_j,d_j}}[f_{p,j}(\beta_j)] \geq \mu^{(p)}_{\tau^{(p)}_j,d_j}(\beta_j \in \mathcal{I}_j) \geq 1- \alpha.$$ The desired conclusion now follows from part (i) and the display above. 


\item[(iii)] The Bayes optimal estimator under the $L^2$ loss is the vector of marginal expectations $\int \bbeta \; \mathrm{d} \mathbb{P}_{\mu^*}(\beta \mid  \mathbf{y})$, with 
 minimum mean-square error (MMSE) 
\begin{align}
\mathrm{MMSE} = \frac{1}{p} \sum_{j=1}^{p} \mathbb{E}_{\P_{\mu^*}} \Big[ \int \Big(\beta_j - \int \beta_j \,\mathrm{d}\mathbb{P}_{\mu^*}(\beta \mid  \mathbf{y} ) \Big)^2 \mathrm{d}\mathbb{P}_{\mu^*}(\bbeta \mid \mathbf{y}) \Big]. \nonumber 
\end{align} 
It suffices to establish that  $$\frac{1}{p} \sum_{j=1}^{p} \mathbb{E}_{\P_{\mu^*}}\Big[ \int (\beta_j - \hat{\beta}_j)^2 \mathrm{d}\mathbb{P}_{\mu^*}(\beta \mid \mathbf{y}) \Big] - \mathrm{MMSE} \to 0.$$

Using part (i) of this corollary (with $f_{p,i}(\beta_i) = (\beta_i - \hat \beta_i)^2/2$) we have, 
\begin{align}
 \Big| \frac{1}{p} \sum_{j=1}^{p} (\beta_j - \hat{\beta}_j)^2  - \frac{1}{p} \sum_{j=1}^{p} \mathrm{Var}_{\mu^{(p)}_{\tau_j^{(p)},d_j}}(\beta_j)  \Big|\stackrel{\mathbb{P}_{\mu^*}}{\to} 0. \nonumber 
\end{align} 
Using the dominated convergence theorem, we conclude, 
\begin{align}
\frac{1}{p} \sum_{j=1}^{p} \mathbb{E}_{\P_{\mu^*}} \Big[ \int (\beta_j - \hat{\beta}_j)^2 \mathrm{d}\mathbb{P}_{\mu^*}(\beta_j | \mathbf{y} ) \Big] - \frac{1}{p} \sum_{j=1}^{p} \mathbb{E}_{\P_{\mu^*}} \Big[  \mathrm{Var}_{\mu^{(p)}_{\tau_j^{(p)},d_j}}(\beta_j) \Big]  \to 0. \label{eq:int3} 
\end{align}
On the other hand, let $\bbeta^{(1)}, \bbeta^{(2)}$ be two independent samples from the posterior $\mathbb{P}_{\mu^*}(\cdot \mid \mathbf{y})$. Again invoking part (i) of this corollary, we get
\begin{align}
 \Big| \frac{1}{2 p} \sum_{j=1}^{p} (\beta_{j}^{(1)} - \beta_{j}^{(2)})^2 - \frac{1}{2p} \sum_{j=1}^{p} \mathbb{E}_{\mu^{(p)}_{\tau_j^{(p)},d_j}} (\beta_{j}^{(1)} - \beta_{j}^{(2)})^2  \Big|\stackrel{\P_{\mu^*}}{ \to} 0. \nonumber 
\end{align} 
Using the dominated convergence theorem, we get
\begin{align} 
\frac{1}{p} \sum_{j=1}^{p} \mathbb{E}_{\mu^*} \Big[ \int (\beta_j - \int \beta_j \mathrm{d}\mathbb{P}_{\mu^*}(\beta_j \mid \mathbf{y}))^2 \mathrm{d}\mathbb{P}_{\mu^*}(\beta_j \mid \mathbf{y})  \Big] - \frac{1}{p} \sum_{j=1}^{p} \mathbb{E}_{\mu^*} \Big[  \mathrm{Var}_{\mu^{(p)}_{\tau_j^{(p)},d_j}}(\beta_j) \Big] \to 0. \label{eq:int4} 
\end{align} 

The proof is complete upon combining \eqref{eq:int3} and \eqref{eq:int4}. \qed

\end{itemize} 



\appendix

\section{Proof of results from Section \ref{sec:Proofs}}\label{Appendix-A}

In this section we will prove Lemma \ref{lem:non-asymptotic-stability}, Lemma \ref{lem:exp_concentration}, Lemma \ref{lemma:covering}, and Example \ref{eg:eigen}. For proving Lemma \ref{lem:non-asymptotic-stability}, we will need the following simple lemma.
\begin{lemma}\label{lem:lower_bound}
With $\Psi_{\mu,\upsilon^2}(\cdot)$ as in \eqref{eq:Upsilon}, we have
 $$\Psi_{\mu,\upsilon^2}(\theta)\ge  \frac{1}{\sqrt{2\pi \upsilon^2}} e^{-\frac{9}{8\upsilon^2}(\theta^2+4)  }.$$
\end{lemma}

 \begin{proof}[Proof of Lemma \ref{lem:lower_bound}]
 If $\theta\in [-2,2]$, then for $\beta\in [-1,1]$, we have $(\theta-\beta)^2\le 9$, and so we have the bound
 \begin{align*}
   \int_{-1}^1e^{-\frac{1}{2\upsilon^2}(\theta-\beta)^2}d\mu(\beta)\ge e^{-9/(2\upsilon^2)}.
 \end{align*}
 On the other hand, if $|\theta|>2$, then for $\beta\in [-1,1]$ we have $(\theta-\beta)^2\le \frac{9\theta^2}{4}$, and so
  \begin{align*}
   \int_{-1}^1 e^{-\frac{1}{2
\upsilon^2}(\theta-\beta)^2}d\mu(\beta)\ge e^{-9\theta^2 /(8\upsilon^2)}.
 \end{align*}
To get a lower bound that holds for any $\theta \in \mathbb{R}$, we multiply the lower bounds obtained in the two cases above; this is a valid general lower bound as each bound is less than 1.
\end{proof}

\begin{proof}[Proof of Lemma~\ref{lem:non-asymptotic-stability}]
In the proof of this lemma we will need to modify $\Sigma$, and so we will explicitly point out the dependence of $\Sigma$ in $\widetilde{m}_{\mu,\Sigma}(\cdot)$ and $\widetilde{g}({\bf v},\mu,\Sigma)$.
Computing the densities from the two different representations of the random variable ${\bf z}$ in \eqref{eq:modified_sequence}, it follows that for any $\mu\in \mathcal{P}$ we have $\widetilde{m}_{\mu,\Sigma}({\bf v})=\widetilde{m}_{\mu\star N(0,\upsilon^2), \bar{\Sigma}}({\bf v})$. Using \eqref{eq:gg}, this gives 
\begin{align*}
    \widetilde{g}(\mathbf{v},\mu_1,\Sigma) - \widetilde{g}(\mathbf{v}, \mu_2, \Sigma) 
    = \widetilde{g}(\mathbf{v}, \mu_1 \star N(0,\upsilon^2),\bar{\Sigma}) - \widetilde{g} (\mathbf{v}, \mu_2 \star N(0, \upsilon^2),\bar{\Sigma})\end{align*}
   Using the formula for $\widetilde{g}$ in \eqref{eq:gh}, and noting that $\mu\star N(0,\upsilon^2)$ has density $\Psi_{\mu,\upsilon^2}$ (see \eqref{eq:Upsilon}) we have
   \begin{align}\label{eq:gh33}
   \notag e^{\widetilde{g}({\bf v},\mu_1\star N(0,\upsilon^2),\bar{\Sigma})}=&
  \int e^{-\widetilde{h}(\mathbf{v},{\bm \theta},\bar{\Sigma})} \prod_{i=1}^{p} \Psi_{\mu_1,\upsilon^2}(\theta_i) d{\bm \theta}\\
   \notag =& \int_{\widetilde{h}(\mathbf{v}, {\bm \theta}, \bar{\Sigma})\leq CKp}e^{-\widetilde{h}(\mathbf{v},{\bm \theta}, \bar{\Sigma})} \prod_{i=1}^{p} \Psi_{\mu_1,\upsilon^2}(\theta_i) d{\bm \theta}\\
    & \hspace{0.5in} + \int_{\widetilde{h}(\mathbf{v}, {\bm \theta}, \bar{\Sigma})> CKp}e^{-\widetilde{h}(\mathbf{v},{\bm \theta}, \bar{\Sigma})} \prod_{i=1}^{p} \Psi_{\mu_1,\upsilon^2}(\theta_i) d{\bm \theta},
\end{align}
where $K$ is as in the statement of the Lemma, and $C$ is a constant to be specified later on.
We note that if $\widetilde{h}(\mathbf{v}, {\bm \theta},\bar{\Sigma}) \leq  CK p $, using \eqref{eq:eigen33} and \eqref{eq:gh}  we have 
\begin{align}
    \frac{C_1}{2}\|\mathbf{v} - {\bm \theta}\|_2^2 \leq \widetilde{h}(\mathbf{v}, {\bm \theta},\bar{\Sigma}) \leq C Kp,\nonumber 
\end{align}
which, in turn, implies that $$\|{\bf v}-{\bm \theta}\|_2^2\le \frac{2CKp}{C_1}\qquad \Rightarrow \qquad \|{\bm \theta}\|_2^2\le 2Kp\Big[1+ \frac{2 C}{C_1}\Big],$$
where the last inequality uses the bound $\|{\bf v}\|_2^2\le Kp$. Then, using \eqref{eq:Psi} the first term in the RHS of \eqref{eq:gh33} can be bounded as follows:
\begin{align*}
   & \int_{\widetilde{h}(\mathbf{v}, {\bm \theta},\bar{\Sigma})\leq CKp}e^{-\widetilde{h}(\mathbf{v},{\bm \theta},\bar{\Sigma})} \prod_{i=1}^{p} \Psi_{\mu_1,\upsilon^2}(\theta_i) d{\bm \theta}\\
   \le &\int_{\widetilde{h}(\mathbf{v}, {\bm \theta},\bar{\Sigma})\leq CKp}e^{-\widetilde{h}(\mathbf{v},{\bm \theta},\bar{\Sigma})} \prod_{i=1}^{p} \Psi_{\mu_2,\upsilon^2}(\theta_i) e^{(p+\|{\bm \theta}\|_2^2)d_{\upsilon^2}(\mu_1,\mu_2)} d{\bm \theta}\\
    &\leq e^{\widetilde{g}({\bf v},\mu_2\star N(0,\upsilon^2),\bar{\Sigma})} e^{p d_{\upsilon^2}(\mu_1,\mu_2)\Big(1+2K \Big[1+\frac{2C}{C_1}\Big]\Big)},
\end{align*}
where the last display uses the bound on $\|{\bm \theta}\|_2^2$ obtained above. Proceeding to control the second term in the RHS of \eqref{eq:gh33}, note that  
\begin{align*}
    \int_{\widetilde{h}(\mathbf{v}, \boldsymbol{\theta}, \bar{\Sigma})> CKp} e^{-\widetilde{h}(\mathbf{v}, \boldsymbol{\theta},\bar{\Sigma}) } \prod_{i=1}^{p} \Psi_{\mu_1, \upsilon^2}(\theta_i) d\boldsymbol{\theta} \leq e^{-CKp}\int_{\R^p} \prod_{i=1}^{p} \Psi_{\mu_1, \upsilon^2}(\theta_i) d\boldsymbol{\theta}=e^{-CKp}.  
\end{align*}
Combining the last two displays, \eqref{eq:gh33} gives
\begin{align*}
    e^{\widetilde{g}({\bf v},\mu_1\star N(0,\upsilon^2),\bar{\Sigma})} \le e^{\widetilde{g}({\bf v},\mu_2\star N(0,\upsilon^2),\bar{\Sigma})} e^{p d_{\upsilon^2}(\mu_1,\mu_2)\Big(1+2K \Big[1+\frac{2C}{C_1}\Big]\Big)}+e^{-CKp}.
\end{align*}
Dividing both sides by $e^{\widetilde{g}({\bf v},\mu_2\star N(0,\upsilon^2),\bar{\Sigma})} $ and taking $\log$ gives
\begin{align}\label{eq:log}
   \notag& \widetilde{g}({\bf v},\mu_1\star N(0,\upsilon^2),\bar{\Sigma})-\widetilde{g}({\bf v},\mu_2\star N(0,\upsilon^2),\bar{\Sigma})\\
   \notag\le& \; \log \left\{ e^{p d_{\upsilon^2}(\mu_1,\mu_2)\Big(1+2K \Big[1+\frac{2C}{C_1}\Big]\Big)}+e^{-CKp-\widetilde{g}({\bf v},\mu_2\star N(0,\upsilon^2), \bar{\Sigma})}\right\}\\
   \notag\le & \; p d_{\upsilon^2}(\mu_1,\mu_2)\Big(1+2K \Big[1+\frac{2C}{C_1}\Big]\Big)+\log \left\{1+ e^{-CKp-\widetilde{g}({\bf v},\mu_2\star N(0,\upsilon^2), \bar{\Sigma}) }\right\}\\
   \le & \; p d_{\upsilon^2}(\mu_1,\mu_2)\Big(1+2K \Big[1+\frac{2C}{C_1}\Big]\Big) +e^{-CKp-\widetilde{g}({\bf v}, \mu_2\star N(0,\upsilon^2),\bar{\Sigma})}.
\end{align}
Finally, using \eqref{eq:eigen33} and \eqref{eq:gh} we get $$\widetilde{h}({\bf v},{\bm \theta},\bar{\Sigma})\le C_2\|{\bf v}-{\bm \theta}\|_2^2\le 2C_2(Kp+\|{\bm \theta}\|_2^2),$$
which gives 
\begin{align*}
   \exp(\widetilde{g}({\bf v},\mu_2\star N(0,\upsilon^2), \bar{\Sigma}))= &\int_{\R^p} e^{-\widetilde{h}({\bf v},{\bm \theta},\bar{\Sigma})} \prod_{i=1}^p \Psi_{\mu_2, \upsilon^2}(\theta_i)d{\bm \theta}\\
   \ge & \; \frac{1}{(2 \pi \upsilon^2)^{p/2}} \int_{\R^p} e^{-2C_2(Kp+\|{\bm \theta}\|_2^2)}e^{-\frac{9}{8\upsilon^2}(\|{\bm \theta}\|_2^2+4p)}d{\bm \theta}\\
   =& \;\frac{1}{(2 \pi \upsilon^2)^{p/2}} e^{-\frac{9p}{2\upsilon^2}-2C_2Kp} \int_{\R^p} e^{-\Big(2C_2+\frac{9}{8\upsilon^2}\Big) \|{\bm \theta}\|_2^2}d{\bm \theta}\\
   =& \; \frac{1}{(2 \pi \upsilon^2)^{p/2}} e^{-\frac{9p}{2\upsilon^2}-2C_2Kp} \Big(\frac{\pi}{2C_2+\frac{9}{8\upsilon^2}}\Big)^{\frac{p}{2}}
\end{align*}
where the inequality uses the above display along with Lemma \ref{lem:lower_bound}. Thus, recalling that $\upsilon^2=(2C_2)^{-1}$, choosing $C$ large enough (depending only on $K,C_2$), \eqref{eq:log} gives
\begin{align*}
     \widetilde{g}({\bf v},\mu_1\star N(0,\upsilon^2), \bar{\Sigma})- \widetilde{g}({\bf v},\mu_2\star N(0,\upsilon^2), \bar{\Sigma})\le p d_{\upsilon^2}(\mu_1,\mu_2)\Big(1+2K \Big[1+\frac{2C}{C_1}\Big]\Big) +e^{-p}.
\end{align*}
The desired bound of the lemma is immediate from this, on noting that the roles of $\mu_1,\mu_2$ are symmetric.
\end{proof}

\begin{proof}[Proof of Lemma~\ref{lem:exp_concentration}]
We note that 
\begin{align}
    g(\boldsymbol{\beta}, {\bm \varepsilon}, \mu) - g(\boldsymbol{\beta}, {\bm \varepsilon}, \mu^*) &= \log \frac{ \int e^{-h(\boldsymbol{\beta},{\bm \varepsilon}, {\bm \theta})} d\mu^{\otimes p} ({\bm \theta}) }{ \int e^{-h(\boldsymbol{\beta},{\bm \varepsilon}, {\bm \theta})}  d{\mu^*}^{\otimes p}({\bm \theta}) } \nonumber \\
    &= \log \frac{\int e^{{\bm \theta}^{\top}\Sigma^{-1} (\bar{\boldsymbol{\beta}} + {\bm \varepsilon}_2) - \frac{1}{2} {\bm \theta}^{\top} \Sigma^{-1} {\bm \theta}   } d\mu^{\otimes p}({\bm \theta})   }{
    \int e^{{\bm \theta}^{\top}\Sigma^{-1} (\bar{\boldsymbol{\beta}} + {\bm \varepsilon}_2) - \frac{1}{2} {\bm \theta}^{\top} \Sigma^{-1} {\bm \theta}   } d{\mu^*}^{\otimes p}({\bm \theta})  
    }, \nonumber
\end{align}
where we use the reparametrization \eqref{eq:modified_sequence}. Set $\bar{{\bm \varepsilon}}_2 := \bar{\Sigma}^{-1/2} {\bm \varepsilon}_2\sim N({\bf 0}, {\bf I})$,  and define a probability distribution on $[-1,1]^{p}$  by the following Radon-Nikodym derivative:
\begin{align}
\frac{dH(\bar{\boldsymbol{\beta}},\bar{{\bm \varepsilon}}_2,\mu)}{d\mu^{\otimes p}}({\bm \theta}) = \frac{e^{ {\bm \theta}^{\top} \Sigma^{-1}( \bar{\boldsymbol{\beta}} + \bar{\Sigma}^{1/2}\bar{{\bm \varepsilon}}_2 ) - \frac{1}{2} {\bm \theta}^{\top} \Sigma^{-1} {\bm \theta}  }}{\int e^{{\bm \alpha}^{\top} \Sigma^{-1}( \bar{\boldsymbol{\beta}} + \bar{\Sigma}^{1/2}\bar{{\bm \varepsilon}}_2 ) - \frac{1}{2} {\bm \alpha}^{\top} \Sigma^{-1} {\bm \alpha} } d\mu^{\otimes p}({\bm \alpha}) }. \nonumber 
\end{align}
Consider now the mapping $(\bar{\boldsymbol{\beta}}, \bar{{\bm \varepsilon}}_2) \to  g (\boldsymbol{\beta}, {\bm \varepsilon}, \mu) - g (\boldsymbol{\beta}, {\bm \varepsilon}, \mu^*)$. By a direct computation, we have 
\begin{align}
    &\frac{\partial}{\partial \bar{\boldsymbol{\beta}}} [g(\boldsymbol{\beta}, {\bm \varepsilon}, \mu) - g(\boldsymbol{\beta}, {\bm \varepsilon}, \mu^*)] = \mathbb{E}_{{\bm \theta} \sim H(\bar{\boldsymbol{\beta}}, \bar{{\bm \varepsilon}}_2,\mu)}[\Sigma^{-1} {\bm \theta}] - \mathbb{E}_{{\bm \theta} \sim H(\bar{\boldsymbol{\beta}},  \bar{{\bm \varepsilon}}_2 ,\mu^* )}[\Sigma^{-1} {\bm \theta}] , \nonumber \\
    &\frac{\partial}{\partial \bar{{\bm \varepsilon}}_2 }[g(\boldsymbol{\beta}, {\bm \varepsilon}, \mu) - g(\boldsymbol{\beta}, {\bm \varepsilon}, \mu^*)] = \mathbb{E}_{{\bm \theta} \sim H(\bar{\boldsymbol{\beta}}, \bar{{\bm \varepsilon}}_2,\mu)}[ \bar{\Sigma}^{1/2}\Sigma^{-1}  {\bm \theta}] - \mathbb{E}_{{\bm \theta} \sim H(\bar{\boldsymbol{\beta}},  \bar{{\bm \varepsilon}}_2 ,\mu^* )}[\bar{\Sigma}^{1/2}\Sigma^{-1}{\bm \theta}]. \nonumber
\end{align}
Using \eqref{eq:eigen2} we have $\lambda_{\max}(\Sigma^{-1}) \leq C_2$. Also, recalling that $\bar{\Sigma}=\Sigma-\upsilon^2{\bf I}$ gives
\begin{align}
    \| \bar{\Sigma}^{1/2}\Sigma^{-1}\|_2 = \| (\mathbf{I} - \upsilon^2 {\Sigma}^{-1})^{1/2} \Sigma^{-1/2}\|_2 \leq (1+ \upsilon^2 \lambda_{\max}(\Sigma^{-1}) ) \lambda_{\max}(\Sigma^{-1/2}) \nonumber 
 \end{align}
which depends only on $C_2$. Finally, noting that $\mu$, $\mu^*$ are supported on $[-1,1]^{p}$, we have, 
\begin{align}
  &\Big\| \frac{\partial}{\partial \bar{\boldsymbol{\beta}}}[ g (\boldsymbol{\beta}, {\bm \varepsilon}, \mu) - g (\boldsymbol{\beta}, {\bm \varepsilon}, \mu^*)]\Big\|_2^2 \leq C p, \nonumber  \\
  &\Big\| \frac{\partial}{\partial \bar{{\bm \varepsilon}}_2}[ g (\boldsymbol{\beta}, {\bm \varepsilon}, \mu) - g (\boldsymbol{\beta}, {\bm \varepsilon}, \mu^*)]\Big\|_2^2 \leq C p  \nonumber 
\end{align}
for some constant $C:= C(C_2)>0$ independent of $p$. The proof is now complete upon noting that both $\bar{\boldsymbol{\beta}}$ and $\bar{{\bm \varepsilon}}$ satisfy the Log-Sobolev inequality \citep[Corollary 1]{chen2021dimension}.
\end{proof}

\begin{proof}[Proof of Lemma~\ref{lemma:covering}] 
Fix $\varepsilon>0$, and let $\delta=\exp\Big(-\frac{c}{\varepsilon^2}\Big)$, where $c$ is to be chosen later. Let $\mathcal{N}$ be a $\delta$-net of $\mathcal{P}$ under the $L^{\infty}$ metric (see~\eqref{eq:L_infty}). For any $\tilde{\mu} \in \mathcal{P}$, let $\mu \in \mathcal{N}$ be such that $\| \Psi_{\mu,\upsilon^2} - \Psi_{\tilde{\mu}, \upsilon^2}\|_{\infty} < \delta$. Then, for any $K>0$ we have 
\begin{align}
\label{eq:bound_sup_norm}
    d_{\upsilon^2}(\mu, \tilde{\mu}) &\leq \sup_{\theta \in [-K,K]} \frac{1}{1+\theta^2} \Big| \log \frac{\Psi_{\mu,\upsilon^2}(\theta)}{\Psi_{\tilde{\mu},\upsilon^2}(\theta)} \Big| + \sup_{|\theta| > K} \frac{1}{1+\theta^2} \Big| \log \frac{\Psi_{\mu,\upsilon^2}(\theta)}{\Psi_{\tilde{\mu},\upsilon^2}(\theta)} \Big|\nonumber \\
    &\leq  \sup_{\theta \in [-K,K]} \Big| \log \frac{\Psi_{\mu,\upsilon^2}(\theta)}{\Psi_{\tilde{\mu},\upsilon^2}(\theta)} \Big| + \sup_{|\theta| > K} \frac{1}{1+\theta^2} \Big| \log \frac{\Psi_{\mu,\upsilon^2}(\theta)}{\Psi_{\tilde{\mu},\upsilon^2}(\theta)} \Big|.  
\end{align}
Using the bound $|\log x-\log y|\le \frac{|x-y|}{\min(x,y)}$ for $x,y>0$ we have
\begin{align}
    \Big| \log \frac{\Psi_{\mu,\upsilon^2}(\theta)}{\Psi_{\tilde{\mu},\upsilon^2}(\theta)} \Big| \leq \frac{|\Psi_{\mu,\upsilon^2}(\theta) - \Psi_{\tilde{\mu},\upsilon^2}(\theta)|}{\min\{\Psi_{\mu,\upsilon^2}(\theta), \Psi_{\tilde{\mu},\upsilon^2}(\theta) \}}. \nonumber 
\end{align}
For $\theta\in [-K,K]$ the denominator can be bounded below as 
\begin{align}
    \min\{ \Psi_{\mu,\upsilon^2}(\theta), \Psi_{\tilde{\mu},\upsilon^2} (\theta) \} \geq \frac{1}{\sqrt{2\pi \upsilon^2} }  \exp\Big(- \frac{(K+1)^2}{2 \upsilon^2} \Big)\ge \frac{1}{\sqrt{2\pi \upsilon^2}}{\exp\Big(-\frac{2K^2}{\upsilon^2}\Big)}, \nonumber 
\end{align}
where the last inequality holds for $K\ge 1$.
Thus we have, 
\begin{align}
\label{eq:bound_interval}
     \sup_{\theta \in [-K,K]} \Big| \log \frac{\Psi_{\mu,\upsilon^2}(\theta)}{\Psi_{\tilde{\mu},\upsilon^2}(\theta)} \Big| \leq \sqrt{2 \pi \upsilon^2} \exp\Big( \frac{2K^2}{ \upsilon^2} \Big) \delta. 
\end{align}
To control the second term in the RHS of \eqref{eq:bound_sup_norm}, for any $K>0$  and $|\theta|>K$ we have,
\begin{align}
    \Big| \log \frac{\Psi_{\mu,\upsilon^2}(\theta)}{\Psi_{\tilde{\mu}, \upsilon^2} (\theta)} \Big|= &\Big|\log \frac{\int e^{ \frac{\theta \beta}{\upsilon^2} - \frac{\beta^2}{2 \upsilon^2}} d\mu (\beta)}{\int e^{ \frac{\theta \beta}{\upsilon^2} - \frac{\beta^2}{2 \upsilon^2}} d\tilde{\mu}(\beta)} \Big| \nonumber \\
    \leq &\Big| \log \int e^{ \frac{\theta \beta}{\upsilon^2} - \frac{\beta^2}{2 \upsilon^2}} d\mu (\beta)\Big| + \Big| \log \int e^{ \frac{\theta \beta}{\upsilon^2} - \frac{\beta^2}{2 \upsilon^2}} d\tilde{\mu} (\beta)\Big|. \nonumber  
\end{align}
To control the RHS above, note that for any probability measure $\mu \in \mathcal{P}$, 
\begin{align}
  - \frac{|\theta|}{\upsilon^2} - \frac{1}{2\upsilon^2} \leq  \log \int e^{\frac{\theta\beta}{\upsilon^2} - \frac{\beta^2}{2\upsilon^2}} d\mu(\beta) \leq \frac{|\theta|}{\upsilon^2}. \nonumber 
\end{align}
In turn, this implies 
\begin{align}
    \Big| \log \frac{\Psi_{\mu,\upsilon^2}(\theta)}{\Psi_{\tilde{\mu},\upsilon^2}(\theta)} \Big| \leq 2 \frac{(2 |\theta| + 1)}{\upsilon^2}, \nonumber 
\end{align}
which gives
\begin{align}
\label{eq:bound_tail}
    \sup_{|\theta|>K} \frac{1}{1+\theta^2} \Big| \log \frac{\Psi_{\mu,\upsilon^2}(\theta)}{\Psi_{\tilde{\mu},\upsilon^2}(\theta)}\Big| &\leq \frac{2}{\upsilon^2}\sup_{|\theta|>K} \frac{2|\theta|+1}{1+\theta^2} \nonumber  \\
    &\leq  \frac{2}{\upsilon^2} \frac{2K+1}{1+K^2} \leq \frac{6}{\upsilon^2 K}
\end{align}
for $K>1$. 
Combining \eqref{eq:bound_sup_norm} with \eqref{eq:bound_interval} and \eqref{eq:bound_tail}, we have, 
\begin{align}\label{eq:delta}
    d_{\upsilon^2}(\mu, \tilde{\mu}) \leq \sqrt{2\pi \upsilon^2} \exp\Big( \frac{2K^2}{ \upsilon^2} \Big) \delta + \frac{6}{\upsilon^2 K}. 
\end{align}
If we choose $K= \frac{12}{ \upsilon^2 \varepsilon}$, then 
$$\frac{\varepsilon}{2}+ \sqrt{2\pi \upsilon^2} \exp\Big(\frac{2K^2}{\upsilon^2}\Big)=\frac{\varepsilon}{2} + \sqrt{2\pi \upsilon^2}\exp\Big(\frac{288}{\upsilon^4\varepsilon^2}\Big)\le \exp \Big(\frac{c}{\varepsilon^2}\Big),$$
for some $c$ depending only on $\upsilon^2$ (and hence $C_2$). Consequently, with $\delta=\exp\Big(-\frac{c}{\varepsilon^2}\Big)$, the RHS of \eqref{eq:delta} is less than $\varepsilon$, and so $\mathcal{N}$ is an $\varepsilon$-net of $\mathcal{P}$ under $d_{\upsilon^2}$. This completes the proof. 
\end{proof}

\begin{proof}[Proof of Lemma~\ref{lemma:chain_rule}]
    Recall from \eqref{eq:modified_sequence} that $\widetilde{m}_{\mu}(\cdot)$ represents the density of $\mathbf{z} = \boldsymbol{\beta} + \Sigma^{1/2} {\boldsymbol{\varepsilon}}$ when the coordinates of $\boldsymbol{\beta}$ are sampled i.i.d. from $\mu$. Using \eqref{eq:eigen2} we have $\Sigma \preceq \frac{1}{C_1} \mathbf{I}$. Let $\tilde{\boldsymbol{\varepsilon}} \sim N(0, \frac{1}{C_1}\mathbf{I} - \Sigma)$ independent of $\boldsymbol{\beta}$ and $\boldsymbol{\varepsilon}$. For $a \in \{1,2\}$, let  $\mathbf{z}^{(a)} = \boldsymbol{\beta}^{(a)} + \Sigma^{1/2} \boldsymbol{\varepsilon}$, where the coordinates of $\boldsymbol{\beta}^{(a)}$ are sampled i.i.d. from $\mu_a$. By the  data-processing inequality for KL divergence,   
    \begin{align}
        \mathrm{D}_{\mathrm{KL}}(\widetilde{m}_{\mu_1} \| \widetilde{m}_{\mu_2}) &= \mathrm{D}_{\mathrm{KL}} (\boldsymbol{\beta}^{(1)} + \Sigma^{1/2} \boldsymbol{\varepsilon} \| \boldsymbol{\beta}^{(2)} + \Sigma^{1/2} \boldsymbol{\varepsilon}) \nonumber \\
        &\geq \mathrm{D}_{\mathrm{KL}} (\boldsymbol{\beta}^{(1)} + \Sigma^{1/2} \boldsymbol{\varepsilon}  + \widetilde{\boldsymbol{\varepsilon}} \| \boldsymbol{\beta}^{(2)} + \Sigma^{1/2} \boldsymbol{\varepsilon}  + \widetilde{\boldsymbol{\varepsilon}}) \nonumber \\
        &= p \,\, \mathrm{D}_{\mathrm{KL}}(\mu_1 \star N(0, C_1^{-1}) \| \mu_2 \star N(0, C_1^{-1})), \nonumber 
    \end{align}
    where the last step follows by the tensorization property of KL divergence and independence across co-ordinates of $ \Sigma^{1/2} \boldsymbol{\varepsilon}  + \widetilde{\boldsymbol{\varepsilon}}\sim N\Big({\bf 0}, \frac{1}{C_1}{\bf I}\Big)$. This completes the proof. 
\end{proof}

\begin{proof}[Proof of Example~\ref{eg:eigen}]
Let the true prior be $\mu^*=\delta_c$ --- the prior that is degenerate at $c$, for some $c\in \R$.  Then the marginal density $m_{\mu^*}({\bf y}) \propto e^{-\frac{1}{2\sigma^2} \|\bf y\|^2}$ does not depend on $c$, as ${\bf e}_i$ is a contrast vector for all $i=1,\ldots,p-1$. Thus there can be no consistent estimator for $\mu^*$. Note that for any probability distribution $\mu$ on $\R$, the above argument can be generalized to show that in this case we cannot have a consistent estimator of $\mu$, as the distribution $\mu$ and any location shift of $\mu$ cannot be distinguished from the marginal likelihood $m_{\mu^*}({\bf y})$.

For computing the eigenvalues of  ${\bf X}^\top {\bf X}$, first note that ${\bf X}^\top {\bf X}$ is a $p\times p$ matrix with rank $p-1$, and so it must have at least one zero eigenvalue. Also, the non-zero eigenvalues of ${\bf X}^\top {\bf X}$ are the same as the non-zero eigenvalues of ${\bf X} {\bf X}^\top = I_{p-1}$. The desired conclusion about the eigenvalues now follow.
\end{proof}

\section{Proof of results from Section \ref{sec:thm_inf}}\label{sec:supporting}

In this section we will verify Lemma \ref{lem:w}, Lemma \ref{lem:unif_cont}, Lemma \ref{lem:dev_bounds}, and Lemma \ref{lem:approximations}.

\begin{proof}[Proof of Lemma \ref{lem:w}]
    This follows on noting that $\mathbf{w} \mid \bbeta\sim N\Big(\frac{1}{\sigma^2}\mathbf{X}^\top \mathbf{X}\bbeta,\frac{1}{\sigma^2}\mathbf{X}^\top \mathbf{X}\Big)$, and so
    $$\E\Big[\|\mathbf{w}\|_2^2 \mid \bbeta\Big]=\Big\|\frac{1}{\sigma^2}\mathbf{X}^\top \mathbf{X}\bbeta\Big\|_2^2+{\rm tr}\Big(\frac{1}{\sigma^2}\mathbf{X}^\top \mathbf{X}\Big)\le (C_2^2+C_2)p,$$
    where we use \eqref{eq:eigen} in the last step and that $\beta_i \in [-1,1]$.
\end{proof}

\begin{proof}[Proof of Lemma \ref{lem:unif_cont}]
Let $\{(\mu^{(k)},\tau_k,d_k)\}_{k\ge 1}$ be a sequence in $\mathcal{P}\times \R\times (0,\infty)$ converging to $(\mu^{(\infty)},\tau_\infty,d_\infty)$ in $\mathcal{P}\times \R\times (0,\infty)$.
\begin{enumerate}
    \item[(i)]
It suffices to show that
$$\lim_{k\to\infty}c_{\mu^{(k)}}(\tau_k,d_k)=c_{\mu^{(\infty)}}(\tau_\infty,d_\infty).$$ To this effect, there exists $\theta_k\sim \mu^{(k)}$ for $1\le k\le \infty$ such that $\theta_k\stackrel{a.s.}{\to}\theta_\infty$. Consequently, dominated convergence theorem ($|\theta_k| \le 1$) gives
\begin{align*}
    e^{c_{\mu^{(k)}}(\tau_k,d_k)}=\E e^{\tau_k\theta_k-\frac{d_k}{2}\theta_k^2}\to \E e^{\tau_\infty\theta_\infty-\frac{d_\infty}{2}\theta_\infty^2}=e^{c_{\mu^{(\infty)}}(\tau_\infty,d_\infty)}.
\end{align*}
The desired conclusion follows from this on taking $\log$.

\item[(ii)]
Let $\theta_k\sim \mu^{(k)}_{\tau_k,d_k}$ for $1\le k\le \infty$. Then for any $\alpha\in \R$ we have
$$\log\E e^{\alpha \theta_k}=c_{\mu^{(k)}}(\alpha+\tau_k,d_k)-c_{\mu^{(k)}}(\tau_k,d_k)\to c_{\mu^{(\infty)}}(\alpha+\tau_\infty,d_\infty)-c_{\mu^{(\infty)}}(\tau_\infty,d_\infty)=\log\E e^{\alpha \theta_\infty},$$
where the convergence above uses part (i). The desired weak convergence follows from the convergence of the moment generating functions.

\item[(iii)]
Let $\theta_k\sim \mu^{(k)}_{\tau_k,d_k}$ for $1\le k\le \infty$, as in part (ii) above. Then we have
$$\dot{c}_{\mu^{(k)}}(\tau_k,d_k)=\E[\theta_k]\to \E[\theta_\infty]=\dot{c}_{\mu^{(\infty)}}(\tau_\infty,d_\infty),$$
where the convergence uses part (ii), along with the dominated convergence theorem.

\item[(iv)]
Let $\{(v_k,d_k)\}_{k\ge 1}$ be a sequence in $I_{\mu}\times (0,\infty)$ converging to $(v_\infty,d_\infty)\in I_\mu\times(0,\infty)$. Setting $\tau_k:=h_{\mu,d_k}(v_k)$ for $1\le k\le \infty$, we need to show that $\tau_k\to\tau_\infty.$ To this effect, set $a:=\inf{\rm supp}(\mu), b:= \sup{\rm supp}(\mu)$, and fix $\varepsilon>0$ such that $v_\infty\in (a+\varepsilon,b-\varepsilon)$. Then there exists $K=K_\varepsilon$ such that for all $k\ge K$ we have $v_k\in (a+\varepsilon, b-\varepsilon)$. This, along with the assumption $\lim_{k\to\infty}d_k=d_\infty$ forces $\limsup_{k\to\infty}|\tau_k|<\infty$ (as otherwise one would contradict that $v_ \infty \in (a+\varepsilon, b-\varepsilon)$).
Setting $M:=\sup_{k\ge 1}|\tau_k|<\infty$ we then have
$$|\dot{c}_\mu(\tau_k,d_\infty)-\dot{c}_\mu(\tau_k,d_k)|\le \sup_{|\tau|\le M}|\dot{c}_\mu(\tau,d_\infty)-\dot{c}_\mu(\tau,d_k)|=o(1),$$
where the last equality uses part (iii). Thus we have
\begin{align*}
\dot{c}_\mu(\tau_k,d_\infty)=\dot{c}_\mu(\tau_k,d_k)+o(1)=v_k+o(1)=v_\infty+o(1)=\dot{c}_\mu(\tau_\infty,d_\infty)+o(1).
\end{align*}
Applying the continuous function $h_{\mu,d_\infty}(\cdot)$ then gives
$\tau_k\to\tau_\infty$, as desired.
\end{enumerate}
\end{proof}

\begin{proof}[Proof of Lemma \ref{lem:dev_bounds}]

   Since $\mu^*$ is non-degenerate and $\mu^{(p)}$ converges weakly to $\mu^*$, it follows that $\mu^{(p)}$ is also non-degenerate for all $p$ large. In this case, for any $\mathbf{r}\in I_{\mu^{(p)}}^p$ we have 
    \begin{align}
        \nabla \M_p(\mathbf{r},\mathbf{w}, \mu^{(p)}) = - A\mathbf{r} + \mathbf{w}- h_{\mu^{(p)},\mathbf{d}}(\mathbf{r}). \nonumber 
     \end{align}
    If $r_i \to \sup \mathrm{Supp}(\mu^{(p)})$ for some $i\in [p]$, $h_{\mu^{(p)},d_i}(r_i) \to \infty$. Similarly, if $r_i \to \inf \mathrm{Supp}(\mu^{(p)})$, $h_{\mu^{(p)},d_i}(r_i) \to - \infty$. This implies that a global optimizer must belong to $I_{\mu^{(p)}}^p$, where $I_\mu$ is as defined in Lemma \ref{lem:unif_cont} part (iv). Consequently, the optimizer $\mathbf{v}^{(p)}$ satisfies
 $\nabla \M_p(\mathbf{v}^{(p)}, \mathbf{w}, \mu^{(p)})=0$, which in turn gives 
    \begin{align}
        \sum_{i =1}^p (\gamma_i^{(p)})^2 = \sum_{i =1}^p (-A{\mathbf{v}}^{(p)} + \mathbf{w})_i^2\le 2\|\mathbf{w}\|_2^2+2p\lambda_{\max}(A)^2. \nonumber 
     \end{align}
     Using Lemma \ref{lem:w} and \eqref{eq:eigen}, the right hand side above is $O_{\P_{\mu^*}}(p)$, from which the desired conclusion follows by Markov's inequality. 
\end{proof}

\begin{proof}[Proof of Lemma \ref{lem:approximations}]

The first conclusion follows on noting that $|\{i\in [p]:|\tau_i^{(p)}|>K\}|=\er(p)$, which follows from Lemma \ref{lem:dev_bounds}. The second conclusion follows from noting that 
$$\lim_{p \to \infty} 
    \sup_{|\theta| \leq K, \frac{1}{K} \leq d \leq K} | \dot c_{\mu^{(p)}}(\theta,d) - \dot c_{\mu^*}(\theta,d) | =0. $$
    To verify the above display, note that the function  $(\mu,\theta,d)\mapsto \dot{c}_\mu(\theta,d)$ is continuous on the compact set $\mathcal{P}\times[-K,K]\times [\frac{1}{K},K]$ (invoking Lemma \ref{lem:unif_cont} part (iii)), and hence uniformly continuous. 
\end{proof}

\section{Optimization algorithm for approximating the NMF-EB estimator}\label{sec:Opt-Algo}
In this section we provide the details of the algorithm introduced in Section~\ref{sec:numerical}. In particular, we give the exact expressions of the gradient of the objective (w.r.t.~the parameters) that can be directly used in a first order optimization algorithm to maximize $\widetilde{M}_p({\bm \gamma},\mathbf{w},\nu({\bf p}))$.
Differentiating with respect to $\gamma_i$, for $i\in [p]$, gives (using $u(\cdot)$ as in~\eqref{eq:u_i})
\begin{align}
\frac{\partial \widetilde{M}_p({\bm \gamma},\mathbf{w},\nu({\bf p}))}{\partial \gamma_i}=\Big[-\sum_{j=1}^p A_{ij}u_j^{(\nu({\bf p}))}(\gamma_j)+w_i-\gamma_i\Big]\frac{\partial u_i^{(\nu({\bf p}))}}{\partial \gamma_i} (\gamma_i). \label{eq:M_tilde_derivative}
\end{align}
On the other hand, differentiating with respect to $p_r$, for $1\le r\le k$,  gives
\begin{align*}
 \frac{\partial c_{\nu({\bf p})}\Big(\gamma_i, d_i\Big)}{\partial p_r} = &  \; \exp\Big(a_r \gamma_i-\frac{a_r^2 d_i}{2}-c_{\nu({\bf p})}\Big(\gamma_i, d_i\Big)\Big) \\
\frac{\partial \widetilde{M}_p({\bf \gamma},\mathbf{w},\nu({\bf p}))}{\partial p_r}=&-\sum_{i,j=1}^p \Big[\frac{\partial u^{(\nu({\bf p}))}_i(\gamma_i)}{\partial p_r} A_{ij}u_j^{(\nu({\bf p}))}(\gamma_j)\Big]+\sum_{i=1}^p w_i\frac{\partial u^{(\nu({\bf p}))}_i(\gamma_i)}{\partial p_r}-\sum_{i=1}^p \gamma_i\frac{\partial u^{(\nu({\bf p}))}_i(\gamma_i)}{\partial p_r}\\
&-\sum_{i=1}^p \exp\Big(a_r \gamma_i-\frac{a_r^2 d_i}{2}-c_{\nu({\bf p})}\Big(\gamma_i, d_i\Big)\Big).
\end{align*}
In the above derivatives, the function ${\bf u}(\cdot)$ and its derivatives can be computed using the following formulas:
\begin{align*}
c_{\nu({\bf p})}\Big(\gamma_i, d_i\Big)=&\log \left[\sum_{r=1}^k  p_r\exp\Big(a_r \gamma_i-\frac{a_r^2 d_i}{2}\Big)\right],\\
u^{(\nu({\bf p}))}_i(\gamma_i)=&\sum_{r=1}^k  p_ra_r \exp\left(a_r \gamma_i-\frac{a_r^2 d_i}{2}-c_{\nu({\bf p})}\Big(\gamma_i,d_i\Big)\right),\\
\frac{\partial u_i^{(\nu({\bf p}))}}{\partial \gamma_i}(\gamma_i)=&\frac{1}{2}\sum_{r,s=1}^k p_r p_s(a_r-a_s)^2 \exp\left(a_r \gamma_i-\frac{a_r^2 d_i}{2}-c_{\nu({\bf p})}\Big(\gamma_i, d_i\Big)\right) \exp\left(a_s \gamma_i-\frac{a_s^2 d_i}{2}-c_{\nu({\bf p})}\Big(\gamma_i, d_i\Big)\right),\\
\frac{\partial u_i^{(\nu({\bf p}))}}{\partial p_r}(\gamma_i)=&(a_r-u_i^{(\nu({\bf p}))}) \exp\Big(a_r\gamma_i-\frac{a_r^2 d_i}{2}-c_{\nu({\bf p})}\Big(\gamma_i, d_i\Big)\Big).
\end{align*}

\noindent
\textbf{Initialization for $\bgamma$:} Note that, from \eqref{eq:M_tilde_derivative}, any stationary point of $\widetilde{M}_p$ satisfies $\gamma_i = - \sum_{j=1}^{p} A_{ij} u_j^{(\nu(\mathbf{p}))} (\gamma_j) + w_i$. We start with a suitable estimate $\hat{\bbeta} = (\hat \beta_1,\ldots, \hat \beta_p)$ (either the OLS estimate, or Lasso with a small regularization parameter), and set $\gamma_i = - \sum_{j=1}^{p} A_{ij} \hat{\beta}_j + w_i$.

\bibliographystyle{imsart-nameyear}
\bibliography{References}
\end{document}